\documentclass[11pt,twoside]{amsart}

  \usepackage{url}
  \usepackage{amssymb,amsmath}
     \usepackage{graphicx}

   \usepackage[a4paper, margin=2.5cm]{geometry}

  \input xypic
  \xyoption{all}

 \usepackage{times}

  \usepackage{amsmath, amsthm, amsfonts}

\newtheorem{theorem}{Theorem}[section]
\newtheorem{lemma}[theorem]{Lemma}
\newtheorem{proposition}[theorem]{Proposition}
\newtheorem{corollary}[theorem]{Corollary}

\newtheorem{assumption}[theorem]{Assumption}

\theoremstyle{definition}
\newtheorem{definition}[theorem]{Definition}
\newtheorem{example}[theorem]{Example}

\theoremstyle{remark}
\newtheorem{remark}[theorem]{Remark}

\numberwithin{equation}{section}

  \newcommand{\cA}{{\mathcal A }}
  \newcommand{\cB}{{\mathcal B}}
  \newcommand{\cI}{{\mathcal I }}
  
  \newcommand{\cK}{{\mathcal K}}
  
  \newcommand{\cF}{{\mathcal F}}
  \newcommand{\cE}{{\mathcal E}}
  \newcommand{\cM}{{\mathcal M}}
   \newcommand{\cN}{{\mathcal N}}

  \newcommand{\cU}{{\mathcal U}}
   \newcommand{\cV}{{\mathcal V}}
    \newcommand{\cW}{{\mathcal W}}
  \newcommand{\cG}{{\mathcal G }}
  \newcommand{\cC}{{\mathcal C }}
\newcommand{\cO}{{\mathcal O }}

\newcommand{\cT}{{\mathcal T }}
\newcommand{\cX}{{\mathcal X }}
\newcommand{\cS}{{\mathcal S }}
\newcommand{\cY}{{\mathcal Y }}

  \newcommand{\bU}{{\mathbf U }}
   \newcommand{\bE}{{\mathbf E }}
\newcommand{\bo}{{\mathbf o }}

   \newcommand{\ba}{\begin{eqnarray}}
   \newcommand{\na}{\end{eqnarray}}

\newcommand{\supp}{{\text  supp}}

  \newcommand{\C}{{\mathbb C}}
  \newcommand{\R}{{\mathbb R}}
  \newcommand{\Z}{{\mathbb Z}}

  \newcommand{\ind}{{\bf  Index }}

  \renewcommand{\a}{\alpha}
  \renewcommand{\b}{\beta}

  \newcommand{\eps}{\epsilon}

\def \wt{\widetilde}
\def \wtB{\widetilde{\mathcal B}}
\def \wtE{\widetilde{\mathcal E}}
\def \wtM{\widetilde{\mathcal M}}
\def \dbar{\bar{\partial}}
\def \xto{\xrightarrow}

    \newcommand{\disp}{\displaystyle}
  \newcommand{\nin}{\noindent}

\begin{document}

\title{Virtual neighborhood technique  for  pseudo-holomorphic spheres}

  \author{Bohui Chen, An-Min Li}
  \address{School of mathematics\\
  Sichuan University\\
 Chengdu, China}
  \email{bohui@cs.wisc.edu \\
  math$\_$li@yahoo.com.cn}

  \author{Bai-Ling Wang}
  \address{Department of Mathematics\\
  Australian National University\\
  Canberra ACT 0200 \\
  Australia}
  \email{bai-ling.wang@anu.edu.au}

\date{}



   \begin{abstract} This is the first  part of three series papers in which    we apply the  theory of virtual manifold/orbifolds developed by the first named author and Tian to study the Gromov-Witten moduli spaces.  In this  paper, we resolve the main  analytic issue arising from the lack of differentiability of the $PSL(2, \C)$-action on spaces of $W^{1, p}$-maps from  the Riemann sphere to a symplectic manifold  $(X, \omega, J)$ with a non-zero homology class $A$.  In particular, we establish the slice  and tubular neighbourhood theorems for the $PSL(2, \C)$-action  along smooth maps.  In  Sections 2 and 3 of this paper, we explain  an integration  theory on   virtual orbifolds using proper \'etale groupoids and establish the virtual neighborhood technique for a general orbifold  Fredholm system.  When
 the moduli space $\cM_{0, 0}(X, A)$ of pseudo-holomorphic spheres in $(X, \omega, J)$  is compact, applying  the   virtual neighborhood technique developed  in Section 3,  we obtain   a virtual system for the moduli space $\cM_{0, 0}(X, A)$ of pseudo-holomorphic spheres in $(X, \omega, J)$  and show that  the genus zero Gromov-Witten invariant is well-defined.
  \end{abstract}

\maketitle

  \tableofcontents


\section{Introduction and statements of main theorems}

Pseudo-holomorphic curves in a symplectic manifold $(X, \omega)$ with a compatible almost complex structure
$J$  were  first  introduced by Gromov in  his seminal paper \cite{Gro}. This has been followed by    deep results in   symplectic topology.  The  Gromov-Witten invariants   ``{\em count }"  stable pseudo-holomorphic  curves  of genus $g$ and   $n$-marked points in a symplectic manifold $(X, \omega, J)$.  The Gromov-Witten invariants   for semi-positive symplectic manifolds were  defined   by Ruan in  \cite{Ruan93} and Ruan-Tian in \cite{RT95} and  \cite{RT96}.      These invariants can  be applied to define a quantum  product on the cohomolgy groups of $X$ which leads  the notion of quantum cohomolog  in  \cite{RT95} for  semi-positive symplectic manifolds.   Since then the  Gromov-Witten  invariants  have  found many applications in symplectic geometry and symplectic topology, see the book  by   McDuff-Salamon \cite{McS} (and  references therein).

The main analytical  difficulty  in defining the   Gromov-Witten invariants for general  symplectic manifolds  is  the  failure of the transversality   of the  compactified moduli space of pseudo-holomorphic curves.   The  foundation to resolve this issue is to construct  a    virtual  fundamental cycle   for  the compactified moduli space.  For  smooth projective varieties, the construction of this  virtual  fundamental cycle was carried out   by Li-Tian in \cite{LiTian98} where they showed that  the Gromov-Witten invariants can be defined purely algebraically.   For  general symplectic manifolds, the virtual fundamental cycle was constructed by Fukaya-Ono in \cite{FO99}, Li-Tian \cite{LiTian96}, Liu-Tian \cite{LiuTian98} and Siebert \cite{Sie}.   Ruan  in \cite{Ruan} proposed  a virtual neighbourhood technique as a dual approach using the Euler class of a virtual neighbourhood,  in which   the  compactified moduli space is  treated as a zero set of a smooth section of  a finite dimensional orbfiold vector bundles over an open orbifold.

 Further developments in Gromov-Witten theory and its applications require more refined  structures  on  the  moduli spaces  involved. Some of the analytical  details have been  provided by Ruan \cite{Ruan}, Li-Ruan \cite{LiRuan} and  Fukaya-Oh-Ohta-Ono
 \cite{FOOO}.  Other methods like the polyfold theory  by Hofer-Wysocki-Zehnder  \cite{HWZ} are developed  to deal with this issue.
Recently, there are some renew interests on a variety of technical
issues in Gromov-Witten theory by McDuff-Wehrheim \cite{McW} and Fukaya-Oh-Ohta-Ono \cite{FOOO12}.  These  technical  issues include the differentiable structures on  Kuranish models of
the Gromov-Witten moduli spaces and  the non-differentiable  issue  of the action of automorphism groups in the infinite dimensional
non-linear Fredholm framework.

Our aim   is motivated by how to define the K-theoretical Gromov-Witten invariants for  general symplectic manifolds. This amounts to the full machinery of the  virtual neighborhood technique   to study  the Gromov-Witten moduli spaces  using the virtual manifold/orbifolds developed  in \cite{ChenTian}.  The language of virtual orbifolds  provides an alternative and simpler approach to establish the required  differentiable structure on  moduli spaces arising from the Gromov-Witten theory.
   We remark that the theory of  virtual manifold/orbifolds and the general framework   of the  virtual neighborhood technique   were established by the first author and Tian in \cite{ChenTian}.

The  virtual neighborhood method avoids  some of subtle and difficult issues  when  further perturbations  are needed.  We use the  following example to demonstrate the idea that  further perturbations are not needed to define Euler  invariants.
Let $E$ be a real oriented vector bundle over a compact smooth manifold $U$, and   let $\Theta$
be a Thom form of $E$ then  for $\omega \in \Omega^*(U)$
\ba\label{ex:Thom}
\int_U s^*\Theta \wedge \omega   = \int_{s^{-1}(0)} \iota^* \omega
\na
for any section $s$ of $E$ which is transversal to the zero section.  Here $\iota: s^{-1}(0) \to U$ is the inclusion map. This formula says that for those
integrands  from the ambient space $U$, there is no need to perturb the section to achieve the transversality in order to calculate the right hand side of (\ref{ex:Thom}), as we know that the final result is given by  the left  hand side of (\ref{ex:Thom})
\ba\label{Thom:2}
\int_U s^*\Theta \wedge \omega
\na
for any  (not necessarily transversal) section  $s$.  Similar  results hold  for  a compact orbifold $U$ with an oriented real orbifold
 vector bundle.  Note that  $U$ can be replaced by  a non-compact  orbifold, then
 the integrand $s^*\Theta \wedge \omega $ needs to be compactly supported.

  In the original proposal of Ruan in \cite{Ruan}, it was proposed that
the compactified moduli space $\cM$ of stable maps
 in a closed symplectic manifold $X$ can be realised
as $ s ^{-1}(0) $ for a    section  $s$  of an orbifold vector
bundle $E$  over a finite dimensional  $C^1$-orbifold $U$. The
triple $(U, E, s)$ is called  a virtual neighbourhood of $\cM$ in
\cite{Ruan}. In practice,  it turns out that such an elegant
triple  is too idealistic to obtain  an orbifold
   Fredholm system for general cases.  Here  an {\bf  orbifold
   Fredholm system }   consists of a triple
   \[
   (\cB, \cE, S),
   \]
 where   $\cE$ is a Banach orbifold bundle over a Banach orbifold  $\cB$ together with a Fredholm section $S$. There  does not exist a single virtual neighbourhood $(U, E, s)$ even  for an orbifold
   Fredholm system. We remark that for  certain 
 orbifold    Fredholm systems, the virtual Euler cycles and their properties have been  developed systematically in an   abstract setting by Lu and Tian in \cite{LuTian}.

 To remedy this,  the theory of virtual manifolds and virtual orbifolds   were  developed  by the first author and Tian,  see \cite{ChenTian}.
 We will give a self-contained review of the theory of  virtual manifolds and virtual orbifolds in Section \ref{2} and further establish the
 integration theory on virtual orbifolds using the language of proper \'etale groupoids.  The language of groupiods is very useful when   an orbifold  atlas  can not be   clearly  described, in particular, in the case of stable maps where the patching data (arrows in the language of groupoids) are induced from biholomorphisms  appearing in  the universal families of curves over the
 Teichm\"uller spaces.

 It was proposed in \cite{ChenTian} that the single  virtual neighbourhood  proposed in \cite{Ruan} should be replaced by a system of virtual neighbourhoods
 \[
 \{(\cV_I, \bE_I, \sigma_I)  \}
 \]
 indexed by a partial ordered index set $\cI =\{I\subset \{1, 2, \cdots, n\}\}$. Here $\{\cV_I\}$ is a virtual orbifold, and $\{\bE_I\}$ is a finite rank virtual orbifold bundle with a  virtual section $\{\sigma_I\}$ such that the zero sets $\{\sigma_I^{-1} (0)\}$ form a cover of  the underlying moduli space $\cM$.  If  the sections $\{\sigma_I\}_I$ were   all 
 transversal, then the  moduli space $\cM$ would be a smooth orbifold. In the general,   this    system of  virtual neighbourhoods will be called a {\bf virtual system} (Cf. Definitions \ref{virtual-system} and  \ref{virtual-sys:orb}).
 The invariants associated to the moduli space can be obtained by applying the integration theory on  $\{\cV_I\}$  to certain  virtual differential forms.
In particular, the collection of Euler forms for $\{\bE_I\}$, called a {\bf virtual Euler form},  is  a   virtual differential form on $\{\cV_I\}$.
In essence, this  virtual Euler form is the dual version of the so-called virtual fundamental class.

 To demonstrate this idea,  global stabilizations({\bf cf. \S3.3.2}) were  introduced  in \cite{ChenTian} to get a virtual system
 from a single   Fredholm system $(\cB, \cE, S)$.  These will be reviewed and further developed to orbifold cases  in terms of  proper \'etale groupoids in
 Section \ref{3}. We also explain how the integration theory for virtual orbifolds (proper \'etale  virtual groupoids)  can be applied to get a well-defined invariant from  the  virtual system.   This process of going  from a  Fredholm system $(\cB, \cE, S)$ to a well-defined invariant will be called the {\bf virtual neighborhood technique}. The main result in
 Section  \ref{3} is to establish  the virtual neighborhood technique for any   orbifold Fredholm system.

\vspace{2mm}

\nin{\bf Theorem A}    (Theorem \ref{thm:orbi-vir-sys})  {\em  Given  an orbifold Fredholm system  $(\cB, \cE, S)$
such that $\cM = S^{-1}(0)$ is compact, then
there exists  a finite dimensional virtual  orbifold system  for $(\cB, \cE, S)$ which is   a collection of triples
 \[
  \{(\cV_I, \bE_I, \sigma_I)|  I\subset \{1, 2, \cdots, n\} \}
  \]
  indexed by a  partially ordered
set $(\cI =2^{\{1, 2, \cdots, n\}}, \subset )$, where
  \begin{enumerate}
\item $\cV=\{\cV_I  \}$ is a finite dimensional  proper  \'etale virtual groupoid,
\item $\bE=\{\bE_I\}$ is a  finite rank virtual  orbifold vector bundle over $\{\cV_I\}$
\item  $\sigma= \{\sigma_I\}$ is a virtual  section  of   $\{\bE_I\}$
whose zero sets $\{ \sigma_I^{-1} (0)\}$ form a  cover of $\cM$.
  \end{enumerate}
 Moreover, under   Assumptions  \ref{assump:orbi} and    \ref{assump:orbi-2},  there is  a choice of a partition of unity  $\eta=\{\eta_I\}$
on $\cV$ and  a  virtual Euler  form $\theta=\{\theta_I\}$  of $ \bE $ such that  each
$\eta_I\theta_I$ is compactly supported in $\cV_I$. Therefore, the virtual integration
\[
\int^{vir}_\cV \alpha   = \sum_I \int_{\cV_I}\eta_I\theta_I\alpha_I
\]
 is well-defined for any virtual differential form $\alpha =\{\alpha_I\}$ on $\cV$ when    both
 $ \cV $ and $ \bE $ are  oriented.

 }

 For the full version of the Gromov-Witten theory,  there are some further  technical issues to
 implement the  virtual neighborhood technique  for the moduli spaces of stable maps, mainly
 \begin{enumerate}
\item[(1)]  the non-differentiable action of non-discrete  automorphism groups in general,  for example,  the  failure of differentiability for the $PSL(2, \C)$-action in the genus zero case, here  $PSL(2, \C)$ is the automorphism group of the standard Riemann sphere.
 \item[(2)] the issue of smoothness in the compactified space  of stable maps where there are lower strata  present.
\end{enumerate}

 In  the second part of this paper,   we   resolve the analytical  issue of the non-differentiable $PSL(2, \C)$-action    and
 establish a  virtual system  for   this genus zero moduli space. We now outline  the main  ideas on how to resolving this issue and
 summarise the main results as Theorems B-D.

 Let $\Sigma=(S^2,j_o)$ be the standard Riemann  sphere with the round metric and  $(X,\omega,J)$
be a compact symplectic manifold with a compatible almost complex structure.  Let $\wtB$ be the  infinite dimensional Banach manifold  of $W^{1,p}$-maps from $\Sigma$ to $X$ with a fixed homology class $A\in H_2(X, \Z)$, that is,
\[
\wtB =\{ u\in W^{1, p}(\Sigma, X)| f_*([\Sigma]) =A\},
\]
where $p$ is an {\bf even }  integer greater than $2$.  Let $\wtM= \wtM_{0, 0}(X, A, J)$ be the solution space to the Cauchy-Riemann equation
\[
\dbar_J u=0
\]
for $u\in \wtB$. Note that  $\wtM$ would be empty if $\omega (A) < 0$.
 The automorphism group
$G=Aut(\Sigma  )$ acts on $\wtB$ by  reparametrization of  the domain $\Sigma$. Under the identification
$\Sigma = \C\cup\{\infty\}$,
\[
G \cong PSL(2, \C) =  GL(2, \C)/\C^\times Id
\]
acts on $  \C\cup\{\infty\}$ as M\"obius transformations. This action preserves  the solution space $\wtM$. The Gromov-Witten moduli space  of pseudo-holomorphic spheres, denoted by $\cM_{0, 0}(X, A, J)$, is the quotient space $\wtM$ by this $G$-action.
 In this paper, we assume that
\[
\cM= \cM_{0, 0}(X, A, J)
\]
is compact which is guaranteed  if  $A$ is indecomposable, see
Theorem 4.3.4 in \cite{McS}.

One of the main analytical issues  in the study of the Gromov-Witten moduli space $\cM_{0, 0}(X, A, J)$ is that the action
\[
G \times \wtB  \longrightarrow \wtB
\]
is not  differentiable (as pointed out by Ruan  in \cite{Ruan}). This non-differentiability issue was extensiveley  discussed in
\cite[Remark 3.15]{McW}.  We can interpret $\wtM$ as a zero set of an orbifold Fredholm system $(\wtB,  \wtE, S)$
where   $\wtE$ is the infinite dimensional Banach bundle over $\wtB$ whose fiber at $u$ is  the space of
$L^p$-sections  of the bundle $ \Lambda^{0,1} T^*\Sigma\otimes_\C u^*TX$ and the section $S=\dbar_J$ is defined by the Cauchy-Riemann operator.
This system $(\wtB,  \wtE, S)$ is $G$-invariant in the sense that $\wtE$ is a topological $G$-equivariant Banach bundle and $S$ is $G$-invariant.

  Due to the non-smooth  action,  the quotient of $(\wtB,  \wtE, S)$ is only a topological Fredholm system $(\cB,  \cE, S)$.
Nevertheless,  the  virtual neighborhood technique  in Section \ref{3} can still be implemented  to this topologically
$G$-equivariant  Fredholm system. This is explained in   Section  \ref{5} which  is   based the analytical preparation in Section \ref{4}.  So we   obtain   a virtual orbifold system
with a partition of unity $\{\eta_I\}$
 and  a  virtual differential form   $\{\theta_I\}$   such that the virtual integration can be applied to get    genus zero Gromov-Witten invariants.  We remark that
 Assumption  \ref{assump:orbi}  is satisfied for any positive even integer $p>2$ and Assumption
   \ref{assump:orbi-2} is guaranteed by the elliptic regularity of the Cauchy-Riemann equations.

In Section \ref{4},   we establish a slice theorem for  the $G$-action at each smooth map   and study the tubular neighbourhood structure along its  orbit.
The following     slice and neighbourhood theorem    plays a
central role in the development of a virtual orbifold system for the moduli space $ \cM_{0, 0}(X, A, J)$ in $\cB$.

\vspace{2mm}

\nin{\bf Theorem B}    (Theorem \ref{slice-theorem:orbifold}) {\em  Let $u_0: \Sigma \to X$ be a smooth  map in $\wtB$ with  a finite isotropy group $G_{u_0}$. Let $\pi_\cB$ be the   quotient map $\pi_\cB: \wtB \to \cB$.
\begin{enumerate}
\item The $G$-orbit $\cO_{u_0}$ through $u_0$ is a smooth sub-manifold of $\wtB$. The  normal bundle  $\cN_{\cO_{u_0}}$ is a smooth
Banach bundle and  can be identified with a topological  $G$-equivariant
smooth sub-bundle of $  T\wtB|_{\cO_{u_0}}$.
\item   There exists a sufficiently small $\eps_0$  such that the exponential map $\exp_{u_0}:  \cN^{\eps_0}_{\cO_{u_0}} \to \wtB$
is well-defined. Here $\cN^{\eps_0}_{\cO_{u_0}}$ is the $\eps_0$-ball bundle of $\cN_{\cO_{u_0}}$. Moreover,
the tubular neighbourhood    $\cT^{\eps_0}_{u_0} = \exp_{u_0} (\cN^{\eps_0}_{\cO_{u_0}})$   is a $G$-invariant open set in $\wtB$ containing
$\cO_{u_0}$.
\item There exists  a  $G_{u_0}$-invariant slice  $\cS^{\epsilon_0}_{u_0} \subset \wtB$ at $u_0$
such that the quotient space  $\cS^{\epsilon_0}_{u_0}/G_{u_0}$ is homeomorphic to
$ \pi_\cB(\cS^{\epsilon_0}_{u_0})$  for which we use the notation
$\bU_{u_0}^{\epsilon_0}$.
The  map
\[
\phi_{u_0}: \qquad  \cN^{\epsilon_0}_{u_0}:   =   \cN^{\epsilon_0}_{\cO_{u_0}}|_{u_0}  \xto{\ \ \exp_{u_0} \ \ } \cS^{\epsilon_0}_{u_0}  \xto{\pi_\cB}  \bU_{u_0}^{\epsilon_0}
\]
provides  a {\bf Banach  orbifold  chart} $( \cN^{\epsilon_0}_{u_0}, G_{u_0}, \phi_{u_0})$  for $\bU_{u_0}^{\epsilon_0}$.
\end{enumerate}
}

\vspace{2mm}

Due to the non-differentiable  action of $G$, these Banach orbifold charts do not provide a smooth orbifold structure on a neighbourhood
\[
\cW= \bigcup_{[u_0]\in \cM}  \bU_{u_0}^{\eps_0}
\]
of $\cM$ in $\cB$,
instead  they define a topological orbifold strucure  on $\cW$ with certain smoothness  in the following sense. Suppose that
$ \bU_{u_0}^{\eps_0} \subset  \bU_{u_1}^{\eps_1}$ for $u_0, u_1 \in \wtM$,
even though the  embedding of orbifold charts
\[
\phi_{u_0, u_1}:  ( \cN^{\epsilon_0}_{u_0}, G_{u_0}, \phi_{u_0}) \longrightarrow ( \cN^{\epsilon_1}_{u_1}, G_{u_1}, \phi_{u_1})
\]
and the induced embedding of their slices
\[
\Phi_{u_0, u_1} = \exp_{u_1} \circ  \phi_{u_0, u_1} \circ \exp_{u_0}^{-1}:
\cS_{u_0}^{\eps_0}  \longrightarrow  \cS_{u_1}^{\eps_1}
\]
are  only  topological embeddings (both defined by the $G$-action).
  We have
\[
\cT^{\eps_0}_{u_0}\subset \cT^{\eps_1}_{u_1}
\]
as a smooth inclusion for   their  tubular neighbourhoods as  prescribed by Theorem A.
We remark that when $f:  \cS_{u_1}^{\eps_1} \to \R$  has a $G$-invariant smooth  extension
$\hat f:  \cT^{\eps_1}_{u_1} \to \R$,  then  $f \circ \Phi_{u_0, u_1}$ is actually a smooth function on
$\cS_{u_0}^{\eps_0}$, as it is the restriction of
 $\hat f$  to  $\cS_{u_0}^{\eps_0}$.   This observation plays a  fundamental role  in overcoming the non-differential
 action of $PSL(2, \C)$ in our study of the moduli space of pseudo-holomorphic spheres.

Under the assumption that  the moduli space $\cM = \cM_{0, 0}(X, J, A)$ is  compact, there exist  finitely  many  points $\{[u_1], [u_2],  \cdots, [u_n]\}$ in $\cM$ with their smooth representatives
 \ba\label{fin:pts}
 u_1,  u_2,  \cdots, u_n
 \na
 in $\wtB$  whose   slices  $\{\cS_{u_i}^{\eps_i}\}$   and  tubular neighbourhoods
 $\{\cT_{u_i}^{\eps_i}\}$ prescribed by Theorem \ref{slice-theorem:orbifold}  define
  a system of Banach orbifold charts
    \ba\label{orbifold:system}
 \{(\cN_{u_i}^{\eps_i}, G_{u_i}, \phi_{u_i}, \bU^{\eps_i}_{u_i})|   i=1, \cdots, n  \}
 \na
  such that
 $
   \bigcup_{i=1}^n \bU ^{\eps_i/3}_{u_i} \supset \cM.
$
Let    $\cU =  \bigcup_{i=1}^n \bU ^{\eps_i}_{u_i}$, then the system of Banach orbifold charts is not compatible in the sense that the  orbifold coordinate changes are not smooth. Nevertheless, as we mentioned earlier,
  we propose  a notion of smooth functions  on $\cU$.  For an   even  positive integer $p>2$ we
construct  a system of  smooth cut-off  functions
\ba\label{cut-off:system}
\{\beta_i:  \bU_{u_i}^{\eps_i}  \to [0, 1]  | i= 1, \cdots, n \}
\na
  with respect to the above
system of Banach orbifold charts in Subsection \ref{4.2}, such that, each
\[
\tilde\beta_i = \beta\circ \pi_\cB: \cT_{u_i}^{\eps_i}  \to [0, 1]\]
satisfies the following properties:
\begin{enumerate}
\item[(i)] $\tilde \beta_i|_{ \cT_{u_i}^{\eps_i/3}}\equiv 1$,
\item[(ii)]  $\supp (\tilde \beta_i)\subset \cT_{u_i}^{2\eps_i/3}$,
 \item[(iii)]   $\tilde \beta_i|_{ \cS_{u_i}^{\eps_i}}$ is $G_{u_i}$-invariant.
  \item[(iv)]   $\tilde \beta_i|_{ \cT_{u_i}^{\eps_i}}$ is $G$-invariant.
  \end{enumerate}

In  Subsection \ref{4.3},   we  construct  a   system of local  obstruction bundles associated to   the system of Banach orbifold charts  in (\ref{orbifold:system}). As an application of Theorem B, we establish the existence of $G$-invariant local virtual neighbourhoods   and their slices
in Subsection \ref{4.4}  (See Propositions \ref{G-virtual:neigh} and \ref{virtual:neigh}).

\vspace{2mm}

\nin{\bf Theorem C}    (Propositions \ref{G-virtual:neigh} and \ref{virtual:neigh}) {\em
 For each point $u_i \in  \wtM$ from (\ref{fin:pts}),
there is  a  $G$-invariant virtual neighbourhood
  \[
  (\wt\cV_{u_i},  \wt E_{u_i},\wt\sigma_{u_i})
  \]
  of $\wtM \cap  \cT_{u_i}^{\eps_i}$  in the following sense.
  \begin{enumerate}
\item    $\wt\cV_{u_i}$ is  a  finite dimensional  smooth  $G$-manifold.
\item  $\wt E_{u_i}$ is a   $G$-equivariant vector  bundle.
\item  $\wt\sigma_{u_i} $  is   a canonical $G$-invaraint section  of $\wt E_{u_i}$ whose zero set is
$ \wtM \cap\cT_{u_i}^{\eps_i}$.
\end{enumerate}
Moreover, there is $G$-slice of $ (\wt\cV_{u_i},  \wt E_{u_i},\wt\sigma_{u_i})$
which provides a $G_{u_i}$-invariant virtual  neighbourhood
 \[
 (\cV_{u_i},    E_{u_i}, \sigma_{u_i})
 \]
  for $\wtM\cap \cS_{u_i}^{\eps_i}$   in the following  sense.
  \begin{enumerate}
\item   The $G$-slice  $ \cV_{u_i}$ is  a  finite dimensional  smooth  $G_{u_i}$-manifold.
\item    $E_{u_i}$ is a   $G_{u_i}$-equivariant  vector bundle.
\item    $\sigma_{u_i}$ is    a canonical $G_{u_i}$-invaraint section whose zeros set is $ \wtM \cap \cS_{u_i}^{\eps_i}$.
  \end{enumerate}
  }

\vspace{2mm}

This  collection of virtual  neighbourhoods
 \[
\{ (\cV_{u_i},    E_{u_i}, \sigma_{u_i} )\}
 \]
is a set  of Kuranishi neighbourhoods for $\cM$ in the sense of \cite{FO99}, but they are not compatible, so do not define
a Kuranishi  structure on $\cM$.   In Section \ref{5},  despite that the action of reparametrization group is non-differentiable, we can  apply the   virtual neighborhood technique  (the local and global stabilizations)
developed in  Section \ref{3} to $(\wtB, \wtE, \dbar_J)$, equipped with the  analytical foundation prepared in Section \ref{4},
 to get a virtual orbifold system. The main theorem about
pseudo-holomorphic spheres in this section
 is  the existence  of a  $G$-invariant virtual system for $(\wtB, \wtE, \dbar_J)$ whose slice is a finite dimensional virtual orbifold system for $(\cB, \cE, \dbar_J)$.

\vspace{2mm}

\nin{\bf Theorem D}  (Theorem \ref{main:thm})
{\em
 Let $(\wtB, \wtE, \dbar_J)$ be the Fredholm system  associated to pseudo-holomorphic spheres with the topological action
 of the reparametrization group $G=PSL(2, \C)$. Then the virtual  neighborhood
 technique in Theorem A can be applied in a $G$-equivariant way such that the following two statements hold.
  \begin{enumerate}
\item  The collection of triples $\{(\wt\cV_I, \wt  E_I,\wt\sigma_I)| I\subset  \{1, 2, \cdots, n\}\}$ is a finite dimensional virtual system  in the $G$-equivariant sense of Definition \ref{virtual-system} such the zero  sets $\{\wt\sigma_I^{-1}(0) \}_I$ form a cover of $\wtM$.
\item The collection of triples $\{( \cV_I, \bE_I, \sigma_I)| I\subset  \{1, 2, \cdots, n\}\}$ is a finite dimensional virtual orbifold
 system  in the  sense of Definition \ref{virtual-sys:orb} such the zero  sets $\{\sigma_I^{-1}(0) \}_I$ form a cover of $\cM$. \end{enumerate}

  }

\vspace{2mm}

 This paper is organised as follows. In Section \ref{2} we review and summarize the theory of  virtual manifolds/orbifolds    and  the integration theory on proper \'stale virtual groupoids.    Many examples  and properties of  virtual manifolds and virtual vector bundles are listed herefor  later use. In Section \ref{3}, we introduce a notion of virtual system  for a Fredholm system $(\cB, \cE, S)$, and develop a full machinery of virtual neighborhood techniques.  Theorem A is proved in Subsections \ref{3.3} and \ref{3.4}.
In Section 4, we study the issue of the  non-differentiable action of $PSL(2, \C)$ and establish the slice and neighbourhood theorem
(Theorem B) in Subsection 4.1. We also introduce a  class of smooth  cut-off functions and smooth local obstruction orbifold bundles in
Subsections 4.2 and 4.3 which are utilised to establish local virtual neighbourhoods of the moduli space of pseudo-holomorphic
spheres (Theorem C).
 In Section 5,  we  prove Theorem D by applying the virtual neighborhood technique developed in Section  \ref{3}
  to get a   $G$-invariant  virtual   system for   the Fredholm system  $(\wtB, \wtE, \dbar_J)$,
 whose slice  forms a  virtual  orbifold   system  for   $(\cB,  \cE,   \dbar_J)$. The invariants associated to this virtual  orbifold   system follow from the general integration theory on proper \'etale virtual groupoids.

  \section{Theory of Virtual manifolds and virtual orbifolds}\label{2}

  In this  section, we    give a self-contained review of  the theory of virtual manifolds/orbifolds  in  \cite{ChenTian} and  further develop the integration theory on proper \'stale virtual groupoids.

  \subsection{Basic Definitions}

\begin{definition}  \label{def:virmfd}  A {\em virtual manifold}  is a collection of  smooth manifolds $\{X_I\}_{I\in \cI} $ indexed by a partially ordered set
$(\cI = 2^{\{1, 2, \cdots, n\}}, \subset )$, together with  patching data
\[
\{(\Phi_{I, J}, \phi_{I, J}) |  I, J\in \cI, I \subset  J \},
\]
 where $\Phi_{I, J}: X_{J, I}\to X_{I, J}$ is a vector bundle with the zero  section $ \phi_{I, J}: X_{I, J}\to X_{J, I}$
 for open sub-manifolds $X_{I, J}$ and $X_{J, I}$ of $X_I$ and $X_J$ respectively,  whenever $I\subset J$. The patching datum
$(\Phi_{I, J}, \phi_{I, J}) $ for $I \subset  J$    will be  simply denoted by
 \[
   \xymatrix{X_{J, I}   \ar[d]_{\Phi_{I, J}} \ar[r]^\subset &  X_J   \\
   X_{I, J}     \ar@/_1pc/[u]_{ \phi_{I, J}   }  \ar[r]^\subset & X_I  . }
 \]
 Moreover,  the
 patching data  $\{ (\Phi_{I, J}, \phi_{I, J})|  I \subset  J \} $
satisfies  the following coherent conditions:\begin{enumerate}
\item Given any ordered triple $I \subset J\subset K$ with  the patching data
\[
 \xymatrix{  X_{J, I}   \ar[d]_{\Phi_{I, J}}   \ar[r]^\subset &X_J      & X_{K, J}   \ar[d]_{\Phi_{ J,K}} \ar[r]^\subset &X_K   &X_{K, I}   \ar[d]_{\Phi_{ I,K}} \ar[r]^\subset &X_K
  \\
X_{I, J}      \ar@/_1pc/[u]_{ \phi_{I, J}   }  \ar[r]^\subset & X_I ,          &X_{J, K}    \ar@/_1pc/[u]_{ \phi_{J, K}  }  \ar[r]^\subset &X_J,    & X_{I, K}   \ar@/_1pc/[u]_{ \phi_{I, K}} \ar[r]^\subset &X_I ,   }
\]
we have \begin{itemize}
\item
$X_{K, I}\subset X_{K, J}$ in $X_K$,
\item $X_{I, K}\subset X_{I, J}$ in $X_I$,
\item $ \Phi_{J, K} (X_{K, I})  = \Phi^{-1}_{I, J} (X_{I, K})  $  in $X_J$,
\item  and  the  cocycle condition  $\Phi_{I, K} =\Phi_{I, J}\circ \Phi_{J, K}$ as   given by  the following diagram
\ba\label{cohe:1}
  \xymatrix{  & X_{K, I}   \ar[dd]^{\Phi_{I, K}} \ar[dl]_{\Phi_{J, K}}    \\
\Phi_{J, K} (X_{K, I}) \ar[dr]^{\Phi_{I, J}}  &  \\
&  X_{I, K} .
   }
   \na
\end{itemize}
   \item Given a pair  $I$ and $J$,   $X_{I\cup J , I\cap J} =  X_{I \cup J, I}\cap  X_{I \cup J, J}$ and  $X_{I\cap J, I\cup J} =  X_{I \cap J, I}\cap  X_{I \cap J, J}$ hold,   and   the  diagram
   \ba\label{fiber:product}
\xymatrix{  & X_{I\cup J, I\cap J}   \ar[dl]_{\Phi_{I, I\cup J}}   \ar[dr]^{\Phi_{ J, I\cup J}} &  \\
X_{I,  J}  \ar[dr]_{\Phi_{I\cap J, I}}  &    &  X_{J, I} \ar[dl]^{\Phi_{I\cap J, J}}    \\
&  X_{I\cap J, I\cup J}  &
   }
\na
is a fiber product  of vector bundles over $ X_{I\cap J, I\cup J}$.  Here $X_{I, J} = \Phi_{I, I\cup J} (X_{I\cup J, I\cap J}) $
and  $X_{J, I} = \Phi_{J, I\cup J} (X_{I\cup J, I\cap J}) $ for any pair $I$ and $J$.

\end{enumerate}

\end{definition}

\begin{remark} \begin{enumerate}
\item   Definition \ref{def:virmfd} is  equivalent to Definitions 2.1 and 2.2 in \cite{ChenTian}.
  \item We can replace the collection of smooth manifolds $\{X_I\}$ by a collection of topological spaces where
  the patching data $\{\Phi_{I, J}:  X_{J, I}\to X_{I, J}\}$ consist of topological vector bundles, then the coherent conditions
  (\ref{cohe:1}) and (\ref{fiber:product}) still make sense in the category of  topological vector bundles. The resulting object is
  called a {\em topological virtual space}.
\end{enumerate}
\end{remark}

\begin{definition} \label{vir:bundle} ({\em Vector bundles and virtual vector bundles})   Let
$\{ X_I, \Phi_{I, J}, \phi_{I, J}| I\subset J \in 2^{\{1, 2, \cdots, n\}}\}$ be a  virtual manifold.
 \begin{enumerate}
 \item A {\bf  vector bundle}  over a virtual manifold
$\{ X_I\}$
is  a virtual manifold $\cF= \{F_I\}$ such that $F_I$ is a vector bundle over $X_I$  for each $I$,  and
 for any ordered pair $I\subset J$,
\ba\label{vir:bundle=1}
F_J|_{X_{J, I}} \cong  \Phi_{I, J}^* ( F_I|_{X_{I, J}}),
\qquad F_{I, J} =  F_I|_{X_{I, J}}, \qquad  F_{J, I} = F_J|_{X_{J, I}}.
\na

A section of a  vector bundle $\cF$ over  $\{ X_I\}$
 is a collection of sections
$
\{
S_I:  X_I \to F_I
\}
$
such that for any ordered pair $I\subset J$,  under the identification (\ref{vir:bundle=1}),
\ba\label{sect:vb}
S_J|_{X_{J, I}} =  \Phi_{I, J}^*\big (   S_I|_{X_{I, J}}\big).
\na
A section $\{ S_I\}$  of a    vector bundle $\cF$ is transversal if each $S_I$ is a transversal section of $F_I\to X_I$.

\item A {\bf  virtual vector bundle}  over a virtual manifold
$\{ X_I\}$ is a virtual manifold $
\bE=\{
E_I \to X_I \}
$
 such that $E_I$ is a vector bundle over $X_I$ for each $I$,  and
 for any ordered pair $I\subset J$,  $E_{I, J} = E_I|_{X_{I, J}}$, $E_{J, I} = E_J|_{X_{J, I}} $ and
\ba\label{vir:bundle=2}
E_J|_{X_{J, I}} \cong  \Phi_{I, J}^* (X_{J, I}\oplus E_I|_{X_{I, J}})
\na
as vector bundles over $X_{J, I}$.

A section of a virtual vector bundle $\bE$ over
$\{ X_I\}$
is   a collection of sections
$
\{
\sigma_I:  X_I \longrightarrow E_I,
\}
$
  called a virtual section of  $\bE$ if
  any ordered pair $I\subset J$,  under the identification (\ref{vir:bundle=2}),
\ba\label{sect:vir}
\sigma_J|_{X_{J, I}} =    (s_{X_{J, I}},  \sigma_I|_{X_{I, J}}\circ   \Phi_{I, J}  \big),
\na
where $s_{X_{J, I}}$ is the canonical section of  the bundle $ \Phi_{I, J}^* (X_{J, I}) $ over $X_{J, I}$.
A section $\sigma = \{ \sigma_I\}$  of a   virtual vector bundle $\bE$ is transversal if each $\sigma_I$ is a transversal section of $E_I\to X_I$.
\end{enumerate}
\end{definition}

\begin{example}\label{example:vir}
\begin{enumerate}
\item For a smooth  compact manifold $X$ with an open cover $\{U_i |  i=1, 2, \cdots, n\}$, there is a
canonical virtual manifold
structure
\ba\label{mfd-is-virfld}
\{ U_I, \Phi_{I, J}, \phi_{I, J}| I\subset J \in 2^{\{1, 2, \cdots, n\}}\}
\na
given by $U_I = \bigcap_{i\in I} U_i$ for any nonempty $I$ and $U_\emptyset = X$,
\[
U_{I, J} =U_{J, I}  = U_I\cap U_J
\]
and for $I\subset J$, $\Phi_{I, J}= \phi_{I, J} =Id_{U_I\cap U_J}$. Both coherent conditions
(\ref{cohe:1}) and  ( \ref{fiber:product})  are trivially satisfied. A virtual vector bundle over the virtual manifold
(\ref{mfd-is-virfld}) corresponds to a usual vector bundle over $X$.

\item ({\em Example 1 in \cite{ChenTian} })
The following example of virtual manifold structure on a compact manifold $X$ plays a central  role in the study of moduli spaces arising from Fredholm systems  in Section \ref{3}.

Let $\{U^{(2)}_i |  i=1, 2, \cdots, n\}$ be an open cover of $X$ such that for each  $U^{(2)}_i$ contains the  closure of an open set $U_i^{(1)}$.   For any   $I\subset \{1, 2, \cdots, n\}$, define
\ba\label{CT:1}
X_I =\big(X \cap\bigcap_{i\in I} U_{ i}^{(2)}\big) \setminus
\big(\bigcup_{j\notin I}\overline{U^{(1)}_{ i}}\big).
\na
Define  $X_{I, J} = X_{J, I} = X_I \cap X_J$ and $\Phi_{I, J} =\phi_{I, J} = Id_{X_I \cap X_J}$ for any
$I\subset J$.  Then  $\{ X_I,   \Phi_{I, J}, \phi_{I, J}\}$  is a virtual manifold as verified in \cite{ChenTian}. We   point out that
$X_I$ was  incorrectly  defined  for $I=\emptyset$ in \cite{ChenTian}.  By (\ref{CT:1}), we know that
\[
X_\emptyset = X \  \setminus
\big(\bigcup_{j\notin I}\overline{U^{(1)}_{ i}}\big).
\]

\item Given a virtual manifold $\{ X_I, \Phi_{I, J}, \phi_{I, J}| I\subset J \in 2^{\{1, 2, \cdots, n\}}\}$, its tangent
bundle
\[
\{
TX_I \longrightarrow X_I\}
\]
is a virtual  vector bundle as for any ordered pair $I\subset J$
\[
TX_J|_{X_{J, I}} \cong \Phi_{I, J}^* (X_{J, I}\oplus TX_I|_{X_{I, J}}),
\]
once a connection  on $\Phi_{I, J}: X_{J, I}\to X_{I, J}$ is chosen.  We remark that  the cotangent bundle
$\{T^*X_I\}$ and its $k$-th exterior power  $\{\bigwedge^k(T^*X_I)\}$ are  neither vector bundles nor  virtual  vector bundles in
 the strict sense
of Definition \ref{vir:bundle}.  As we will use later,
a section  $\omega =\{\omega_I\}$ of $\{\bigwedge^k(T^*X_I)\}$ will   be called
a degree $k$ {\bf  differential form} on $\{X_I\}$ if the following condition holds
\[
\omega_J|_{X_{J, I} } = \Phi_{I, J}^*  \big(\omega_I|_{X_{J, I} }\big).
\]
\end{enumerate}

The following proposition plays a key role in the construction of virtual system in this paper.
\begin{proposition}\label{key:prop}    Let
$\{ X_I, \Phi_{I, J}, \phi_{I, J}| I\subset J \in 2^{\{1, 2, \cdots, n\}}\}$ be a  virtual manifold.
\begin{enumerate}
\item Given a vector bundle $\cF =\{F_I \to X_I\}$ with a transversal section $\{S_I\}$, then the  collections of zero  sets
\[
\{Z_I =S_I^{-1}(0)\}
\]
admits a canonical  virtual manifold structure.  where  If $\{S_I\}$ is not transversal to the  zero section, then  $\{Z_I =S_I^{-1}(0)\}$ is only   a topological virtual space.

\item Given a virtual vector bundle $\{E_I\to X_I\}$ with a transversal  section $\{\sigma_I\}$,  then the collection of zero sets
\[
\{Y_I= \sigma_I^{-1}(0)\}, 
\]
under  the induced patching data,  forms a smooth manifold. In the absence of  transversality,
$\{Y_I= \sigma_I^{-1}(0)\}$ forms  a topological space under  the homeomorphism  $  \phi_{I, J}: Y_{J, I} \cong Y_{I, J}$.
\end{enumerate}

\end{proposition}

\begin{proof}  \begin{enumerate}
\item  Being a transversal section $\{S_I\}$ , each zero set $S_I^{-1}(0)$ for a smooth manifold. For $I\subset J$, we have 
\begin{itemize}
\item  $Z_{I, J} = Z_I \cap X_{I, J} = (S_I|_{X_{I, J}} )^{-1}(0)$,
\item $Z_{J, I} = Z_J\cap X_{J, I} = (S_J|_{X_{J, I}} )^{-1}(0)$,
\item The conditions (\ref{vir:bundle=1}) and  (\ref{sect:vb})  imply that $\Phi_{I, J}|_{ Z_{J, I}}: Z_{I, J} \to Z_{I, J}$ is a vector bundle.
\end{itemize}
It is easy to check  that the induced patching data satisfy the coherent conditions  (\ref{cohe:1}) and
(\ref{fiber:product}).   If $ \{S_I\}$ is not transversal to the  zero section, then   each $Z_I =S_I^{-1}(0)$ is only   a topological  space. The induced patching data defines a virtual topological structure. That is, 
n  $\{Z_I =S_I^{-1}(0)\}$ is only   a topological virtual space.

\item In this case, the collection of zero sets $\{Y_I= \sigma_I^{-1}(0)\}$ consists of smooth manifolds of the same dimention. Note that    for $I\subset J$, we have 
\begin{itemize}
\item  $Y_{I, J} = Y_I \cap X_{I, J} = (\sigma_I|_{X_{I, J}} )^{-1}(0)$,
\item $Y_{J, I} = Y_J\cap X_{J, I} = (\sigma_J|_{X_{J, I}} )^{-1}(0)$, under $\Phi_{I, J}$, is diffeomorphic  to $Y_{I, J}$.
\end{itemize}
Therefore, $\{Y_I= \sigma_I^{-1}(0)\}$ forms a smooth manifold.

Without the transversality condition, obviously,   $\{Y_I= \sigma_I^{-1}(0)\}$ forms a topological space, as
$Y_{J, I} = Y_J\cap X_{J, I} = (\sigma_J|_{X_{J, I}} )^{-1}(0)$, under $\Phi_{I, J}$, is only homeomorphic   to $Y_{I, J}$. \end{enumerate}
\end{proof}

\end{example}

\begin{remark} \label{virtual:bdy}
 As explained in  \cite{ChenTian}, we can define a virtual manifold   with boundary.  Here we require that each manifold $X_I$ is a manifold with a boundary $\partial X_I$, and the patching data satisfy the following condition
\[
\partial X_{J, I} = \Phi_{I, J}^{-1} (\partial X_{I, J}).
\]
Then $ \{\partial X_I\}_I$ is a virtual manifold. Moreover, Proposition \ref{key:prop} takes the following form.
 Let
$\{ X_I, \Phi_{I, J}, \phi_{I, J}| I\subset J \in 2^{\{1, 2, \cdots, n\}}\}$ be a  virtual manifold with a boundary
$\{\partial X_I\}_I$.
\begin{enumerate}
\item Given a vector bundle $\cF =\{F_I \to X_I\}$ with a transversal section $\{S_I\}$, then the  collections of zero  sets
\[
\{Z_I =S_I^{-1}(0), \partial Z_I\}
\]
is a virtual manifold with boundary.  where  If $\{S_I\}$ is not transversal to the  zero section, then  $\{Z_I =S_I^{-1}(0)\}$ is only   a topological virtual space with boundary.

\item Given a virtual vector bundle $\{E_I\to X_I\}$ with a transversal  section $\{\sigma_I\}$,  then the collection of zero sets
\[
\{Y_I= \sigma_I^{-1}(0)\}, 
\]
under  the induced patching data,  forms a smooth manifold with boundary. In the absence of  transversality,
$\{Y_I= \sigma_I^{-1}(0)\}$ forms  a topological space with boundary
\end{enumerate}
\end{remark}

\begin{remark} The definitions of virtual manifold  and virtual vector bundle  can be easily generalised to
equivariant cases   under Lie group actions and to orbifold cases. The virtual manifold with boundary can be defined in these set-up too. 
\begin{enumerate}
\item  ({\em Virtual $G$-manifold}) Given a Lie group $G$, a virtual $G$-manifold is  a collection of $G$-manifolds $\{X_I\}_{I\in \cI} $ indexed by a partially ordered set
$(\cI = 2^{\{1, 2, \cdots, n\}}, \subset )$, together with  patching data
\[
\{(\Phi_{I, J}, \phi_{I, J}) |  I, J\in \cI, I \subset  J \},
\]
 where $\Phi_{I, J}: X_{J, I}\to X_{I, J}$ is a $G$-equivariant vector bundle with the zero  section $ \phi_{I, J}: X_{I, J}\to X_{J, I}$
 for open $G$-invariant  sub-manifolds $X_{I, J}$ and $X_{J, I}$ of $X_I$ and $X_J$ respectively,  whenever $I\subset J$.  Moreover,  the
 patching data $\{ (\Phi_{I, J}, \phi_{I, J})|  I \subset  J \} $
satisfy  the  coherent conditions (\ref{cohe:1}) and  ( \ref{fiber:product})   in the  $G$-equivariant sense.
\item  ({\em Virtual orbifold})  A virtual orbifold is  a collection of orbifolds $\{\cX_I\}_{I\in \cI} $ indexed by a partially ordered set
$(\cI = 2^{\{1, 2, \cdots, n\}}, \subset )$, together with  patching data
\[
\{(\Phi_{I, J}, \phi_{I, J}) |  I, J\in \cI, I \subset  J \},
\]
 where $\Phi_{I, J}: X_{J, I}\to X_{I, J}$ is an orbifold  vector bundle with the zero  section $ \phi_{I, J}: \cX_{I, J}\to \cX_{J, I}$
 for open   sub-orbifolds $\cX_{I, J}$ and $\cX_{J, I}$ of $\cX_I$ and $X_J$ respectively,  whenever $I\subset J$.  Moreover,  the
 patching data $\{ (\Phi_{I, J}, \phi_{I, J})|  I \subset  J \} $
satisfy  the  coherent conditions (\ref{cohe:1}) and  ( \ref{fiber:product}) in the category of orbifold vector bundles.
\end{enumerate}
\end{remark}

It becomes  more evident  that  the language of  proper \'etale groupoids provides a convenient and economical  way to describe orbifolds.      We briefly recall the   definition of    a proper \'etale groupoid.

\begin{definition}  \label{Lie-gpoid:def} ({\em Lie groupoids and proper \'etale groupoids})
 A Lie groupoid
 $\cG = (\cG^0,  \cG^1)$  consists of two smooth manifolds $\cG^0$ and $\cG^1$, together with five smooth maps $(s, t, m, u, i)$ satisfying the following properties.
  \begin{enumerate}
\item  The source map  and the target map $s, t: \cG^1 \to  \cG^0$ are submersions.
\item The composition map
\[
m:  \cG^{[2]}: =\{(g_1, g_2) \in  \cG^1 \times  \cG^1: t(g_1) = s(g_2)\} \longrightarrow \cG^1
\]
written as $m(g_1, g_2) = g_1\circ  g_2$ for composable elements $g_1$ and $g_2$,
satisfies the obvious associative property.
\item The unit map $u: \cG^0 \to \cG^1$ is a two-sided unit for the composition.
\item The inverse map $i: \cG^1 \to \cG^1$, $i(g) = g^{-1}$,  is a two-sided inverse for the composition.
\end{enumerate}
In this paper,   a groupoid $\cG=(\cG^0, \cG^1)$   will be denoted by  $\cG = (\cG^1 \rightrightarrows \cG^0)$ where $\cG^0$ will be called the space of objects or units, and $\cG^1$ will be called the space of arrows.
A Lie groupoid $\cG = (\cG^1 \rightrightarrows \cG^0)$   is {\bf  proper } if $(s, t): \cG^1 \to \cG^0 \times \cG^0$ is proper, and called {\bf   \'etale } if $s$ and $t$ are local diffeomorphisms.   Given a proper \'etale groupoid $ (\cG^1 \rightrightarrows \cG^0)$, for any $x\in \cG^0$,  $(s, t)^{-1}(x, x) = s^{-1}(x)\cap t^{-1}(x)$ is a finite group, called the {\bf isotropy group}  at $x$.

\end{definition}

\begin{remark}  Given a proper \'etale Lie groupoid $\cG$, there is  a canonical orbifold structure on its orbit space
$|\cG|$ (\cite[Prop. 1.44]{ALR}).   Here the orbit space $|\cG|$ is the quotient space of $\cG^0$ by the equivalent relation from $\cG^1$.  Two Morita equivalent
proper \'etale Lie groupoids define two diffeomorphic  orbifolds  (\cite[Theorem 1.45]{ALR}).
Conversely,     given an orbifold $\cX$ with a given  orbifold  atlas, there is a canonical proper \'etale Lie groupoid $\cG_{\cX}$,  locally given by the action groupoid  of the orbifold charts.  See \cite{MoePr} and \cite{LU}.    Two equivalent  orbifold atlases define
two  Morita equivalent proper \'etale Lie groupoids.  Due to this correspondence, a proper \'etale Lie groupoid  will also be
called an orbifold groupoid. Often this same notion will be used particularly for the  proper \'etale Lie groupoid $\cG_\cX$
constructed from an orbifold atlas on an orbifold $\cX$.

\end{remark} \label{equ:orbi}

 An orbifold vector bundle $E \to \cX$  corresponds to a vector bundle  over  the  groupoid  $\cG_\cX$, which is a vector bundle  $\pi: E^0\to \cG^0$ with a fiberwise linear action
\[
\mu:  \cG^1\times_{(s, \pi)} E^0 \longrightarrow E^0
\]
covering the canonical action of $\cG^1$ on $\cG^0$ and
satisfying some obvious compatibility conditions. Here
\[
 \cG^1\times_{(s, \pi)} E^0 = \{ (g, v) \in \cG^1\times E^0|  s(g) = \pi(v)\}.
 \]
 In general, a vector bundle over a Lie groupoid $(\cG^1 \rightrightarrows \cG^0)$
is a vector bundle $E^0$ over $\cG^0$
with a fiberwise linear action
\ba\label{action:bundle}
\mu:  \cG^1\times_{(s, \pi)} E^0 \longrightarrow E^0,
\na
and a section of a vector bundle over $\cG$ is a section of $E^0$ which is invariant under the action of (\ref{action:bundle}).

   Note that the action  (\ref{action:bundle}) is defined by a section $\xi$ of the bundle $Iso(s^*E^0, t^*E^0) \to \cG^1$  where $Iso(s^*E^0, t^*E^0)$ is the bundle of bundle isomorphisms from $s^*E$ to
$t^*E$.   That means, given an arrow $\alpha \in \cG^1$, there is a linear  isomorphism of vector spaces
\[
\xi(\alpha):  E^0_{s(\a)}  \longrightarrow E^0_{t(\a)}
\]
such that $\xi (\a\circ \b) = \xi (\a) \circ \xi(\b)$.  In fact, the action  (\ref{action:bundle}) defines a  Lie groupoid
\[  E_1  :=    \cG^1\times_{(s, \pi)} E^0  \rightrightarrows  E^0
\]
with the source map $\tilde  s$  given by the projection, and the target map $\tilde  t$ given by the action.
This motivates the following characterisation of  vector bundles in the language of Lie  groupoids.

\begin{proposition} \label{bundle:gpoid}
Given a Lie groupoid  $\cG=(\cG^1 \rightrightarrows \cG^0)$,  a Lie groupoid
$( E^1  \rightrightarrows E^0)$ is a vector bundle over $\cG$ if and only if there is
 a strict Lie groupoid morphism $\pi:  ( E^1  \rightrightarrows E^0) \to (\cG^1 \rightrightarrows \cG^0)$ given by
 the commutative diagram
\ba\label{pull-back:pgoid}
\xymatrix{
E^1 \ar@<.5ex>[d]\ar@<-.5ex>[d] \ar[r]^{\pi_1} & \cG^1\ar@<.5ex>[d]\ar@<-.5ex>[d]\\
E^0 \ar[r]_{\pi_0}& \cG^0 }
\na
in the category of Lie groupoids with strict morphisms (this means,  the maps $\pi_1$ and $\pi_0$ commute with the source maps   and
 the target maps, and  are compatible to the composition, unit and inverse maps), such that
\begin{enumerate}
\item  the diagram (\ref{pull-back:pgoid}) is a pull-back groupoid diagram,
\item both $\pi_1: E^1\to \cG^1$ and $ \pi_0: E^0\to \cG^0$
are vector bundles.
\item for $ \a \in \cG^1$,  the pull-back arrows
\[
\{(v_x, \a, v_y) | (v_x, v_y) \in E_{s(\a)} \times E_{t(\a)} \}
\]
define  a linear isomorphism  $\xi (\a):  E_{s(\a)} \to E_{t(\a)} $ sending $v_x$ to $x_y$.
\end{enumerate}
A section  $s$ of a vector bundle $( E^1  \rightrightarrows E^0)$ over $(\cG^1 \rightrightarrows \cG^0)$ is  given by  a pair of sections
$(s_1, s_0)$ such that the following diagram commutes in category of Lie groupoids with strict morphism
\[
\xymatrix{
E^1 \ar@<.5ex>[d]\ar@<-.5ex>[d] \ar[rr]_{\pi_1} && \cG^1\ar@<.5ex>[d]\ar@<-.5ex>[d]  \ar@/_1pc/[ll]_{s_1}\\
E^0 \ar[rr]_{\pi_0}&& \cG^0  \ar@/_1pc/[ll]_{s_0} .}
\]
\end{proposition}
\begin{proof}
The proof is straightforward, so  it is omitted.
  \end{proof}

  With this  proposition  understood,  the definition of virtual $G$-manifolds and virtual orbifolds given in Remark \ref{equ:orbi}
   can be  written in terms of   Lie groupoids.  The resulting geometric object is called a virtual Lie groupoid.  We leave
   the explicit definition  of general virtual Lie groupoids to the  readers.  In this paper, we will only be interested in the study of
   virtual $G$-manifolds and virtual orbifolds (or equivalently, proper   \'etale  virtual groupoids).
   Here  a   {\bf proper   \'etale  virtual groupoid}   is  a collection of  proper   \'etale   groupoids
   $\{\cG_I\}_{I\in \cI} $ indexed by a partially ordered set
$(\cI = 2^{\{1, 2, \cdots, n\}}, \subset )$, together with  patching data
\[
\{(\Phi_{I, J}, \phi_{I, J}) |  I, J\in \cI, I \subset  J \},
\]
 where $\Phi_{I, J}: \cG_{J, I}\to \cG_{I, J}$ is a  vector bundle with the zero  section $ \phi_{I, J}: \cG_{I, J}\to \cG_{J, I}$
 for open   sub-groupoids $\cG_{I, J}$ and $\cG_{J, I}$ of $\cG_I$ and $\cG_J$ respectively,  whenever $I\subset J$.  Moreover,  the
 patching data $\{ (\Phi_{I, J}, \phi_{I, J})|  I \subset  J \} $
satisfy  the coherent conditions (\ref{cohe:1}) and  ( \ref{fiber:product}) in the category of    proper   \'etale    groupoids  with  the strict morphisms.

We remark that a finite rank  (virtual) vector bundle over a proper   \'etale  virtual groupoid can be defined in a similar manner.  Examples in Example \ref{example:vir} all have their proper   \'etale  virtual groupoid counterparts. In particular, we have the foliowing proposition  whose proof is is analogous to  the proof of Proposition \ref{key:prop}.  There is a version of this proposition for a   proper   \'etale  virtual groupoid with boundary.

\begin{proposition}   Let
$\{ \cG_I, \Phi_{I, J}, \phi_{I, J}| I\subset J \in 2^{\{1, 2, \cdots, n\}}\}$ be a   proper   \'etale  virtual groupoid.
\begin{enumerate}
\item Given a vector bundle $\cF =\{F_I \to \cG_I\}$ with a transversal section $\{S_I\}$, then the  collections of zero  sets
\[
\{Z_I =S_I^{-1}(0)\}
\]
defines  a canonical   proper   \'etale  virtual groupoid.  where  If $\{S_I\}$ is not transversal to the  zero section, then  $\{Z_I =S_I^{-1}(0)\}$ is only   a topological proper   \'etale  virtual groupoid.

\item Given a virtual vector bundle $\{E_I\to \cG_I\}$ with a transversal  section $\{\sigma_I\}$,  then the collection of zero sets
\[
\{Y_I= \sigma_I^{-1}(0)\}, 
\]
under  the induced patching data,  forms a smooth   proper   \'etale  groupoid.  In the absence of  transversality,
$\{Y_I= \sigma_I^{-1}(0)\}$ forms  a topological   proper   \'etale  groupoid,  under  the homeomorphism  $  \phi_{I, J}: Y_{J, I} \cong Y_{I, J}$.
\end{enumerate}
\end{proposition}

 As virtual manifolds are special cases of virtual orbifolds, we only  develop  the integration theory for   virtual orbifolds using the language of proper   \'etale  groupoids, following closely the corresponding integration theory for  virtual manifolds in \cite{ChenTian}. This will done in Subsection \ref{2.2} after we discuss the orientation  on proper   \'etale  virtual groupoids.

\subsection{Differential forms  and orientations on proper   \'etale  virtual groupoids} \label{2.2}

A differential form on a proper   \'etale   groupoid $\cG=  (\cG^1 \rightrightarrows \cG^0)$ is a differential form on
$\cG^0$ which is invariant under the action of $\cG^1$.  By Proposition
 \ref{bundle:gpoid}, a differential $k$-form $\omega$  on $\cG$ is a pair of sections
 $(\omega_0, \omega_1)$  of the $k$-th exterior  power of the cotangent bundles
 such that the diagram commutes
 \ba\label{diag:form}
 \xymatrix{
\left( \bigwedge^k  T^*\cG\right)^1 \ar@<.5ex>[d]\ar@<-.5ex>[d] \ar[rr]_{ \pi_1 } && \cG^1\ar@<.5ex>[d]\ar@<-.5ex>[d] \ar@/_1pc/[ll]_{ \omega_1} \\
\left(  \bigwedge^k T^*\cG\right)^0 \ar[rr]_{ \pi_0 }&& \cG^0 \ar@/_1pc/[ll]_{ \omega_0}.}
 \na
We remark that  $\left(  \bigwedge^k T^*\cG\right)^0 =  \bigwedge^k T^*\cG^0$ and $ \left(  \bigwedge^k T^*\cG\right)^1$   defined  as in  Proposition  \ref{bundle:gpoid} agrees with $ \left(  \bigwedge^k T^*\cG^1\right) $,  the existence of the section $\omega_1$ is guaranteed by $\omega_0$ and the
invariance property under the action
  $\cG^1$.

  A differential form $\omega$ on $\cG$ is {\bf  compactly  supported} if the support of
  $\omega$ is $\cG$-compactly support in the sense that the $\cG^1$-quotient  of the support  is compact in $|\cG|$.
   The space of differential $k$-forms (with $\cG$-compact support, respectively) will be denoted by
   $\Omega^k (\cG)$ and $\Omega_c^k (\cG)$ respectively.

 In order to define the integral of certain  differential forms (to be specified later) over    proper   \'etale  virtual groupoids, we need to define a notion of orientation on it.  First we recall  the orientable condition  on a  proper   \'etale  groupoid.  Given a proper   \'etale  groupoid $\cG$,  the  tangent bundle $T\cG$,  by applying Proposition
 \ref{bundle:gpoid},   is a vector bundle over $\cG$ given by  a strict morphism $\pi:  T\cG \to \cG$,  that is,  a commutative diagram in the category of  \'etale  virtual groupoids
 \[
 \xymatrix{
(T\cG)^1 \ar@<.5ex>[d]\ar@<-.5ex>[d] \ar[r]^{\pi_1} & \cG^1\ar@<.5ex>[d]\ar@<-.5ex>[d]\\
(T\cG)^0 \ar[r]_{\pi_0}& \cG^0 }
\]
where $(T\cG)^0 = T\cG^{0}$ is the usual tangent bundle of $\cG^{0}$.    An orientation of an  oriented  proper  \'etale    groupoid $\cG$ is equivalent to a choice of a trivialisation of the determinant (the highest exterior power) bundle
\[
 \pi:  \big( ( \det  T\cG)^1 \rightrightarrows  (\det  T\cG)^0  \big) \longrightarrow   \big(\cG^1 \rightrightarrows \cG^0 \big)
\]
of  the tangent bundle $T\cG$,
with $(\det  T\cG)^0 = \det (T\cG^0)$ and $(\det  T\cG)^1 = \det (T\cG^1)$. That is, a pair of nowhere vanishing sections $\bo_0$ and $\bo_1$ of
$\pi_i: ( \det  T\cG)^i \to \cG^i, i=1, 2$, such that
the diagram
\ba\label{orient:gpoid}
\xymatrix{
(\det  T\cG)^1 \ar@<.5ex>[d]\ar@<-.5ex>[d] \ar[rr]_{ \pi_1 } && \cG^1\ar@<.5ex>[d]\ar@<-.5ex>[d] \ar@/_1pc/[ll]_{\bo_1} \\
(\det T\cG)^0 \ar[rr]_{ \pi_0 }&& \cG^0 \ar@/_1pc/[ll]_{\bo_0}}
\na
commutes in the category of  groupoids  with  strict morphisms. An orientation on $\cG$  will be  denoted by a nowhere vanishing section  $\bo_\cG:   \cG \longrightarrow  \det T\cG$.

Similarly, given a real vector bundle  $E$ over a groupoid $\cG$, $E$ is orientable if and only if $\det E$ is trivializable in the sense that there is a nowhere
vanishing section $\bo_E: \cG \to \det E$ (defining an orientation on $E$).

 Now we can give a definition of an orientation on  a proper   \'etale  virtual groupoid.

\begin{definition} \label{orient:vir} Given  a proper   \'etale  virtual groupoid $\{(\cG_I, \Phi_{I, J}, \phi_{I, J}) | I\subset J\}$, we say that
it is orientable if there exists a coherent  choice of orientations on $\{\cG_I\}$ and $\{\Phi_{I, J}: \cG_{I, J}
\to \cG_{I, J} | I\subset J\}$. That is, a system of nowhere vanishing sections (called an orientation on $\{\cG_I\}$)
\[
\bo_I:  \cG_I \longrightarrow  \det T\cG_I,   \qquad \bo_{I, J}:  \cG_{I, J}  \longrightarrow  \det \cG_{J, I}
\]
for every $I$   and   every  ordered pair  $I\subset J$, satisfying
\[
\bo_J|_{\cG_{J, I}} =  \Phi_{I, J}^* (\bo_I|_{\cG_{I, J}}  \otimes \bo_{I, J}   ),
\]
   under the  isomorphisms
\[
(\det  T\cG_J) |_{\cG_{J, I}}
\cong  \Phi_{I, J}^*\big( ( \det    T\cG_I) |_{\cG_{I, J}} \otimes  (\det  \cG_{J, I})    \big).
\]
A virtual vector bundle $\{E_I\}$ over an oriented proper  \'etale  virtual groupoid $\{\cG_I\}$ is called oriented if
 there is a collection of orientations $\{\bo_{E_I}\}$ such that  the bundle isomorphisms (\ref{vir:bundle=2})
 preserve the orientations.
\end{definition}

  Given an oriented  proper \'etale  groupoid $\cG$, introduce the notation
   \[
   \Omega^*_c(\cG) =\bigoplus_k \Omega_c^k (\cG).
   \]
  Then  there is a well-defined linear functional
   \[
   \int_\cG:   \Omega^*_c(\cG) \longrightarrow \R
   \]
   given by
   \ba\label{int:top}
   \int_\cG \omega = \int_\cG \omega^{top},
   \na
    where $\omega^{top}$ is the  degree $m$ component of $\omega$ with $m =\dim \cG_0$.
We now recall  the definition of integration over a  proper
     \'etale  groupoid   for example as  in \cite{ALR}.  Denote by $|\cG|$ the quotient space of $\cG^0$ by the equivalence  relation from $\cG^1$. Let  $\{(G_i\ltimes \tilde U_i \rightrightarrows \tilde U_i)\}$ be
 a collection of sub-groupoids  with each $G_i$ being a finite group such that  the collection of quotient spaces $\{ U_i =  \tilde U_i /G_i\}$ forms an open cover of $|\cG|$. There is a collection of smooth functions
 $\{\rho_i\}$ with the following property
 \begin{enumerate}
 \item $\rho_i$ is a $G_i$-invariant function on $\tilde U_i$, hence defines a  continuous function $\bar\rho_i$
 on  $U_i$.
\item $ \supp (\rho_i) \subset \tilde U_i$ is compact.
\item $0\leq \rho_i \leq 1$.
\item For every $x\in |\cG|$, $\sum_i \bar \rho_i (x) =1.$
\end{enumerate}
A partition of unity subordinate to $\{(\tilde U_i, G_i)\}$ can be obtained from a  routine construction as  the manifold case. Then the integration (\ref{int:top})
\ba\label{int:def}
\int_\cG \omega = \sum_{i} \dfrac{1}{|G_i|} \int_{\tilde U_i} \rho_i \cdot \omega
\na
  is independent of the choice of $\{\rho_i\}$ and the covering sub-groupoids $\{(G_i\ltimes \tilde U_i \rightrightarrows \tilde U_i)\}$.  For simplicity, we write
\ba\label{int:simple}
\dfrac{1}{|G_i|} \int_{\tilde U_i} \rho_i \cdot \omega  = \int_{ U_i} \bar \rho_i \cdot \omega .
\na

  In order to develop  a meaningful integration theory on proper   \'etale   virtual  groupoids, we will introduce a special
  class of differential forms, called  twisted virtual forms below.  Recall that a {\bf Thom form}   for a rank $m$ oriented  vector bundle
  $\pi:  \cE = ( E^1  \rightrightarrows E^0) \to  \cG = (\cG^1 \rightrightarrows \cG^0)$ is a differential $m$-form  $\Theta_\cE$ on $\cE$
  with compact  vertical support such that under the integration along the fiber
  \[
  \pi_*:   \Omega^k_{cv} (\cE) \longrightarrow  \Omega^{k-m}  (\cG),
  \]
  we have
  \[
  \pi_* (\Theta_\cE) =1  \in  \Omega^{0}  (\cG).
  \]
  Here $  \Omega^k_{cv} (\cE) $ is the space of differential $k$-forms on $\cE$ with compact  vertical support. The corresponding
  cohomology will be denoted by $H^{*}_{cv}(\cE)$.
  The construction of Thom forms with support in an  arbitrary small neighbourhood of the zero section is quite standard, one can adapt  the construction in Chapter 1.6 in \cite{BGV} to  the cases of   oriented  proper   \'etale  groupoids.

    The Thom form plays an important role in differential geometry through  the Thom isomorphism,  the  projective and product  formulae.
  \begin{enumerate}
\item ({\bf Thom isomorphism})  The map sending  $\omega$ in $ \Omega^*  (\cG)$  to $\pi^*\omega \wedge \Theta_\cE
\in  \Omega^{*+k}_{cv} (\cE)$ defines a  functorial isomorphism
\[
H^*(\cG) \cong H^{*+m}_{cv}(\cE)
\]
whose inverse map is   $\pi_*$.
\item ({\bf Projection formula}) Suppose that $\cG$ is oriented,  then  for every closed differential form $\omega \in \Omega^{*+k}_{c} (\cG)
$,
 \ba\label{proj}
 \int_\cE \Theta_\cE \wedge \pi^*  \omega  = \int_\cG \omega
 \na
 \item ({\bf Product  formula})   Let $\cE_1$ and $ \cE_2$ be  two oriented vector bundles over $\cG$, and $\pi_1$ and $\pi_2$ be  two projections from
 $\cE_1\oplus  \cE_2$ to   $\cE_1$ and  $\cE_2$. If  $\Theta_{\cE_1}$ and $\Theta_{\cE_2}$  are  Thom forms of  $\cE_1$ and  $\cE_2$  with compact  vertical
 support, then
  \[
 \pi_1^* \Theta_{\cE_1} \wedge \pi_2^*  \Theta_{\cE_2}
 \]
 is a   Thom form   of  $\cE_1 \oplus\cE_2$ with compact  vertical
 support.
\end{enumerate}
We remark that  Thom isomorphism and the Projective/Product formulae follow from the same proofs for  smooth manifold cases as in
\cite{BottTu}, as the existence of a partition of unity and the Mayer-Vietoris principle hold for  proper   \'etale   groupoids.

  \begin{definition} \label{Thom:form} Let  $\{\cG_I\}   = \{(\cG_I, \Phi_{I, J}, \phi_{I, J}) | I\subset J\}$ be a  proper   \'etale oriented   virtual groupoid  with an
   orientation, that is, a coherent orientation system $\{\bo_I, \bo_{I, J}: I\subset J\}$ as in Definition \ref{orient:vir}.
  \begin{enumerate}
\item A collection of Thom forms
$
\Theta = \{\Theta_{I, J} |  I\subset J \},
$
where $\Theta_{I, J}$ is a Thom form of   the bundles  $\{\Phi_{I, J}:  \cG_{J, I}\to \cG_{I, J}\}$,  is called a
 {\bf  transition Thom form }  for  $\{\cG_I\}$,  if the following two conditions are satisfied.
\begin{enumerate}
\item For any ordered triple $I\subset J\subset K$, $\Theta_{I, K} = \Theta_{J, K} \wedge \Phi_{J, K}^* \big( \Theta_{I, J}\big)$ as differential forms
on $X_{K, I}$.
\item For any $I$ and $J$, $\Theta_{I\cap J, I\cup J} = \Phi_{I, I\cup J}^* \big(\Theta_{I\cap J, I}\big)
\wedge \Phi^*_{J, I\cup J} \big( \Theta_{I\cap J, J}\big)$
on $\cG_{I\cup J, I \cap J}$.
\end{enumerate}
\item Given  a transition  Thom form $\Theta  =\{\Theta_{I, J} |  I\subset J \}$, a collection of differential forms
\[
\omega =\{\omega_I \in \Omega^k_c(\cG_I)  \}
\]
is called a  {\bf $\Theta$-twisted   differential form}  on $\{\cG_I\}$ if $\{\omega_I\}$ satisfy
the following condition
\ba\label{Theta:twisted}
\omega_J |_{\cG_{J, I}} = \Phi_{I, J}^*( \omega_I  |_{\cG_{I, J}} )\wedge \Theta_{I, J}.
\na
\end{enumerate}
  \end{definition}

   The space of all  $\Theta$-twisted differential
   forms  with compact support on $\{\cG_I\}$ is denoted by
   $\Omega^*_c (\{ \cG_I\}, \Theta)$. The de Rham differential
    is well-defined on  $\Omega^*_c (\{ \cG_I\}, \Theta)$,
    so $(  \Omega^*_c (\{ \cG_I\}, \Theta), d)$ is a complex,  whose cohomology
   \[
   H^k_c(\{ \cG_I\}, \Theta) = \dfrac{\{ \text {closed $\Theta$-twisted  differential forms on $\{\cG_I\}$\}}} {\{
  \text {exact  $\Theta$-twisted  differential forms on $\{\cG_I\}$}\}},
  \]
is  called the $\Theta$-twisted cohomology of  $\{\cG_I\}$. Up to a canonical isomorphism, this cohomology 
is independent of the choice of the transition Thom form $\Theta  =\{\Theta_{I, J} |  I\subset J \}$.

   \subsection{Integration on  proper   \'etale oriented   virtual groupoids}

In this subsection, we define a notion of integrationcmap on the space of  $\Theta$-twisted virtual differential forms
\[
\int_{\{\cG_I\}}:  \Omega^*_c (\{ \cG_I\}, \Theta) \longrightarrow \R
\]
 for an  oriented  proper   \'etale   virtual groupoid  $\{\cG_I\}$ with  a  transition Thom form $\Theta  =\{\Theta_{I, J} |  I\subset J \}$.   We need to assume  that
a  partition of unity exists for  a topological space obtained from $\{\cG_I\}$.

For any   proper   \'etale   virtual groupoid   $\{\cG_I\}$, there is a topological space
\ba\label{support:vir}
|\{\cG_I\}| = \bigsqcup_I   | \cG_I |  /\sim
\na
where $\sim$ is an equivalence relation: for $x\in  | \cG_I |$ and $y\in  | \cG_J | $, $x\sim y$ if and only if
there exists a $K\subset I\cap J$ such that
\ba\label{equi:rel}
|\Phi_{K, I}|  (x) = |\Phi_{K, J} | (y)
\na
where $|\Phi_{K, I}|: |\cG_{I, K}| \to |\cG_{K, I}|$ and $|\Phi_{K, J}|: |\cG_{J, K}| \to |\cG_{K,J}|$ are the quotient maps of $\Phi_{K, I}$ and $\Phi_{K, J}$ respectively. Denote by $\pi_I:  |\cG_I| \to |\{\cG_I\}| $   the obvious projection map.  We call the space
$|\{\cG_I\}|$ the {\em support } of   $\{\cG_I\}$, a section of a (virtual)  vector bundle over  $\{\cG_I\}$ is called {\em compactly supported} if
it has   compact support in  $|\{\cG_I\}|$.
 We remark that the above constructions can be applied to topological virtual spaces.

\begin{definition} \label{POU:vir} Let $\{\cG_I\}$ be a  proper   \'etale   virtual groupoid.  A  collection of smooth functions
$
\{\rho_I \in C^\infty (\cG_I ) \}
$
is called  a partition of unity for
$\{\cG_I\}$  if the following conditions are satisfied.
\begin{enumerate}
\item each $\rho_I$ is invariant under the above equivalence relation, hence defines a continuous function $\bar\rho_I:   \pi_I ( |\cG_I|) \to \R$.
\item $\supp (\bar\rho_I)$ is compact, this condition can be omitted when the $\Theta$-twisted form  $\omega_I$ can be chosen such that the product $\rho_I \omega_I$ is compactly supported on $\cG_I$.
\item For every $x\in |\{\cG_I\}|$, $\sum_I \bar\rho_I (x) =1$.
\end{enumerate}
\end{definition}

Under the assumption  that a transition Thom form  $\Theta$ and    a partition of unity  $\{\rho_I \in C^\infty (\cG_I ) \}$
 exist   for   an  oriented  proper   \'etale  virtual groupoid $\{\cG_I\}$, we define
\ba\label{vir:int}
\int_{\{\cG_I\}}:  \Omega^*_c (\{ \cG_I\}, \Theta) \longrightarrow \R
\na
to be
\[
\int_{\{\cG_I\}} \omega = \sum_I \int_{\cG_I}  \rho_I \cdot \omega_I
\]
for any $\Theta$-twisted differential form
$ \omega =\{\omega_I \in \Omega^\ast_c(\cG_I)  \}  \in \Omega^*_c (\{ \cG_I\}, \Theta)$. In practice, one needs to establish the existence  of  a transition Thom form   $\Theta$ and    a partition of unity  $\{\rho_I \in C^\infty (\cG_I ) \}$
 exist   for   a  proper   \'etale oriented   virtual groupoid $\{\cG_I\}$.

\begin{theorem} \label{int:well-def} Given  an  oriented   proper   \'etale   virtual groupoid  $\{\cG_I\}$, assume that
a transition Thom form  $\Theta$ and    a partition of unity  $\{\rho_I \in C^\infty (\cG_I ) \}$
 exist.    The integration map  (\ref{vir:int}) is well-defined, that is, it is independent of
the choice of      a partition of unity  $\{\rho_I \in C^\infty (\cG_I ) \}$.  Assume further that 
$\{\cG_I\}$ is an  oriented   proper   \'etale   virtual groupoid  with boundary. Then 
the Stokes' formula hold 
\[
\int_{\{\cG_I\}}  d \omega  = \int_{\{\partial \cG_I\}}  \iota^*  \omega  
\]
for $\omega \in \Omega^*_c (\{ \cG_I\}, \Theta)$. Here $\iota: \{\partial \cG_I\} \to \{  \cG_I\}$ is the inclusion map.
\end{theorem}
\begin{proof}
It suffices to show that for every sufficiently small  open subset $U $ of $ |\{\cG_I\}|$,
\[
 \sum_I \int_{\pi_I^{-1}(U )}    \rho_I \cdot \omega_I
 \]
 is independent of
the choice of      a partition of unity  $\{\rho_I \in C^\infty
(\cG_I ) \}$ for any   $ \omega =\{\omega_I \in
\Omega^\ast_c(\cG_I) \}  \in \Omega^*_c (\{ \cG_I\}, \Theta)$.

 When $U$ small enough, there exists an $I_0$ such that
$\pi_{I_0}(\pi_{I_0}^{-1}(U))=U$. Then we claim that
\begin{equation}\label{eqn2.3.1}
\sum_I\int_{\pi_I^{-1}(U)} \rho_I\omega_I=\int_{\pi_{I_0}^{-1}(U)}
\omega_{I_0}.
\end{equation}
For simplicity, we introduce the  notation
$$
U_I=\pi_I^{-1}(U),\;\;\; U_{I_0,I}=\pi_{I_0}^{-1}(\pi_I(U_I)).
$$
Then we claim that
\begin{equation}\label{eqn2.3.2}
\int_{U_I} \rho_I\omega_I=
\int_{U_{I_0,I}}  \rho_I\omega_{I_0}.
\end{equation}
In fact, by the definition of $\Theta$ and $\Theta$-twisted virtual differential form, we have 
$$
\int_{U_I} \rho_I\omega_I=
\int_{\Phi_{I_0\cap I,I} (U_I)} \rho_I\omega_{I_0\cap I}
=\int_{\Phi_{I_0\cap I,I_0} (U_{I_0, I})} \rho_I\omega_{I_0\cap I}
=\int_{U_{I_0}} \rho_I\omega_{I_0}.
$$
Hence, the left hand side of \eqref{eqn2.3.1} becomes
$$
\sum_I\int_{U_{I_0,I}} \rho_I\omega_{I_0}
=\sum_I\int_{U_{I_0}} \rho_I\omega_{I_0}=\int_{U_{I_0}}\omega_{I_0}.
$$

The proof of the   Stokes' formula is straightforward, see  \cite{ChenTian} for example for the virtual manifold case. 
\end{proof}

\begin{remark} The assumption of the existence of a transition Thom form and  and    a partition of unity  in Theorem \ref{int:well-def}
will automatically be  satisfied for those  proper  \'etale  virtual groupoids arising from the Gromov-Witten theory.
\end{remark}

   \section{Virtual neighborhood technique for smooth orbifold  Fredholm systems}\label{3}

  In this section, we will introduce  a notion of virtual systems arising from  Fredholm systems
   following closely  \cite[Sections 5 and 6]{ChenTian}.
   Recall that  a {\bf Fredhom system}  in \cite{ChenTian} is a triple
$
  (\cB, \cE, S),
$
consisting of
  \begin{enumerate}
\item a smooth Banach manifold  $\cB$,
\item a smooth Banach bundle $\pi: \cE \to \cB$,
\item a  smooth   Fredholm section $S: \cB\to \cE$.
\end{enumerate}
The zeros of the section $S$, denoted by $M =S^{-1}(0)$, is called the moduli space of the Fredholm system 
$  (\cB, \cE, S).$ 
 We remark that a  section $S$ is called Fredholm if
 the differential in the fiber direction
as  a linear operator
\[
D_xS:   T_x\cB \longrightarrow T_{S(x)} \cE \longrightarrow \cE_x,
\]
for any $x$ in a small neighbourhood of $M$, is a Fredholm operator. This operator
depends on a  choice of a connection on $\cE$, but the
Fredholm property doesn't depend on this choice. In all geometric applications, the operator 
$D_xS$ is Fredholm for all $x\in \cB$. 
 On the other hand,
for  every $x\in M$, the  operator $D_xS$
is given by the linearization of $S$ which does not depend on  the 
choice of a connection on $\cE$.  If the section $S$ is transversal to the zero section, that is,
the
 operator
 \[
 D_xS:   T_x\cB \longrightarrow \cE_x
\]
is surjective at $x\in M$, then $M$ is a smooth manifold of dimension given by the Fredholm index of
$D_xS $.  When the transversality of $S$ fails, one needs to apply  the  virtual neighborhood technique to get a
virtual manifold as originally  proposed in \cite{ChenTian}.  We  assume    that $M= S^{-1}(0)$ is compact.

We remark that this Section only serves as a guideline to
 overcome the failure of transversality issues. In practice,such as for the  moduli spaces
of stable maps, the Fredholm system is often not smooth, just as we encountered in Section \ref{4}  for
the moduli spaces of pseudo-homolomorphic spheres. Even in the absence of pseudo-homolomorphic spheres,
the underlying Fredholm system is only  smooth  in the  stratified sense.  One would get a Fredholm system for each strata. These issues
of non-smoothness  and stratifications  pose some new  technical difficulties in obtaining virtual manifold/orbifolds in general. These issues have to be dealt with  in each  case individually  when they come up.   Nevertheless, the main idea to get a virtual system
   is essentially the  same.  The main objective  of this Section is to explain the  constructions in \cite{ChenTian}
    and show that  the existence of a transition Thom form and  and    a partition of unity  follow naturally from the infinite dimensional Banach system.

  \subsection{From  smooth Fredholm systems to finite dimensional  virtual systems}

  \begin{definition}  \label{virtual-system}
  ({\bf Finite dimensional virtual system})  A    {\em finite dimensional virtual system}   is a
  collection of  triples
 \[
  \{(\cV_I, \bE_I, \sigma_I)|  I\subset \{1, 2, \cdots, n\} \}
  \]
  indexed by a   partially ordered
set $(\cI =2^{\{1, 2, \cdots, n\}}, \subset )$, where
  \begin{enumerate}
\item $\cV=\{\cV_I \}$ is a fiinite dimensional virtual manifold;
\item $\bE=\{\bE_I\}$ is a  finite rank virtual vector bundle over $\{X_I\}$;
\item $\sigma=\{\sigma_I\}$ is a virtual section  of the  virtual vector bundle $\{\bE_I\}$.
 \end{enumerate}
  \end{definition}
\noindent Then by the arguments in  Example \ref{example:vir} (5),
 the  zero sets
 \[
 \{  \sigma_I^{-1}(0) |  I\subset \{1, 2, \cdots, n\} \}
\]
forms a topological   space, denoted by $|\{  \sigma_I^{-1}(0)\}|$.

   Let $ (\cB, \cE, S)$ be a Fredholm system such that $M=S^{-1}(0)$ is compact.  We will construct a   finite dimensional virtual system $ \{(X_I, \bE_I, \sigma_I)|  I\subset \{1, 2, \cdots, n\} \}$   using local and global stabilizations such that
   there exist a collection of homeomorphisms
   \[
   \{\psi_I:  \sigma_I^{-1}(0) \longrightarrow  U_I
   \}\]
  from   $\sigma_I^{-1}(0)$
 to an open subset $U_I\subset M$,
  where  $\{U_I\}$ is  an open cover of $M$.

  \subsubsection{Local stabilizations}\label{local:stab}

  \begin{definition} \label{l.v.n}
   Given a topological space $M$ and a point $x\in M$, a {\em  local  virtual neighbourhood  } at $x$ is a 4-tuple
   \[
   (\cV_x , \bE_x, \sigma_x, \psi_x)
   \]
   consisting of
   \begin{enumerate}
\item a smooth finite dimensional manifold $\cV_x$,
\item a smooth vector bundle $\bE_x$ with a section $\sigma_x$ and
\item a homeomorphism $\psi_x:  \sigma^{-1}(0) \to U_x$ for an open neighbourhood $U_x$ of $x$ in $M$.
\end{enumerate}

  \end{definition}

Given a Fredholm system $(\cB, \cE, S)$,   for $x\in M = S^{-1}(0)$, setting $F=\cE_x$ (the fiber of $\cE$ at $x$),  there    is
a neighborhood $U$ of $x$ such that $\cE$ is trivialized
over $U$, that is, a bundle isomorphism
\ba\label{trivial:cE}
\phi_x: \cE|_{U}\to  U\times F.
\na
 Denote by $\phi_{x;y}: \cE_y\to F$  the   isomorphism induced by (\ref{trivial:cE}) for $y\in U$.

Consider the Fredholm operator $DS_x: T_x\cB\to F$. Choose
$K_x\subset F$  to be a  finite dimensional subspace of $F$ such that
\ba\label{surj:S}
K_x+ D_xS  (T_x\cB)=F.
\na
Define a thickened Fredholm system $(U\times K_x, \cE|_{U}\times K_x, \tilde S_x)$, where
\[
\tilde S_x(y,k)= S(y)+\phi_{x;y}^{-1}(k)  \in  \cE_y =( \cE|_{U}\times K_x)_{(y, k)} .
\]
By (\ref{surj:S}), there exists a smaller  neighborhood $U$ of $x$
such that $D_{(y,k)}\tilde S_x$ is surjective
for any $y\in U$.

We impose  the following assumption on $ \cB$.

\begin{assumption}\label{assump}
There are open neighborhoods  $U_{x}^{(i)}$ of $x \in M$ for $i=1,2,3$ such that
$$
U_{x}^{(1)}\subset \overline{U_{x}^{(1)}}
\subset U_{x}^{(2)}\subset \overline{U_{x}^{(2)}}
\subset U_{x}^{(3)}\subset \overline{U_{x}^{(3)}}\subset U
$$
and there is
 a {\bf  smooth} cut-off function
$\beta_x:  U_{x}^{(3)} \to [0, 1]$ such that
$\beta_x\equiv1$ on  $U_x^{(1)}$ and
is supported in
$U_x^{(2)}$.  Here $\overline{U_{x}^{(i)}}$ is the closure of $U_{x}^{(i)}$.
\end{assumption}

Such a nested triple $U_{x}^{(1)}\subset  U_{x}^{(2)}\subset U_{x}^{(3)} $ in $U$  and a cut-off function $\beta_x$ can be found
as $U$ is an open  neighborhood  of $x$  in a Banach manifold $\cB$.  The only non-trivial part of Assumption
(\ref{assump})  is the smoothness
of the cut-off function.
 One  only needs to  verify the smoothness of this cut-off function in any  application of the virtual neighborhood technique.

\begin{definition}
A {\em local stabilization} of $ (\cB, \cE, S)$ at $x\in M$
  is defined to be  the   thickened Fredholm system
  \[
  (U^{(3)}_x\times K_x, \cE|_{U^{(3)}_x}\times K_x,
  S_x)
\]
with
$
  S_x(y,k)= S(y)+\beta_x (y)\phi_{x;y}^{-1}(k).
$
\end{definition}

The local stabilization process turns the Fredholm system $ (\cB, \cE, S)$ at each point  $x\in M$ to a local transversal Fredholm system in the sense that the local Fredholm system
\[
 (U^{(1)}_x\times K_x, \cE|_{U^{(1)}_x}\times K_x,
  S_x)
 \]
 is transversal.
 Let
\[
\cV_x = S_x^{-1}(0) \cap  \big( U^{(1)}_x\times K_x\big),
\]
 then $\cV_x$ is a smooth finite dimensional manifold of dimension given by $\ind (D_x S) + \dim K_x$. There is a trivial  vector bundle $\bE_x = \cV_x \times K_x \to \cV_x$ with a canonical section $\sigma_x$  given by the component  of
$K_x$ in $\cV_x$.
There is a homeomorphism
\[
\psi_x:  \sigma_x^{-1} (0) \longrightarrow M\cap  U^{(1)}_x.
\]
So the local stabilization of of $ (\cB, \cE, S)$ at $x\in M$ provides  a local  virtual neighbourhood
\ba\label{vir:x}
(\cV_x , \bE_x, \sigma_x, \psi_x)
\na
of $x \in M$  in the sense of  Definition \ref{l.v.n}.

\begin{remark}
We  point out that the   virtual neighbourhood
(\ref{vir:x}) provides a Kuranishi chart of $M$ at $x$ as in the sense of  \cite{FO99}.  Our construction of this
particular Kuranishi chart of $M$ at $x$ is different to the construction in \cite{FO99}.
\end{remark}

\subsubsection{Global stabilizations}\label{global:stab}

To patch together local  virtual neighbourhoods at every point $x\in M$ to get a virtual system on $M$, we need to introduce a notion of global stablization  as in \cite{ChenTian} for $M$.

As  $M$ is compact, there exist  finite points, say
 $x_i, 1\leq i\leq n,$ such that the collection $\{U_{x_i}^{(1)}\}$ in Assumption \ref{assump}  is an open cover of
$M$, and the local stabilizations at  $\{x_i\}$ provide a collection of  virtual neighbourhoods
\[
\{(\cV_{x_i}, \bE_{x_i}, \sigma_{x_i}, \psi_{x_i})\}.
\]

 Set
$$
X^{(1)}=\bigcup_{i=1}^n U_{x_i}^{(1)},\;\;\;
X^{(3)}=\bigcup_{i=1}^n U_{x_i}^{(3)}.
$$

  For any $I\subset \{1, 2,  \ldots,n\}$,  define
\[
X'_I= \left( \bigcap_{i\in I} U_{x_i}^{(3)}\right) \setminus
\left(\bigcup_{j\notin I}\overline{U^{(2)}_{x_i}}\right).
\]
Then $\beta_{x_j} =0$ on $X'_I$ for any $j\notin I$.  Note that in this case, $X_\emptyset' \cap M=\emptyset$, hence,  $X_\emptyset'$ does not play any role in the global stablization.

\begin{lemma}\label{cover:X_I}
$\{X'_I\}_{I\subset\{1, 2, \ldots,n\}}$ is an open covering of $X^{(3)}$.
\end{lemma}
\begin{proof}
For any $y\in X^{(3)}$ let
$I$ be the collection of $i\in\{1, 2, \ldots,n\}$ such that
$y\in U_{x_i}^{(3)}$. Then $y\in \bigcap_{i\in I}U_{x_i}^{(3)}$.
On the other hand, $y$ is not in $U_{x_j}^{(3)}, j\notin I$,
hence,  $y\notin \overline{U_{x_j}^{(2)}}$. Hence by the definition
of $X_I'$, we conclude that $y\in X_I'$.
\end{proof}

Define
$$
X_I=X'_I\cap X^{(1)}.
$$
As a corollary  of Lemma \ref{cover:X_I}, we have
\[
X^{(1)}=\bigcup_{I\subset\{1, 2, \ldots,n\}} X_I.
\]
 Set
$
X_{I,J}=X_{J,I}=X_I\cap X_J, $
then $\{X_I\}$ gives a virtual manifold structure for
$X^{(1)}$ (Cf.  Example  \ref{example:vir}  (2)). We remark that $X_\emptyset =\emptyset$. 

For any $I\subset \{1, 2, \ldots,n\}$,  setting  $K_I=\prod_{i\in I}K_{x_i}$,
we construct a  thickened Fredholm  system as follows.
Let
\[
(C_I, F_I, S_I)
\]
be a Fredholm system where
$
C_I=X_I\times K_I$ and $
F_I=\cE|_{X_I}\times K_I \to C_I
$  is a Banach bundle,
and   the  section $S_I: C_I\to F_I$ is given  by
$$
S_I(y, (k_i)_{i\in I})
=S(y)+\sum_{i\in I}\beta_{x_i}(y)\phi_{x_i;y}^{-1} (k_i)
\in \cE_y.
$$
Then $(C_I, F_I, S_I)$ is a Fredholm system.

\begin{proposition} \label{global:mfd} \begin{enumerate}
\item  The collection  $\cC=\{C_I\} $ forms
 an  infinite dimensional   virtual Banach manifold   by setting
 \begin{enumerate}
 \item
 $
C_{I,J}=X_{I,J}\times K_I
$ for any pair $I,J$, and
\item
 $
\Phi_{I, J}:
C_{J,I}=C_{I,J}\times \prod_{i\in J\setminus I}K_{x_i}
\to
C_{I,J}
$  is a  vector bundle  for any pair $I\subset J$.
 \end{enumerate}
 \item  $\cF =\{F_I\} $ is  a  Banach bundle over $\cC$ in the sense of Definition \ref{vir:bundle}.
 \item $\cS= \{S_I\}$ is  a   {\em transversal  Fredholm}   section  of  $\cF$.
\end{enumerate}
This collection  $\{(\cC, \cF, \cS)\}$ is
a Fredhom system  of infinite dimensional virtual manifolds and is
called a {\bf global stabilization} of the original Fredholm system
$(\cB, \cE, S)$.
\end{proposition}

\begin{proof}  \begin{enumerate}
 \item This follows as  $\{X_I\}$ with $ X_{I,J}=X_{J,I}=X_I\cap X_J$ is an infinite dimensional   virtual Banach manifold, where the patching data are all given by the identity map. The identification
\[
C_{J,I}=C_{I,J}\times \prod_{i\in J\setminus I}K_{x_i}
\]
is obvious  from the definition of $K_I$'s.   The coherent conditions (\ref{cohe:1}) and (\ref{fiber:product}) hold on the nose.
So $\cC=\{C_I =X_I\times K_I \}$ forms
 an  infinite dimensional   virtual Banach manifold.

\item For   $I\subset J$, it is straightforward to check that
$F_J|_{C_{J, I}} =\Phi_{I, J}^* \big( F_I|_{C_{I, J}} \big)$, so
 \[
 \cF =\{F_I = \cE|_{X_I}\times K_I \to C_I \}
 \]
  is  a  Banach bundle over $\cC$.

  \item We can check by  a direct calculation
 that
 \[
 S_J|_{C_{J, I}} =  \Phi_{I, J}^*\big (   S_I|_{C_{I, J}}\big),
 \]
that is,  for $(x,(k_j)_{j\in J})\in C_{J,I}  = X_{J, I} \times K_J$, note that $X_{I, J}
=X_{J,I}= X_I\cap X_J$, we have
\ba\label{section=}\begin{array}{lll}
S_J(x,(k_j)_{j\in J})&
=&S(x)+\sum_{j\in J}\beta_{x_j}(x)\phi_{x_j;y}^{-1} (k_j)\\
&=&S(x)+\sum_{i\in I}\beta_{x_i}(x)\phi_{x_i;y}^{-1} (k_i) \\
& = & S_I \circ \Phi_{J,I}(x,(k_j)_{j\in J}). \end{array}
\na
Here we use the fact that $\beta_{x_j}(x)=0$ when
$j\notin I$ and $x\in X_{I,J}$.
  The Fredholm property of $\cS$ follows from the Fredholm property of $S$ in  the Fredholm system
$(\cB, \cE, \cS)$. The transversality of each $S_I$ is due to the choices of $K_I$'s and the fact that there exists 
$i\in I$ such that $\beta_{x_i}(x)\not= 0$ for 
  $x\in X_{I}$.
This completes the proof. \end{enumerate}
\end{proof}

Given a  global stabilization  $\{(\cC, \cF, \cS)\}$  of  a Fredholm system
$(\cB, \cE, S)$,  there is a canonical  virtual vector bundle $\cO =\{O_I\}$  over $\cC$, defined by
\[
O_I=C_I\times K_I,
\]
with a canonical section $\sigma =\{\sigma_I\}: \cC \to \cO$ given by
\[
 \sigma_I (x,(k_i)_{i\in I}) = (x, (k_i)_{i\in I}, (k_i)_{i\in I})
 \]
  for $(x,(k_i)_{i\in I}) \in C_I =X_I \times K_I$.
    Define
\ba\label{cV_I}
\cV_I=S_I^{-1}(0) \subset C_I, \;\;\;
\cV_{I,J}=S_I^{-1}(0)\cap C_{I,J}.
\na
Being a transversal  Fredholm section of the vector bundle $\cF$ over a virtual manifold  $\cC$, its zero set
\ba\label{zero:set}
\{\cV_I = S_I^{-1}(0)  \}
\na
is a collection of finite dimensional smooth manifolds (cf. (4) in Example (2.4)). Note that  from the identity (\ref{section=}) we have
\[
\cV_{J, I} = \cV_{I, J} \times \prod_{j\in J\setminus I}K_{x_i}
\]
We still denote the bundle projection map $\cV_{J,I}\to
\cV_{I,J}$ by $\Phi_{I, J}$. Then
$\cV=\{\cV_I\}$ is a virtual manifold.

By restricting the  virtual  bundle $\cO$
and the section $\sigma$ to  $\cV$, we get a virtual
bundle
\ba\label{res:bE}
\bE = \{\bE_I  =   O|_{\cV_I}\}
\na
 over  $\cV$ and a section $\sigma  =\{\sigma_I\}$.  For a given $I$,  there is a canonical inclusion
 \[
\psi_I:    \sigma_I^{-1}(0) \cong  \{ x\in X_I| S_I(x) =0 \} \longrightarrow  X_I\cap M,
  \]
and  $M =\bigcup_I  \psi_I \big( \sigma_I^{-1}(0) \big) $.

In summary, we obtain   the following finite dimensional  virtual  system  from  a Fredholm system
$(\cB, \cE, S)$.

\begin{theorem} \label{thm:vir:system} Given a  Fredholm system
$(\cB, \cE, S)$ with a global stabilization   $\{(\cC, \cF, \cS)\}$  of $M =S^{-1}(0)$.  Then
 the collection of   triples
\[
\{( \cV_I,  \bE_I,  \sigma_I ) |   I\subset \{1, 2,  \ldots,n\} \}
\]
defines a virtual system in the sense of Definition \ref{virtual-system}.
Moreover, there exists a collection of
inclusions
\[
\{ \psi_I:  \sigma_I^{-1}(0)    \longrightarrow    M |  I\subset \{1, 2,  \ldots,n\} \}
\]
such that $M = \bigcup_I  \psi_I \big( \sigma_I^{-1}(0) \big) $.
\end{theorem}


  \subsection{Integration   and invariants for virtual systems}\label{i:i:v:s}

In Section \ref{2}, the integration on an  oriented virtual manifold (a special case of an  oriented proper \'etale  virtual groupoid) is defined under the assumption of the existence of  a partition of  unity and a transition Thom form. In this subsection, we show that
a partition of  unity and a transition Thom form naturally  exist for the   finite dimensional  virtual  system
\[
\{( \cV_I,  \bE_I,  \sigma_I ) |   I\subset \{1, 2,  \ldots,n\} \}\]
  from  applying the global stabilization of  a Fredholm system
$(\cB, \cE, S)$.

\subsubsection{Partition of unity}
In this subsection, we will construct a partition of unity for a virtual system
$\{(\cV_I,  \bE_I, \sigma_I)\}$ in Theorem \ref{thm:vir:system}.

For any $z\in M \subset \bigcup_I X_I \subset \cB  $,  there exists an $I_z$ (which will be fixed)
such that $z\in U_z^{(2)}$, an open neighbourhood of $z$ in  $X_{I_z}$.
By Assumption \ref{assump},
there exists a  {\bf smooth}  cut-off function
$\eta_z'$ supported in $U_z^{(2)}\subset X_{I_z}$  and
$ \eta'_z\equiv 1$ in an open neighborhood $U^{(1)}_{z}$ of $z$. Hence,
$\eta'_z \in C^\infty(X_{I_z})$.

 Since $M$ is compact,  there exist finitely many those 
points $z_k,1\leq k\leq m,
$ with $\eta'_{z_k} \in C^\infty(X_{I_{z_k}})$,  such that
$$
M\subset \cU=\bigcup_{k=1}^m U^{(1)}_{z_k}.
$$
On $\cU$, since $\sum_{k=1}^m \eta'_{z_k}\not=0$, we 
define
\begin{equation}
\eta_{z_k}=\frac{\eta'_{z_k}}{\sum_{l=1}^m \eta'_{z_l}}.
\end{equation}
 Then
\ba\label{POU}
\sum_{k=1}^m\eta_{z_k}(y)=1,\;\;\; y\in \cU.
\na

Note  that  $\eta'_{z_k} \in C^\infty (X_{I_k})$ for some $I_k = I_{z_k}$, then
$\eta_{z_k}$ is a function on $\cU\cap X_{I_k}$.
By composing with the projection onto $\cU\cap X_{I_k}$, we get    a function on
$$
C'_{I_k}:=(\cU\cap X_{I_k})\times K_{I_k}\subset C_{I_k} = X_{I_k} \times K_{I_k},
$$
which will  still  be denoted   by $\eta_{z_k}$. It is clear that $\eta_{z_k}$ is an  invariant
  function on $C_{I_k}'$  under the equivalence relation (\ref{equi:rel}), that is, for any $J\subset I_k$,  the cut-off
  function
$\eta_{z_k}$ is constant  along the fiber of  the bundle
\[
C_{I_k,J}'\longrightarrow  C_{J,I_k}'.
\]

\begin{lemma}   \label{POU:C'}For any $I\subset \{1, 2, \cdots, n\}$, define
$
\eta_I=\disp{\sum_{k\in \{k|I_k=I\}}} \eta_{z_k} \in C^\infty(C_I'), then
$
$\{\eta_I\}$
 forms a partition of
unity  on the virtual manifold  $\cC':=\{C_I'\}$.
\end{lemma}
\begin{proof} Each $\eta_I $ is  an  invariant
  function on $C_{I_k}'$  under the equivalence relation (\ref{equi:rel})  as each  $ \eta_{z_k} \in C^\infty(C_I')$
  satisfies this property.   Now for any point $ x\in |\{C_I'\}|$, we need to check that
  \ba\label{unity}
  \sum_I \bar \eta_I (x) =1,
  \na
  where $\bar \eta_I$ is the induced function on $\pi_{I_0} (C_I') \subset  |\{C_I'\}|$.

   Choose  $(y,k)\in C_{I_0}'=(\cU\cap X_{I_0})\times K_{I_0}$ such that $\pi_{I_0}(y, k) = x$. We write
      \[
   \{
   J| y\in X_J\} = \{J_\ell | \ell=1, 2 \cdots, p\} \subset \{I_1, I_2, \cdots, I_m\}.
   \]
Then among $\{ J_1, J_2 \cdots,  J_p \}$
there exists
a smallest element, say $J_1$, in the sense that
$J_1\subset J_\ell, $ for any $\ell$. This follows from the
fact that  for any pair $X_I$ and
$X_J$ containing $y$,  then $ y\in X_{I\cap J}$.
Note that in \cite{ChenTian},
 $X_{J_1}$
is called the support of $y$.

Setting  $(y,k')=\Phi_{J_1, I_0}(y,k)$ under the bundle map $\Phi_{J_1, I_0}:  C'_{I_0, J_1} \to C'_{ J_1, I_0}$, then the  condition
(\ref{unity}) is equivalent to
\ba\label{sum=1}
\sum_{\ell=1}^p \bar \eta_{J_\ell}\big([\Phi_{J_1, J_\ell}^{-1}(y,k')]\big)=1.
\na
In this  expression, $[\Phi_{J_1, J_\ell}^{-1}(y,k')]$ denotes the equivalence class in $\pi_{J_\ell} (C'_{ J_\ell}) \subset   |\{C_I'\}|$. This is exactly
$x \in \pi_{J_\ell} (C'_{ J_\ell}) \subset   |\{C_I'\}|$, so $ \bar \eta_{J_\ell} (x) = \eta_{z_k} (y)$ where
$z_k$ is determined by the relation
\[
\{J_\ell | \ell=1, 2 \cdots, p\} \subset \{I_1, I_2, \cdots, I_k, \cdots,  I_m\}.
 \]
 Hence  (\ref{sum=1}) follows directly  from  (\ref{POU}).
 \end{proof}

Here,  we do not require that
the support of $\eta_I$ is compact  as  in  Definition \ref{POU:vir}, since  we can find a  well-controlled   twisted  form $\theta =\{\theta_I\}$ such that
$\eta_I\theta_I$ is compactly supported.

Let $\cV'_I=\cV_I\cap C'_I$.  Then $\{\cV'\}$ is a virtual submanifold of $\{\cV_I\}$.

\begin{corollary}
$\{\eta_I\} $ forms a partition of
unity of $\cV':=\{\cV_I'\}$.
\end{corollary}

\subsubsection{Transition Thom forms and  $\Theta$-twisted  forms}

In the local/global stabilizations, we have chosen a collection of finite points $\{x_i \in M | i=1, 2, \cdots, n\}$ and a finite dimensional
 $K_{x_i} \subset \cE_{x_i}$ for each $x_i$.   Let $\Theta_i,
1\leq i\leq n$ be a volume form on    $K_{x_i}$  that is
supported in a small  $\epsilon$-ball  $B_{\epsilon}$ of the origin   in $K_{x_i}$.
Then the volume form $
 \bigwedge_{i\in I}\Theta_i
$ on $K_I =\oplus_{i\in I} K_{x_i}$
defines a Thom form $\Theta_I$ on the bundle $O_I = C_I  \times K_I\ \to C_I$.

Recall that $C_{J,I}=C_{I,J}\times \prod_{i\in J\setminus I}K_{x_i}$ for $I\subset J$, so the volume form
$
 \bigwedge_{i\in J\setminus I}\Theta_i
$
defines a  Thom form $\Theta_{I, J}$     of the  bundle
$\Phi_{I, J}:  C_{J,I} \to C_{I,J}$. Then
the collection of Thom forms $\{\Theta_{I, J} |I\subset J\} $ is a transition Thom  form in the sense of  Definition \ref{Thom:form}.

 Under the inclusions,
\[
\cV_I  = S_I^{-1} (0) \subset C_I  \subset O_I
\]
where the last inclusion is given by the zero section of $O_I$, we can restrict the Thom form $\Theta_I$
to $\cV_I$ to get an Euler  form  for  the bundle $\bE_I= O_I |_{\cV_I}$. Denote this Euler  form of $\bE_I$ by
  $\theta_I$.
  \begin{lemma} \label{3.11}   \begin{enumerate}
\item   $\Theta = \{ \Theta_{I, J}|_{ \cV_{J,I}} |I\subset J\} $ is a the transition Thom form for $\{\cV_I\}$
 \item  $\theta=
\{\theta_I  \}  $
  is a $\Theta$-twisted virtual  differential  form    on $\{\cV_I\}$ with respect to the transition Thom form
$\Theta =\{ \Theta_{I, J}|_{ \cV_{J,I}} |I\subset J\} $. This twisted virtual differential  form  will be called a {\bf virtual Euler form} of the virtual vector bundle
$\{\bE_I\}$.
\end{enumerate}
\end{lemma}
 \begin{proof}
 The proof is just a  straightforward calculation.
 \end{proof}

The following assumption is crucial  in finishing off  the virtual neighborhood technique for a Fredholm system.

\begin{assumption}\label{assump:2}[virtual convergence]
Given  a sequence
$\{(y_N,k_N) |N=1, 2, \cdots,  \}$  in $C_I=X_I\times K_I$ such that
$S_I(y_N,k_N)=0$ and $k_N\to 0$, then  the sequence $\{y_N\}$ converges in
$X_I$.
\end{assumption}

\begin{theorem}\label{cpt:sup}
Under the Assumption \ref{assump:2},
  $\epsilon$  in the $\epsilon$-balls  in $ K_{x_i}$'s can be chosen sufficiently  small  so that  the support of $\theta_I$
is  contained in $\cV'_I = \cV_I \cap Y_I'$. In particular, the support of $\Theta_i,1\leq i\leq n, $ can be chosen
sufficiently small so  that
 $\eta_I\theta_I$ is compactly supported in $\cV_I'\subset \cV_I = S_I^{-1}(0)\subset C_I$.

\end{theorem}
\begin{proof}
For simplicity, suppose $I=\{1,\ldots,
\ell\}, \ell\leq n$. If the claim of the theorem is
not true, then for any $\epsilon=1/N$,
there exists a point
$(y_N, k_{N})$
that solves the equation $S_I(y_N,k_N)=0$,   $\|k_N\| < 1/N$  and $y_N\notin (X_I\cap \cU)
\subset \cU$.

As $N\to\infty$, $k_{N}\to 0$,
by the Assumption \ref{assump:2},
$y_N$ converges to $y_\infty\in \cB$ and
$ S(y_\infty) =S_I(y_\infty, 0) = 0$. Therefore,  $y_\infty\in M \cap \cU \subset X_I\cap \cU$. This contradicts
to the fact that $y_N\not\in  X_I\cap\cU$.
\end{proof}

 Let $\{( \cV_I,  \bE_I,  \sigma_I ) |   I\subset \{1, 2,  \ldots,n\} \}$ be the virtual system obtained in Theorem \ref{thm:vir:system}.  We assume that $\{( \cV_I,  \bE_I,  \sigma_I ) \} $ is {\bf oriented} in the sense that both
 $\{\cV_I\}$ and  $\{\bE_I\} $ are  oriented.  Let  $\eta=\{\eta_I\}$
be a partition of  unity and $\theta=\{\theta_I\}$  be  a  $\Theta$-twisted differential form  constructed in Lemma \ref{3.11}.

Given  $\alpha =\{\alpha_I\}$   a  degree $d$ virtual  differential  form on $ \cV=\{\cV_I\}$, that is, a smooth section of
$\bigwedge^d (T^*\cV)$ such that under the bundle map $\Phi_{I, J}: \cV_{J, I}\to \cV_{I, J}$ for $I\subset J$, we have
\[
\alpha_J|_{\cV_{J, I} }= \Phi^*_{I, J} \big( \alpha_I|_{\cV_{I, J}}\big).
\]

 \begin{definition}  Define a virtual
integration of $\alpha$ to be
\ba\label{def:integration}
\int_{\cV}^{vir} \alpha
=\sum_{I}\int_{\cV_I} \eta_I \cdot \theta_I \cdot \alpha_I.
\na
\end{definition}

Now given a  cohomology class  $ H^d(\cB)$, suppose that under the following smooth maps
\[
\cV_I \subset C_I = X_I \times K_I \longrightarrow X_I \subset \cB,
\]
for   each $ I\subset \{1, 2, \cdots, n\}$,
the pull-back cohomology class   can be represented by a
closed degree  $d$  virtual differential  form $\alpha_I$.   Then we can
define an   invariant   associated to the Fredholm system  $(\cB, \cE, S)$ to be
\ba\label{invariant:Fred-triple}
\Phi(\alpha)=\int_{\cV}^{vir}\alpha.
\na

\begin{remark}\label{vir:tech}  Given a proper  Fredholm system $(\cB, \cE, S)$ so that $M=S^{-1}(0) $ is compact, the virtual neighborhood technique for this system entails
\begin{enumerate}
\item a local stabilization  at $x\in M= S^{-1}(0)$: a choice of  a finite dimensional subspace $K_x\subset \cE_x $ and a choice of nested
triple  $U_{x}^{(1)}\subset  U_{x}^{(2)}\subset U_{x}^{(3)} $ of $x$ in $\cB$  with  a cut-off function $\beta_x$ as in Assumption \ref{assump}
such that the resulting thickened local   Fredholm system
\[
  (U^{(3)}_x\times K_x, \cE|_{U^{(3)}_x}\times K_x,
  S_x)
\]
  is a transversal on $U_{x}^{(1)}\times K_x$.   This provides a canonical virtual neighbourhood
\[(\cV_x , \bE_x, \sigma_x, \psi_x)
\]
of $x \in M$  in the sense of  Definition \ref{l.v.n}.  Here $\psi_x$ is given by the  inclusion $\sigma_x^{-1}(0) \subset M$.
\item a  global stabilization:  a process  to assemble the collection of
thickened local   Fredholm system
\[
  (U^{(3)}_{x_i}\times K_{x_i}, \cE|_{U^{(3)}_{x_i}}\times K_{x_i},
  S_{x_i})
\]
at finitely many points $\{x_i | i=1, 2, \cdots, n,  \bigcup_{i=1}^n  U_{x_i}^{(1)}  \supset M\}$  into a  virtual Fredholm system $\{ (\cC, \cF, \cS)\}$ where
\begin{enumerate}
\item $\cC=\{\cC_I \}$ is a collection of  virtual Banach manifolds,
\item $\cF =\{\cF_I \to \cC_I\}$ is a collection of Banach bundles.
\item $\cS= \{S_I\}$ is  a collection of  {\em transversal}   Fredholm   sections   of  $\cF$.

\end{enumerate}

The consequence  of a  global stabilization is a finite dimensional virtual system  for $(\cB, \cE, S)$, that is a collection of triples
 \[
  \{(\cV_I, \bE_I, \sigma_I)|  I\subset \{1, 2, \cdots, n\} \}
  \]
  indexed by a   partially ordered
set $(\cI =2^{\{1, 2, \cdots, n\}}, \subset )$, where
  \begin{enumerate}
\item $\cV=\{\cV_I  \}$ is a finite dimensional virtual manifold,
\item $\bE=\{\bE_I\}$ is a  finite rank virtual vector bundle over $\{\cV_I\}$
\item $\sigma=\{\sigma_I\}$ is a virtual section  of the  virtual vector bundle $\{\bE_I\}$ whose zeros
$ \sigma^{-1}(0)  $ form a cover of $M$, that is, $M = \bigcup_I  \sigma^{-1}(0)$.
 \end{enumerate}

\item   under   Assumption   \ref{assump:2}, there is  a choice of a partition of unity  $\eta=\{\eta_I\}$   and  a  virtual  Euler form $\theta=\{\theta_I\}$ of $\{\bE_I\}$ such that
$\eta_I\theta_I$ is compactly supported in $\cV_I$. Therefore, the virtual integration
\[
\int^{vir}_\cV \alpha   = \sum_I \int_{\cV_I}\eta_I\theta_I\alpha
\]
 is well-defined for any virtual differential form
  $\alpha =\{\alpha_I\}$ on $\cV$ when both  $\cV$  and $\bE_I$ are assumed to be oriented.
\end{enumerate}
\end{remark}

  \subsection{Fredholm  orbifold systems   and their virtual  orbifold systems}\label{3.3}

   In this subsection, we generalise the virtual neighborhood technique (Cf. Remark \ref{vir:tech} ) to  the orbifold case using the language of proper
   \'etale groupoids.

  By  an orbifold Fredholm system $(\cB, \cE, S)$, we mean that
  \begin{enumerate}
\item $\cB$  is  a proper  \'etale groupoid $\cB = (\cB^1 \rightrightarrows \cB^0)$, that is,
  $|\cB|$ is a Banach orbifold,
\item  $ \cE =  (\cE^1 \rightrightarrows \cE^0) \to \cB$ is a Banach vector bundle in the sense of Proposition \ref{bundle:gpoid},
\item $S$ is a Fredholm section of $\cE$ given by a pair sections $(S_1, S_0)$ such that  the following diagram commutes
\ba\label{diag:0}
\xymatrix{
\cE^1 \ar@<.5ex>[d]\ar@<-.5ex>[d] \ar[rr]_{\pi_1} && \cB^1\ar@<.5ex>[d]\ar@<-.5ex>[d]  \ar@/_1pc/[ll]_{S_1}\\
\cE^0 \ar[rr]_{\pi_0}&& \cB^0  \ar@/_1pc/[ll]_{S_0} .}
\na
Here $S$ is Fredholm if $S_0$ is Fredholm.
We say that  an orbifold Fredholm system  $(\cB, \cE, S)$ is a transversal if $S_0$ is transversal to the zero section of $\cE_0\to \cB_0$.
\end{enumerate}

  \begin{lemma}  Given a transversal  orbifold Fredholm system $(\cB, \cE, S)$, the zero  set
  \[
  \cM = S^{-1} (0) = \left(   S_1^{-1}(0)  \rightrightarrows  S_0^{-1}(0) \right )
  \]
  with the induced groupoid structure from $\cB$ is a smooth orbifold.
  \end{lemma}
\begin{proof}
As $S_0$ is  a transversal section,
$S^{-1}_0(0)$ is  a finite dimensional smooth manifold. Using the  \'etale properties
  of $\cB$  and $\cE$,  $S_1$ is
also a transversal section. Hence $S_1^{-1}(0)$ is  a finite dimensional smooth manifold. The diagram (\ref{diag:0})  implies that
  $\cM$ is  a proper \'etale groupoid.
\end{proof}

Consider  a general orbifold Fredholm system $(\cB, \cE, S)$ when $S$  is  not  transversal to the zero section. The virtual neighborhood technique can still be developed for this orbifold  Fredholm system  to get  a finite dimensional virtual orbifold system.

\begin{definition}  \label{virtual-sys:orb} ({\bf Finite dimensional virtual orbifold system}) A finite dimensional virtual orbifold system
is  a collection of triples
 \[
  \{(\cV_I, \bE_I, \sigma_I)|  I\subset \{1, 2, \cdots, n\} \}
  \]
  indexed by a a partially ordered
set $(\cI =2^{\{1, 2, \cdots, n\}}, \subset )$, where
  \begin{enumerate}
\item $\cV=\{\cV_I  \}$ is a finite dimensional  proper  \'etale virtual groupoid,
\item $\bE=\{\bE_I\}$ is a  finite rank virtual vector bundle over $\{\cV_I\}$
\item $\sigma=\{\sigma_I\}$ is a virtual section  of the  virtual vector bundle $\{\bE_I\}$.
 \end{enumerate}
\end{definition}

\subsubsection{Local stabilizations}

  Let $\pi:  \cM\to |\cM|$ be the quotient map.
  Given $x\in |\cM|$, let $\tilde x\in  \pi^{-1}(x)$ with its isotropy group $G_{\tilde x} = (s, t)^{-1}(\tilde x, \tilde x)$. Then there
  exists   a small neighbourhood $U_x$ of $x$ in $|\cB|$, and a $G_{\tilde x} $-invariant neighbourhood  $U_{\tilde x}$ of
  $\tilde x$ in $\cB^0$ such that  the triple
  \[
  (U_{\tilde x}, G_{\tilde x}, U_x)
  \]
  is an orbifold chart at  $x \in |\cB|$ and $\cE|_{ U_{\tilde x}}$ has a $G_{\tilde x}$-equivariant trivialisation
\[
\phi_{\tilde x}:  \cE^0|_{U_{\tilde x}} \longrightarrow U_{\tilde x} \times F_{\tilde x}
\]
where $ F_{\tilde x}$ is the fiber of $\cE$ at $\tilde x$.  Let $\phi_{\tilde x, y}: \cE^0_y \to F_{\tilde x} $ be the
induced isomorphism for $y\in U_{\tilde x}$.
Then the local stabilization in \S \ref{local:stab} can be carried over to a $G_{\tilde x}$-invariant
 local stabilization  of   $(\cB, \cE, S)$ at $\tilde x$ under the following assumption.

\begin{assumption}\label{assump:orbi}
For any $\tilde x \in \pi^{-1}(x)$, there are  open  $G_{\tilde x}$-invariant neighborhoods $U_{\tilde x}^{(i)}$ of $\tilde x$ for $i=1,2,3$ such that
\ba\label{nest:1}
U_{\tilde x}^{(1)}\subset \overline{U_{\tilde x}^{(1)}}
\subset U_{\tilde x}^{(2)}\subset \overline{U_{\tilde x}^{(2)}}
\subset U_{\tilde x}^{(3)}\subset \overline{U_{\tilde x}^{(3)}}\subset U_{\tilde x}
\na
and
 a {\bf  smooth} $G_{\tilde x}$-invariant  cut-off function
$\beta_{\tilde x}:  U_{\tilde x}^{(3)} \to [0, 1]$ such that
$\beta_{\tilde x}\equiv1$ on  $U_{\tilde x}^{(1)}$ and
is supported in
$U_{\tilde x}^{(2)}$.  Here $\overline{U_{\tilde x}^{(i)}}$ is the closure of $U_{\tilde x}^{(i)}$.
\end{assumption}

We choose these nested open   $G_{\tilde x}$-invariant neighborhoods
\[
U_{\tilde x}^{(1)}\subset U_{\tilde x}^{(2)}\subset  U_{\tilde x}^{(3)}\subset U_{\tilde x}
\]
such that under the quotient map $\pi$, we get
\ba\label{nest:2}
U_{x}^{(1)}\subset U_{x}^{(2)}\subset  U_{ x}^{(3)}\subset U_{ x}
\na
which are {\bf independent}  of the choices of $\tilde x \in \pi^{-1}(x)$.   Under  Assumption \ref{assump:orbi},  the cut-off functions  $\{\beta_{\tilde x}| \tilde x \in \pi^{-1}(x)\}$ can be  chosen to be  invariant under the action of $\cB^1$.

The $G_{\tilde x}$-invariant
 local stabilization  of   $(\cB, \cE, S)$ at $\tilde x$ is given by a thickened  $G_{\tilde x}$-invariant
 Fredholm system
 \[
 ( U_{\tilde x}^{(3)}\times K_{\tilde x}, \cE|_{ U_{\tilde x}^{(3)}} \times K_{\tilde x}, S_{\tilde x, 0})
 \]
 with $S_{\tilde x, 0} (y, k) = S_0(y)  +\beta_{\tilde x} (y) \psi^{-1}_{\tilde x, y} (k)$.
 Here $K_{\tilde x}$
 is an even  dimensional
 $G_{\tilde x}$-invariant  linear subspace of $F_{\tilde x}$
  such that $S_{\tilde x, 0}$ is transversal to the
  zero section on $U_{\tilde x}^{(1)}\times K_{\tilde x}$.
   Such a $K_{\tilde x}$ can be found by setting
   $K_{\tilde x}=\sum_{g\in G_{\tilde x}}g\cdot K$ for some
   $K$
  defined  by the Fredholm property of $S_0: \cB^0\to\cE^0$ as in    (\ref{surj:S}).
 We further require that
 $\{K_{\tilde x}\}_{\tilde x \in \pi^{-1} (x)}$ is invariant under the action of $\cB^1$.

Denote the $G_{\tilde x}$-invariant zero set  by
\[
\cV_{\tilde x} = S_{\tilde x, 0}^{-1}(0) \cap \big( U_{\tilde x}^{(1)}\times K_{\tilde x}\big),
\]
then the $G_{\tilde x}$-invariant
 local stabilization  of   $(\cB, \cE, S)$ at $\tilde x$ provides a
 $G_{\tilde x}$-invariant local virtual neighbourhood
 \[
 (\cV_{\tilde x} , \bE_{\tilde x} = \cV_{\tilde x} \times K_{\tilde x}, \sigma_{\tilde x}, \psi_{\tilde x})
 \]
 where $\bE_{\tilde x} \to \cV_{\tilde x}$ is a $G_{\tilde x}$-equivariant vector bundle over $\cV_{\tilde x}$.

Due to the $G_{\tilde x}$-invariance, this local    stabilization  of   $(\cB, \cE, S)$ at $\tilde x$  can be
extended to a $\cB^1$-invariant  local stabilization  of   $(\cB, \cE, S)$ at $x\in |\cM|$ in the following sense.
    We can choose $U_x$ small enough such that
    \[
    \{ U_{\tilde x}|\tilde x \in \pi^{-1}(x)\} \]
     are all disjoint. Denote by
   \[
  \iota_x:  \wt U^0_x  = \bigsqcup_{\tilde x\in \pi^{-1}(x)}  U_{\tilde x} \subset \cB^0,
   \]
 the obvious inclusion map,   then $\wt U^0_x$ is $\cB^1$-invariant. Let
 \[
   \wt U^1_x  = \cB\times_{s, \iota_x}  \wt U^0_x =\{ (\gamma, y)|  y \in \wt U^0_x, \gamma \in \cB, s(\gamma) = y\}
   \]
   Then  $ \wt \bU_x = (\wt U^1_x   \rightrightarrows \wt U^0_x)$,  with the source and target maps given by
   $(\gamma, y) \mapsto y$ and $(\gamma, y) \mapsto \gamma \cdot y$, is a sub-groupoid of $\cB$. In fact, this
   groupoid is Morita equivalent to the action groupoid
    $G_{\tilde x}\ltimes U_{\tilde x}  \rightrightarrows U_{\tilde x}$.  So both groupoids define the same
    orbifold structure on $U_x$.  We will call $ \wt \bU_x = (\wt U^1_x   \rightrightarrows \wt U^0_x)$
    the   sub-groupoid   of $\cB$  generated  by  the  $\cB^1$-action on $ \wt U^0_x  $.   Similarly, given a $\cB$-manifold $X$, then the action groupoid $(\cB\ltimes X  \rightrightarrows X)$ is called the
    groupoid  generated by  the $\cB^1$-action  on $X$. As $\cB$ is proper  and \'etale, the groupoid
    $(\cB\ltimes X  \rightrightarrows X)$ is also proper  and \'etale.

Applying  the local    stabilization to
\[
(\wt \bU_x,  \cE|_{\wt \bU_x}, S),
\]
we get the following local stabilization theorem for an orbifold Fredholm system.

\begin{theorem} \label{thm:orbi-local-stab} Let $x\in |\cM|$, under   Assumption \ref{assump:orbi}  for any $\tilde x\in \pi^{-1} (x) $,
we get a local   orbifold Fredholm system
\[
( \cK_x, \wtE_x, S_x)
\]
consisting of
\begin{enumerate}
\item  $\cK_x$ is the groupoid  generated by the  $\cB^1$-action on
$
\bigsqcup_{\tilde x\in \pi^{-1} (x)}
\big( U_{\tilde x}^{(3)}\times K_{\tilde x} \big),
$
\item $\wtE_x$ is the Banach  bundle over $ \cK_x$  generated by the  $\cB^1$-action on $\bigsqcup_{\tilde x\in \pi^{-1} (x)}   \big(\cE|_{U_{\tilde x}^{(3)}}\times K_{\tilde x}\big)$,
\item $S_x$ is a section of $\wtE$ defined by $\bigsqcup_{\tilde x\in \pi^{-1} (x)}   S_{\tilde x, 0}$.
\end{enumerate}

Define  $\cV_x $ to be the  groupoid  generated by the  $\cB^1$-action on
\[
 \bigsqcup_{\tilde x\in \pi^{-1} (x)} \cV^0_{\tilde x} = \bigsqcup_{\tilde x\in \pi^{-1} (x)}  \left( S_{\tilde x, 0}^{-1}(0) \cap \big( U_{\tilde x}^{(1)}\times K_{\tilde x}\big)\right).
 \]
  Let  $\bE_x$ be the restriction of  the vector bundle generated by
the  $\cB^1$-action on
\[
\bigsqcup_{\tilde x\in \pi^{-1} (x)} \left( \big(U_{\tilde x}^{(1)}\times K_{\tilde x}\big) \times K_{\tilde x} \right).
\]
Note that  $\bE_x$ has a canonical section $\sigma_x$. Then we get a finite dimensional  local  orbifold  system
(see Definition \ref{virtual-sys:orb})
 \[
 (\cV_x , \bE_x, \sigma_x)
 \]
of $(\cB, \cE, S)$ at $x\in |\cB|$.  Moreover, there is an inclusion
\[
\psi_x: | \sigma_x^{-1}  (0)| \longrightarrow |\cM|\cap U_x^{(1)},
\]
where $U_x^{(1)}$ is the quotient  space  $U_{\tilde x}^{(1)}/G_{\tilde x}$  for any $\tilde x\in \pi^{-1}(x)$.

 \end{theorem}

\begin{remark} \label{Morita}
 Note that the orbifold Fredholm system  $( \cK_x, \wtE_x, S_x)$ is Morita equivalent to  the
orbifold Fredholm system defined by the
$G_{\tilde x}$-invariant  Fredholm system
\[
( U_{\tilde x}^{(3)}\times K_{\tilde x}, \cE|_{ U_{\tilde x}^{(3)}} \times K_{\tilde x}, S_{\tilde x, 0})
\]
for any $\tilde x\in \pi^{-1}(x)$.  In fact, the latter should  be thought as a {\bf slice}  of the  $\cB$-action on  $( \cK_x, \wtE_x, S_x)$.
On the one hand, the advantage of the more  involved  orbifold Fredholm system is the convenience of considering
coordinate changes for different points in $|\cM|$. On the other hand, the integration and  a partition of unity
 are  defined in terms of these  slices, see (\ref{int:def}).
\end{remark}

 \subsubsection{Global stabilization} 

  Given an   orbifold Fredholm system $(\cB, \cE, S)$, assume that $|\cM | =|S^{-1}(0)$ is compact, then the local stabilization process  provides a finite collection of local transversal
   orbifold Fredholm systems
\[
 \{( \cK_{x_i}, \wtE_{x_i}, S_{x_i})|  i= 1, 2, \cdots, n\},
\]
 and its  corresponding  orbifold systems
 \[
 \{(\cV_{x_i} , \bE_{x_i}, \sigma_{x_i})|  i= 1, 2, \cdots, n\}
 \]
 such that the images of  $\{\psi_{x_i}: | \sigma_{x_i}^{-1} (0)| \to |\cM|\}_i$ form a open  cover of $|\cM|$.  

 \begin{theorem}\label{thm:orbi-vir-sys}
 There exists  a finite dimensional virtual  orbifold system  for $(\cB, \cE, S)$ which is   a collection of triples
 \[
  \{(\cV_I, \bE_I, \sigma_I)|  I\subset \{1, 2, \cdots, n\} \}
  \]
  indexed by a   partially ordered
set $(\cI =2^{\{1, 2, \cdots, n\}}, \subset )$, where
  \begin{enumerate}
\item $\{\cV_I |  I\subset \{1, 2, \cdots, n\}  \}$ is a finite dimensional  proper  \'etale virtual groupoid,
\item $\{\bE_I\}$ is a  finite rank virtual vector bundle over $\{\cV_I\}$
\item $\{\sigma_I\}$ is a section  of the  virtual vector bundle $\{\bE_I\}$ whose zeros $\{ \sigma_I^{-1} (0)\}$  form a cover of $\cM$.
 \end{enumerate}
\end{theorem}
\begin{proof}  The proof is parallel to that of Proposition \ref{global:mfd} and Theorem \ref{thm:vir:system},  with some extra care  to deal with orbifold structures using  the language of groupoids.  The proof that follows is presented   to spell out the difference to the proof of Theorem \ref{thm:vir:system}.

Under Assumption \ref{assump:orbi}, for each $x_i$, we have  a local stabilization
\[
 ( \cK_{x_i}, \wtE_{x_i}, S_{x_i})
\]
provided by Theorem \ref{thm:orbi-local-stab},  where  $\cK_{x_i}$ is a proper \'etale groupoid with its unit space (the space of objects) given by
 \[
 \bigsqcup_{\tilde x_i\in \pi^{-1} (x_i)}
\big( U_{\tilde x_i}^{(3)}\times K_{\tilde x_i} \big),
\]
 and  $\wtE_{x_i}$ is   a proper \'etale groupoid with its unit space  given by
 \[
 \bigsqcup_{\tilde x_i\in \pi^{-1} (x_i)}   \big(\cE|_{U_{\tilde x_i}^{(3)}}\times K_{\tilde x_i}\big).
 \]
Note that  $\wtE_{x_i}$ is   the Banach  bundle over $ \cK_{x_i}$   with a section  $S_{x_i}$   which is
transversal on
\[
 \bigsqcup_{\tilde x_i\in \pi^{-1} (x_i)}
\big( U_{\tilde x_i}^{(1)}\times K_{\tilde x_i} \big),
\]
The choices of $\{x_1, x_2, \cdots, x_n\}$ are  such that
\ba\label{good:choice}
\bigcup_{i=1}^n \left(  \bigsqcup_{\tilde x_i\in \pi^{-1} (x_i)}
  U_{\tilde x_i}^{(1)} \right) \supset \cM.
  \na
which implies $   |\cM| \subset \bigcup_{i=1}^n
  U_{ x_i}^{(1)} \subset |\cB|$.

Notice that,  from the choices of nested neighbourhoods in (\ref{nest:1}) and (\ref{nest:2}), we have the following inclusions of  proper \'etale groupoids
\[
\wt \bU_{x_i}^{(1)} \subset   \wt \bU_{x_i}^{(2)} \subset \wt \bU_{x_i}^{(3)}
\]
which are defined by the inclusions  of unit spaces
\[
\bigsqcup_{\tilde x_i\in \pi^{-1} (x_i)}
 U_{\tilde x_i}^{(1)}  \subset \bigsqcup_{\tilde x_i\in \pi^{-1} (x_i)}
  U_{\tilde x_i}^{(2)}   \subset \bigsqcup_{\tilde x_i\in \pi^{-1} (x_i)}
  U_{\tilde x_i}^{(3)}.
\]
Similarly we have
\[
\cK_{x_i}^{(1)} \subset \cK_{x_i}^{(2)}  \subset \cK_{x_i}
\]
which are defined by  the inclusions  of unit spaces
\[
\bigsqcup_{\tilde x_i\in \pi^{-1} (x_i)}
\big( U_{\tilde x_i}^{(1)}\times K_{\tilde x_i} \big) \subset \bigsqcup_{\tilde x_i\in \pi^{-1} (x_i)}
\big( U_{\tilde x_i}^{(2)}\times K_{\tilde x_i} \big)  \subset \bigsqcup_{\tilde x_i\in \pi^{-1} (x_i)}
\big( U_{\tilde x_i}^{(3)}\times K_{\tilde x_i} \big).
\]
Therefore,  we have the inclusions of local Fredholm orbifold systems
\[
( \cK^{(1)}_{x_i}, \wtE^{(1)}_{x_i}, S_{x_i}) \subset   (\cK^{(2)}_{x_i}, \wtE^{(2)}_{x_i}, S_{x_i}) \subset   ( \cK_{x_i}, \wtE_{x_i}, S_{x_i}).
\]

Define
\ba\label{A'_I}
\cA'_I = \left( \bigcap_{i\in I}  \wt \bU_{x_i}^{(3)}  \right) \setminus
\left(\bigcup_{j\notin I}\overline{\wt\bU^{(2)}_{x_i}}\right),
\na
where $\overline{\wt\bU^{(2)}_{x_i}}$ is the closure of the groupoid $\wt\bU^{(2)}_{x_i}$.  Then
$\cA'_I$ is the groupoid generated by the $\cB^1$-action on
\[
\pi^{-1} (   \left(\bigcap_{i\in I}  U_{x_i}^{(3)} \right)   \setminus
\left(\bigcup_{j\notin I}\overline{ U^{(2)}_{x_i}}\right) ).
\]

Set
\ba\label{A_I}
\cA_I = \cA'_I \cap \left(  \bigcup_i   \wt \bU_{x_i}^{(1)}\right).
\na
Let $\cA_{I, J} = \cA_{J, I} =\cA_I \cap \cA_J$ with $\Phi_{I, J} = Id$ for any $I\subset J$. Then
$ \{\cA_I\}$ covers $\cM$,
so $\{\cA_I\}$ provides a virtual orbifold structure for $ \bigcup_i   U_{x_i}^{(1)}$, a small
neighbourhood of $|\cM|$ in  $|\cB|$.  We remark that the collection of $\{\cA_I\}$ plays exactly the same role as
$\{X_I\}$ in Proposition \ref{global:mfd}.

Following the constructions in Proposition \ref{global:mfd}, we have a collection of  (thickened) {\bf transversal} Fredholm
orbifold systems
\[
\{(\cC = \{\cC_I\}, \cF = \{\cF_I \to \cC_I\}, \cS= \{S_I\})\}
\]
which is a global stabilization of the original  orbifold Fredholm system $(\cB, \cE, S)$. Here
$\cC_I$ is  the total space of a finite rank vector bundle over $\cA_I$.    Note that  $\cA_I^0$,  the unit space of
$\cA_I $,  is  a subset of
\ba\label{multi:cap}
 \bigsqcup_{\{ \tilde x_i\}_{i\in I} }  \left(  \bigcap_{\{\tilde x_i\}_{i\in I}   } U_{\tilde x_i}^{(3)} \right),
 \na
 where the disjoint union is taken over  the $|I|$-tuple of points
 \[
 \{ \tilde x_i\}_{i\in I} \in  \prod_{i\in I} \pi^{-1}(x_i)
 \]
  and  given a collection of points $\{ \tilde x_i\}_{i\in I} $ the intersection inside
 the bracket is taken over those $\tilde x_i$'s.
 Over the non-empty  component $ \bigcap_{\{\tilde x_i\}_{i\in I}   } U_{\tilde x_i}^{(3)}$  labelled by $(\tilde x_i)_{i\in I}$, the restriction of  the bundle $\cC_I$ is defined to be
 \[
\left( \bigcap_{\{\tilde x_i\}_{i\in I}   } U_{\tilde x_i}^{(3)} \right) \times K_{\{\tilde x_i\}_{i\in I} }
\]
with  $ K_{\{\tilde x_i\}_{i\in I} } = \prod_{\{\tilde x_i\}_{i\in I} }  K_{\tilde x_i}$ where the product is taken over
those $\tilde x_i$'s in the $|I|$-tuple  $\{\tilde x_i\}_{i\in I} $ with the non-empty  component $ \bigcap_{\{\tilde x_i\}_{i\in I}   } U_{\tilde x_i}^{(3)}$.  We remark that unlike  the constructions in
Theorem \ref{thm:vir:system}, here we have a groupoid action on  $ K_{\{\tilde x_i\}_{i\in I} }$,
  It is easy to verify that $\cC = \{\cC_I\}$ is an infinite dimensional virtual orbifold Banach manifold.

The Banach bundle $\cF_I  \to \cC_I$ can be constructed in a similar manner, that is, over
a non-empty  component $  \left(  \bigcap_{\{\tilde x_i\}_{i\in I}   } U_{\tilde x_i}^{(3)} \right)  \times K_{\{\tilde x_i\}_{i\in I} }$,
$\cF_I$ is given by
\[
\cE|_{\bigcap_{\{\tilde x_i\}_{i\in I}   } U_{\tilde x_i}^{(3)}} \times  K_{\{\tilde x_i\}_{i\in I} }.
\]
Over this component, the section $S_I$ is defined  by
\[
(y, (k_{\tilde x_i})_{\{\tilde x_i\}_{i\in I} }) \mapsto S (y) + \sum_{\{\tilde x_i\}_{i\in I} } \beta_{\tilde x_i} (y) \phi^{-1}_{\tilde x_i , y} (k_{\tilde x_i}).
\]
 The  transversality  of $\cS=\{S_I\}$ follows from the same arguments  as in the proof of Proposition \ref{global:mfd}.

  With this  understanding of  the collection of    transversal  Fredholm
orbifold systems
\[
\{(\cC = \{\cC_I\}, \cF = \{\cF_I \to \cC_I\}, \cS= \{S_I\})\},
\]
define $\cV_I = S_I^{-1}(0) \subset \cC_I$. Then $\cV_I$  is a finite dimensional proper  \'etale groupoid. There is a canonical
finite rank vector bundle  $\bE_I$  over $\cV_I$ whose fiber at
\[
(y,  (k_{\tilde x_i})_{\{\tilde x_i\}_{i\in I} }) \in S_I^{-1}(0) \cap  \left( \big( \bigcap_{\{\tilde x_i\}_{i\in I}   } U_{\tilde x_i}^{(3)}\big)  \times K_{\{\tilde x_i\}_{i\in I} } \right)
\]
is $ \prod_{\{\tilde x_i\}_{i\in I} }  K_{\tilde x_i}$.
  Then $\bE_I$ has  a canonical section $\sigma_I: (y,  (k_{\tilde x_i})_{\{\tilde x_i\}_{i\in I} }) \mapsto
  (y,  (k_{\tilde x_i})_{\{\tilde x_i\}_{i\in I} },   (k_{\tilde x_i})_{\{\tilde x_i\}_{i\in I} })$.

The rest of the proof is just an obvious adaptation of the proof of
Theorem \ref{thm:vir:system} to the proper \'etale groupoid   cases.
 \end{proof}

\subsection{Integration and invariants for virtual orbifold systems}\label{3.4}

The existence of a partition of unity for a virtual system arising from a Fredholm
system  in  \S \ref{i:i:v:s}   can be  also adapted with some minor changes  to get a partition of unity for the  virtual  orbifold
system
\[
 \{(\cV_I, \bE_I, \sigma_I)|  I\subset \{1, 2, \cdots, n\} \}
  \]
obtained in  Theorem \ref{thm:orbi-vir-sys}.

 For any $z\in |\cM| \subset \bigcup_I  | \cA_I |   \subset  |\cB|  $,  there exists an $I_z$ (which will be fixed)
such that $z\in U_z^{(2)}$, an open neighbourhood of $z$ in  $|\cA_{I_z}|$.
By Assumption \ref{assump:orbi},
there exists a  $\cB^1$-invariant,  smooth   cut-off function
$\eta_z' \in C^\infty(\cA_{I_z}) $ such that  the induced function $\bar\eta_z'$  on the orbit space $|\cA_I|$ is  supported in $U_z^{(2)}\subset |\cA_{I_z}|$  and
$ \eta'_z\equiv 1$ in an open neighborhood $U^{(1)}_{z}$ of $z$.   We denote the corresponding
proper \'etale groupoids  by
\[
\bU^{(1)}_{z} \subset \bU^{(2)}_{z}
\]
as sub-groupoids of $\cA_I$.

   Since we assume that  $|\cM|$ is compact,  there exist  finitely many
points $z_k,1\leq k\leq m,
$ such that
$$
 \cM  \subset \cU=\bigcup_{k=1}^m \bU^{(1)}_{z_k}.
$$
On  the proper \'etale groupoid $\cU$, since $\sum_{k=1}^m \eta'_{z_k}\not=0$,
define
\begin{equation}
\eta_{z_k}=\frac{\eta'_{z_k}}{\sum_{l=1}^m \eta'_{z_l}}.
\end{equation}
 Then
\ba\label{POU:orbi}
\sum_{k=1}^m\eta_{z_k}(y)=1,\;\;\; y\in \cU.
\na

Suppose that  $\eta'_{z_k} \in C^\infty (\cA_{I_k})$ for $I_k\subset \{1, 2, \cdots, n\}$, then
$\eta_{z_k}$ is a function on $\cU\cap \cA_{I_k}$.
By composing with the projection onto $\cU\cap \cA_{I_k}$, we get    a  smooth function on
$$
\cC'_{I_k}:=  p_{I_k} ^{-1} (\cU\cap \cA_{I_k})  \subset \cC_{I_k},
$$
where $p_{I_k}: \cC_{I_k} \to \cA_{I_k}$ is the projection map associated to the bundle $\cC_{I_k} \to \cA_{I_k}$
constructed in the proof of Theorem \ref{thm:orbi-vir-sys}.

 We will still  denote    the resulting function on $\cC'_{I_k}$ as  $\eta_{z_k}$. Then we have a collection of
 smooth functions
 \[
 \{\eta_{z_k} \in C^\infty (\cC'_{I_k})|  k= 1, 2, \cdots m\}
 \]
 for a sub-collection  $\{I_k   | k=1, 2, \cdots, m\} $ of  the index set $2^{\{1, 2, \cdots, n\}}$.
   It is clear that $\eta_{z_k}$ is an  invariant
  function on $\cC_{I_k}'$  under the equivalence relation (\ref{equi:rel}) , that is, for any $J\subset I_k$,  the cut-off
  function
$\eta_{z_k}$ is constant  along the fiber of bundle
\[
\cC_{I_k,J}'\longrightarrow  \cC_{J,I_k}'.
\]

\begin{lemma}    For any $I\subset \{1, 2, \cdots, n\}$, define
$
\eta_I=\disp{\sum_{k\in \{k|I_k=I\}}} \eta_{z_k} \in C^\infty(\cC_I').
$, then
$\{\eta_I\}$
is  a partition of
unity  on the virtual manifold  $\cC':=\{\cC_I'\}$  in the sense of Definition \ref{POU:vir}.  Let
$
\cV'_I  =  \cC'_I \cap \cV_I,
$
then  the restriction of $\{\eta_I\}$
is   a partition of
unity  on the virtual manifold   $\{\cV'_I  \}$.
\end{lemma}
\begin{proof} The proof of this lemma is analogous to the proof of Lemma \ref{POU:C'}.
We leave it to
interested readers as an exercise.
\end{proof}

 Integrations and invariants for a  virtual orbifold system  can be obtained just as in
\S \ref{i:i:v:s},
under a   technical assumption which is the orbifold version of Assumption \ref{assump:2}.

\begin{assumption} \label{assump:orbi-2}
Given a sequence $\{ y_N | N=1, 2, \cdots, \}$ in $\cC^0_I$ (the unit space of $\cC_I$)  such that
\begin{itemize}
\item $S^0_I (y_N) =0$.
\item $\lim_{N\to \infty } \|k_N\| =0$  where $\| k_N\|$  is  the Banach norm of the fiber component of $ y_N$ over $p_I(y_N)$. Here $p_I$ is the projection for the bundle $\cC_I \to \cA_I$.
\end{itemize}
Then  the sequence $\{p_I(y_N)\}$ converges in $|\cC_I|$.
\end{assumption}

As pointed out in  Remark \ref{Morita}, to define an integration over a proper \'etale  groupoid
$\cG = (\cG^1 \rightrightarrows \cG^0)$,
 we only need to cover the orbit space by a collection of
 groupoids from a sub-collection of orbifold charts, rather than taking   a collection of
 $\cG$-invariant subgroupoids. From this viewpoint,  instead of considering  the groupoid $\cA_I$, we shall
 only need to  study the  non-empty components  in (\ref{multi:cap}). For  this purpose, we fix an $|I|$-tuple
 \[
 \{\tilde x_i\}_{i\in I}  \in \prod_{i\in I} \pi^{-1}(x_i)
 \]
 such that $\bigcap_{\{\tilde x_i\}_{i\in I}   } U_{\tilde x_i}^{(3)}\neq \emptyset$.  
 Let  
 \[
 \cY_I = (\cY_I^1  \rightrightarrows  \cY_I ^0)
 \] 
 be the  full sub-groupoid of $\cA_I$ with the unit space  $\cY^0_I = \cA^0_I  \cap \left( \bigcap_{\{\tilde x_i\}_{i\in I}   } U_{\tilde x_i}^{(3)}\right)$. 
  Denote by
\[
 K_I := \prod_{   \{\tilde x_i\}_{i\in I} }K_{\tilde x_i},
\]
 then $
 \cC^0_I |_{\cY^0_I }  = \cY^0_I  \times K_I \longrightarrow \cY^0_I
$
 generates through $\cA_I$-action a vector bunlde $\cC_I|_{\cY_I }$  over $\cY_I$.  For any ordered pair $I\subset J$, we may choose
 the $|I|$-tuple
$ \{\tilde x_i\}_{i\in I}$ and the $|J|$-tuple
$ \{\tilde x_i\}_{i\in J}$ such that
\[
\{\tilde x_i\}_{i\in I} \subset  \{\tilde x_i\}_{i\in J}.
\]
We set $ \cY^0_{I, J} =\cY^0_{J, I} = \cY^0_I\cap \cY_J^0$, then
\ba\label{J-I}
\cC^0_J |_{\cY^0_{J, I}} = \cC^0_I|_{\cY^0_{I, J}}  \times  \prod_{i\in J\setminus I}K_{\tilde x_i},
\na
generates the bundle $\cC_{J, I}|_{\cY _{J, I}}  \to  \cC_{I, J}|_{\cY _{I, J}}$ so that 
$\{\cC_I|_{\cY_I }\}$ is a virtual proper \'etale groupoid.

 Define $\cW_I = \cV_I \cap \big( \cC_I|_{\cY_I } \big)   \subset \cV_I$, then the  virtual  orbifold
system
\[
 \{(\cV_I, \bE_I, \sigma_I)|  I\subset \{1, 2, \cdots, n\} \}
  \]
arising  from Theorem \ref{thm:orbi-vir-sys} is Morita equivalent to  the  virtual  orbifold
system
\[
\{(\cW_I, \bE_I|_{\cW_I}, \sigma_I)|  I\subset \{1, 2, \cdots, n\} \}.
  \]

Given $I$ with a fixed  $|I|$-tuple
$
 \{\tilde x_i\}_{i\in I}  \in \prod_{i\in I} \pi^{-1}(x_i)
$
 such that $\bigcap_{\{\tilde x_i\}_{i\in I}   } U_{\tilde x_i}^{(3)}\neq \emptyset$.
For $i\in I$, let  $\Theta_i$ be a volume form on  the finite dimensional  $G_{\tilde x_i}$-invariant  linear space $K_{\tilde x_i}$  that is
supported in a small  $\epsilon$-ball  $B_{\epsilon}$ of $0$ in $K_{\tilde x_i}$.
Then the volume form $
 \bigwedge_{i\in I}\Theta_i
$ on $K_I =\oplus_{i\in I} K_{\tilde x_i}$
defines a Thom form $\Theta_I$ on bundle $  \cC_I|_{\cY_I } \ \to \cY_I$.

By (\ref{J-I}),    the volume form
$
 \bigwedge_{i\in J\setminus I}\Theta_i,
$
defines a  Thom form $\Theta_{I, J}$     of the  bundle
\[
\cC_{J, I}|_{\cY _{J, I}}  \to  \cC_{I, J}|_{\cY _{I, J}}.
  \]
 Then one can check  that
the collection of Thom forms $\{\Theta_{I, J} |I\subset J\} $ is a transition Thom  form
for $\{\cC_I |_{\cY_I}\}$.

 Under the inclusions,
\[
\cW_I  = S_I^{-1} (0) \cap \big( \cC_I|_{\cY_I }\big)    \subset     \bE_I|_{\cW_I}
\]
where the last inclusion is given by the zero section of $\bE_I$, we can restrict the Thom form $\Theta_I$
to $\cW_I$ to get an Euler  form  for  the bundle $\bE_I|_{\cW_I}$. Denote this Euler  form of  $\bE_I|_{\cW_I}$  by
  $\theta_I$.
  Then
  \[
  \theta=
\{\theta_I  | I\subset \{1, 2, \cdots, n\}  \}
\]
  is a  twisted virtual  form on $\{\cW_I\}$  with respect to the transition Thom form
$\Theta =\{ \Theta_{I, J}|_{ \cW_{J,I}}:  I\subset J\} $. This twisted virtual form will be called a
{\bf virtual Euler form }  of $\{\bE_I\}$.

 \begin{theorem}
Under the Assumption \ref{assump:orbi-2},  for any $I$, the vertical support of $\Theta_I$   can be chosen sufficiently  small  such that
 $\eta_I\theta_I$ is compactly supported in $\cW_I$.
\end{theorem}
\begin{proof} The proof  is the same as the proof  of Theorem \ref{cpt:sup}. We omit it  here.
\end{proof}

 We assume that the virtual orbifold system  $\{( \cV_I,  \bE_I,  \sigma_I ) \} $ is {\bf oriented}.  Let $\theta=\{\theta_I\}$ and $\eta=\{\eta_I\}$
be the  partition of  unity and  the  virtual Euler form  constructed as above.

Given  $\alpha =\{\alpha_I\}$   a  degree $d$   differential  form on $ \cV=\{\cV_I\}$, we define a virtual
integration of $\alpha$ to be
\ba\label{def:integration:orbi}
\int_{\cV}^{vir} \alpha
=\sum_{I}\int_{\cW_I} \eta_I \cdot \theta_I  \cdot (\alpha_I)|_{\cW_I}.
\na

Now,  given   cohomology class  $ H^d(\cB)$, suppose that under the following smooth maps
\[
\cV_I \subset C_I = X_I \times K_I \longrightarrow X_I \subset \cB,
\]
for   each $ I\subset \{1, 2, \cdots, n\}$,
the pull-back cohomology class   can be represented by a
closed degree  $d$  virtual differential  form $\alpha_I$.   Then we can
define an   invariant   associated to the Fredholm system  $(\cB, \cE, S)$ to be
\ba\label{invariant:orbi}
\Phi(\alpha)=\int_{\cV}^{vir}\alpha.
\na

\begin{remark}\label{rem:ind}
\begin{enumerate}
\item    Given two sets
of local stabilisations with the resulting orbifold virtual systems denoted by
\[
 \{(\cV'_I, \bE'_I, \sigma'_I)|  I\subset \{1, 2, \cdots, m\} \}  \qquad \text{and} \qquad 
 \{(\cV''_I, \bE''_I, \sigma''_I)|  I\subset \{1, 2, \cdots, n\} \}
  \]
  respectively. 
  We can constructed an orbifold  virtual system with boundary from the Fredholm system  $(\cB\times [0, 1], \cE\times [0, 1], S
  \times [0, 1])$, equipped with a  local stabilisation extending two  local stabilisations at the two ends.
 Then we apply the Stokes' formula for the virtual integration in Theorem \ref{int:well-def} to show that 
the invariant defied in (\ref{invariant:orbi}) does not depend on the choice of local stabilisations.
  
\item When the expected dimension of $\cM$,  the difference between the  dimension of $\cV$ and  the virtual rank of $\bE$, is zero, then
$\Phi (1)$ should be thought of
   a virtual Euler number of $\bE$.   If  it happens that $\cM$ consists of  a collection of  smooth   orbifold
   points $\{(x_1, G_{x_1}), (x_2, G_{x_2}), \cdots , (x_n, G_{x_n})\}$, then
   \[
   \Phi (1) = \sum_{i=1}^n \dfrac{1}{|G_{x_i}|},
   \]
     agrees with the  orbifold
   Euler characteristic of $\cM$ (a rational number in general).

\end{enumerate}

\end{remark}

\begin{remark}
In the Gromov-Witten theory, there are some further technical issues in applying
the virtual neighborhood technique to  the moduli spaces of stable maps. There are two main issues:
\begin{enumerate}
\item
the underlying space of
stable maps  is  often stratified,
\item  some of the stratum (like  the domain curves being
spheres with less than three marked point)
 are not smooth due to the fact that the reparametrization
group action is not differentiable.
\end{enumerate}
So in general, the virtual neighborhood technique developed in this section should be thought as a guiding principle
rather than a complete recipe to define the Gromov-Witten invariant. In the remaining part of the paper,
we will explain how to overcome the non-differentiability issue of a reparametrization group and how to get a virtual orbifold system for the  presence of spheres with no marked points.
In this case, as explained in the introduction, the reparametrization group is $PSL(2, \C)$ acting non-smoothly
on  the underlying Fredholm system $(\wtB, \wtE, \dbar_J)$.  It turns out that we can apply the
virtual neighborhood technique directly to $(\wtB, \wtE,  \dbar_J)$  in a  $PSL(2, \C)$-equivariant way, even though the action is not smooth.
In order to establish a $G$-equivariant virtual neighborhood technique when the $G$-action is not smooth, we establish  a slice and tubular neighbourhood theorem for
the  $PSL(2, \C)$-action  on $\wtB$, and propose a notion of smoothness on these slices even though
coordinate changes between these slices are only continuous.  We establish these analytical  results in the next Section.

\end{remark}

\section{Analytical foundation for pseudo-holomorphic spheres}
\label{4}

As in the introduction, let $\wtB$ be the  space
 of $W^{1,p}$-maps from  the standard Riemann sphere $\Sigma$ to
  a compact symplectic manifold  $(X, \omega,J)$   with a fixed homology class $A\in H_2(X, \Z)$,  that is,
\[
\wtB =\{ u\in W^{1, p}(\Sigma, X)| f_*([\Sigma]) =A\},
\]
where $p$ is an even integer greater than $2$.
Then $\wtB$  is an   infinite dimensional Banach manifold whose tangent space at a point $u:\Sigma\to X$ is the Banach space
$
T_u\wtB = W^{1, p}(\Sigma, u^*TX),
$
the space of $W^{1, p}$-sections of the pull-back tangent bundle $u^*TX$.  The exponential map for the Riemannian manifold
$(X, \omega, J)$ provides a coordinate chart  at $u$
\[
\exp_u:   T_u\wtB \longrightarrow \wtB
\]
which is injective on an $\epsilon$ -ball in $T_u\wtB $ for small $\epsilon$.
There is a natural almost complex structure on $\wtB$ induced by $J$.

  Let $\wtE$ be the infinite dimensional Banach bundle over $\wtB$ whose fiber at $u$ is $L^p (\Sigma, \Lambda^{0,1}  T^* \Sigma\otimes_\C u^*TX)$, the space of
$L^p$-section of the bundle $ \Lambda^{0,1} T^*\Sigma\otimes_\C u^*TX$.  The Cauchy-Riemann operator defines a  smooth  Fredholm   section $\dbar_J:  \wtB \to \wtE$.  The  zero set of this section
\[
\wt\cM =\dbar_J^{-1}(0) \subset \wtB
\] is our moduli space of parametrized pseudo-holomorphic spheres in $(X,\omega, J)$ with homology class
$A$.

  The actual moduli space of unparametrized pseudo-holomorphic spheres in $(X,\omega, J)$ is
  the quotient of $\wt\cM$ under the action of the reparametrization group $PSL(2, \C)$.
 Let $G=PSL(2, \C)$ act on $\wtB$ as reparametrization of the domain $\Sigma$
\ba\label{action}
\begin{array}{lll}
G\times \wtB &\longrightarrow & \wtB \\
(g, u) &\mapsto & g\cdot u = u\circ g^{-1}
\end{array}
\na
This action can be lifted to an action on $\wtE$
\[
G\times \wtE \longrightarrow   \wtE
\]
which is given by the pull-back $(g^{-1})^* (\eta)   \in \wtE_{g\cdot u}$ for any $g\in G$ and $\eta \in \wtE_{u}$.
The moduli space $\wt\cM$ is $G$-invariant as the section $\dbar_J$ is $G$-invariant.
The main technical issue in implementing the virtual neighborhood technique is that  the action (\ref{action})  is    continuous but   not  differentiable, as in \cite{Ruan} and \cite{McW}.
 So the bundle $\wtE\to \wtB$ is only a topological $G$-equivariant bundle with a $G$-invariant section
 $\dbar_J$.  Denote by
 \[
 \cB =\wtB/G, \qquad   \cE =\wtE/G, \qquad \cM = \wtM/G
 \]
 the corresponding quotient spaces. Then $\cB$ is a topological orbifold  and $\cE$ is a  topological  orbifold vector bundle with a continuous section $\dbar_J: \cB \to \cE$. We will assume that
\[
\cM= \cM_{0, 0}(X, A, J)
\]
is compact which is guaranteed by some conditions on $A$, such as $A$ is indecomposable, see
Theorem 4.3.4 in \cite{McS}.

Therefore, we obtain a smooth Fredholm system
\[
(\wtB, \wtE, \dbar)
\]
whose quotient  by the reparametrization group $G$ is only a topological orbifold
 Fredholm system
\[
(\cB, \cE, \dbar),
\]
A priori,  we can not  apply the virtual neighborhood technique developed in Section \ref{3} to get a virtual system. As $\wtM$
consists of smooth maps, the key observation in this paper is that we can establish a slice and
 tubular neighbourhood theorem at any point in $\wtM$.   This enables us to
apply the principle of virtual neighborhood technique to $(\wtB, \wtE, \dbar)$ in a $G$-equivariant way.

In this section, we will  prove the slice theorem  for the   $G$-action at each smooth point $u_0\in \wtB$, that means,
$u_0\in C^\infty(\Sigma, X)$,   and establish a tubular neighbourhood for the $G$-orbit through  $ u_0$.  Each slice along a point on the orbit
$\cO_{u_0}= G_\bullet u_0$ is homeomorphic an open neighbourhood of $[u_0]\in \cB = \wtB/G$.   Theorem B in
the Introduction will be proved in Subsection \ref{4.1}, see Theorem \ref{slice-theorem:orbifold}.

We remark  that these slices  are not diffeomorphic and do not form a compatible smooth  orbifold structure near $[u_0]\in \cB$, but their union is  a $G$-invariant smooth neighbourhood of the $G$-orbit in $\wtB$.  This enables us to  propose  certain classes of smooth functions/sections on a small neighbourhood of the moduli space $\cM_{0, 0}(X, A, J)$ in
Subsection \ref{4.2}.   In particular, we construct certain cut-off functions  which fulfill the requirement of Assumption \ref{assump:orbi} for local stabilizations and
a partition of unity for the virtual orbifold system.

To demonstrate how these slices   and  tubular neighbourhoods   can be applied to study the moduli space, we
introduce  a local smooth  obstruction   bundle  $\bE_u$  over  each slice at any point $u \in \wtM$, and show that it can be extended to a smooth sub-bundle of the restriction of $\wtE$ to the tubular neighbourhood of the orbit through $u$.  This is crucial to carry out the local and global  stabilizations in $G$-equivariant way, even though
the  $G$-action
is not smooth.

\subsection{Slice and tubular neighbourhood theorem} \label{4.1}

Given $u_0 \in C^\infty(\Sigma, X)$, for simplicity, we assume that the isotropy group $G_{u_0}$ is trivial. In general,
$G_{u_0}$  is  a finite group, we will return to this general case later.  We will define  a coordinate chart
for a neighbourhood of $[u_0]$ in $\cB$ which is achieved by constructing a slice at $u_0$ for the
$G$-action on $\wtB$.

\begin{lemma} \label{smooth-orbit} The $G$-orbit $\cO_{u_0}= G_\bullet u_0 \subset \wtB$ is a smooth submanifold of $\wtB$.
 In particular,
the tangent bundle $T\cO_{u_0}$ is a smooth sub-bundle of $T\wtB|_{\cO_{u_0}}$.
\end{lemma}
\begin{proof} The inclusion map $\phi:  \cO_{u_0} \to \wtB$ is given by
\[
\phi(g\cdot u_0) = u_0\circ g^{-1} : \Sigma \longrightarrow X,
\]
whose  differential   $d\phi$  at   $u_0$ is
\[
  d\phi:  T_{u_0}\cO_{u_0} = T_{Id}G \longrightarrow  T_{u_0 } \wtB
  \]
  mapping $\xi\in T_{Id}G$ to $ - du_0 (\xi)$. Here $\xi$ is identified with a holomorphic vector field on $\Sigma$
generated by the infinitesimal action of $\xi$, and $du_0$ is the differential of $u_0$.
Consider the map
$$
\Phi: G\times \Sigma\xto{\ H \ } \Sigma\xto{\ u_0 \ }X
$$
where $H(g,x)=g^{-1}\cdot x$. Then  $\Phi$ is  a smooth map. This implies that
$d\phi$ is smooth.
\end{proof}

 Next we define  a
$G$-invariant  $L^2$-metric  on $T\wtB|_{\cO_{u_0}}$ so that the  $L^2$-orthogonal bundle of
$ T\cO_{u_0}$   in $T\wtB|_{\cO_{u_0}}$ can be defined.  Recall that there is an
 $L^2$-inner product  $\Omega_{u_0}:  T_{u_0} \wtB \times T_{u_0}\wtB  \to \R
$ that
is given by
\ba\label{at u_0}
\Omega_{u_0}( v, w) = \int_\Sigma h_{u_0(x)}(v(x),w(x)) dvol_\Sigma(x),
\na
for $v, w\in T_{u_0}\wtB$.
Here $h$ is the Riemannian metric on $X$ induced by $(\omega, J)$ and  $dvol_\Sigma(x)$ is  the volume element  for  the round  metric on  $\Sigma = S^2$.  In order to  get a $G$-invariant $L^2$-metric on $T\wtB|_{\cO_{u_0}}$, we only need
to define an $L^2$-inner product on
$
T_{g\cdot u_0} \wtB
$
as
\ba\label{invariant:metric}
\Omega_{g\cdot  u_0}( v, w) = \Omega_{u_0}( g^{-1}\cdot v,  g^{-1}\cdot w)
\na
for $v, w\in T_{g\cdot u_0} \wtB = W^{1, p}(\Sigma, (u_0\circ g^{-1})^* TX)$.
Note that when $g$ varies in $G$,   $g^{-1}\cdot v$  and $g^{-1}\cdot w$ are not smooth sections of  $T\wtB|_{\cO_{u_0}}$, so  the following  lemma is   surprisingly pleasant  and is one of key lemmas in this paper.

\begin{lemma} \label{L^2-metric}  The $L^2$-metric $\Omega$ is a $G$-invariant smooth metric on $T\wtB|_{\cO_{u_0}}$.
\end{lemma}
\begin{proof} The $G$-invariant property
$
\Omega_{u_0}( v, w) = \Omega_{g\cdot  u_0}( g\cdot v, g\cdot w)
$
follows from the definition.
For  $v, w\in T_{g\cdot u_0} \wtB$,
 \ba\label{at gu_0}
\Omega_{g\cdot  u_0}( v, w)  =  \int_\Sigma h_{u_0(x)}(v(g(x)),w(g(x))) dvol_\Sigma(x).
\na
Using the  substitution $y=g( x)$, we get
\[
\Omega_{g\cdot  u_0}( v, w)  =  \int_\Sigma h_{u_0(g^{-1}(x))}(v(y),w(y)) Jac_g^{-1} dvol_\Sigma(y)
\]
where $Jac_g$ is the Jacobian associated to the automorphism $g$ of $\Sigma$. As $u_0$ is a smooth map
and  the $G$-action on  $\Sigma$ is smooth, so $\Omega$ is a smoothly varying metric  on $T\wtB|_{\cO_{u_0}}$.
\end{proof}

The main result of this section is  the following theorem about the existence of slices and tubular neighbourhoods  for a smooth $G$-orbit.

\begin{theorem} \label{slice-theorem}  Let $u_0: \Sigma \to X$ be a smooth  map in $\wtB$ with trivial isotropy group.  Let $\pi_\cB$ be the   quotient map $\pi_\cB: \wtB \to \cB$.
\begin{enumerate}
\item The $G$-orbit $\cO_{u_0}$ through $u_0$ is a smooth sub-manifold of $\wtB$ whose normal bundle  $\cN_{\cO_{u_0}}$ can be identified with a
smooth sub-bundle of $T\wtB|_{\cO_{u_0}}$.
\item  There is  a sufficiently small  $\eps_0>0$ and  an   open $\epsilon_0$-ball  $  \cN^{\eps_0}_{u_0}$ in the fiber of $ \cN_{\cO_{u_0}}$
at $u_0$
with respect to  the $W^{1, p}$-norm  such that   $\cS^{\eps_0}_{u_0} = \exp_{u_0} (  \cN^{\eps_0}_{u_0})  $  is a
{\bf slice}  at $u_0$ for the  $G$-action in the sense that
\[
\pi_\cB:  \cS^{\eps_0}_{u_0}  \to \cB
\]
is a homeomorphism onto its image $  \pi_\cB(\cS^{\eps_0}_{u_0}) \subset \cB$ for which we use the notation   $\bU^{\eps_0}_{u_0}$.
The  map
\[
\phi_{u_0}:  \cN^{\eps_0}_{u_0}  \xto{\ \ \exp_{u_0} \ \ } \cS^{\eps_0}_{u_0}  \xto{\pi_\cB}  \bU^{\eps_0}_{u_0}
\]
is  called a {\bf  Banach coordinate chart} of $\cB$  at $[u_0]$.
\item  There exists  a  tubular neighbourhood, denoted by $\cT^{\eps_0}_{u_0}$, which is  a $G$-invariant open set in $\wtB$ containing $\cO$.
  The restricted quotient map $ \pi_\cB:  \cT^{\eps_0}_{u_0} \longrightarrow  \bU^{\eps_0}_{u_0}$ is a topologically trivial principal $G$-bundle with a trivialising  section
given by the slice $ \cS^{\eps_0}_{u_0}$.
\end{enumerate}
\end{theorem}
\begin{proof} The first part of Claim (1) follows from Lemma \ref{smooth-orbit}.    Let $\cN$ be the normal bundle of $\cO$ in $\wtB$ defined by the following
exact sequence of vector bundles
\[
0\to T\cO_{u_0} \to T\wtB|_{\cO_{u_0}} \to \cN_{\cO_{u_0}} \to 0.
\]
Applying  the $G$-invariant $L^2$-metric $\Omega$ defined in (\ref{at gu_0}),  we have
\[
\cN_{\cO_{u_0}}  \cong (T \cO_{u_0})^\perp,
\]
 the $L^2$-orthoginal  bundle  to $T \cO_{u_0}$ in    $T\wtB|_{\cO_{u_0}}$.  Here  we use the fact that $W^{1, p}\subset L^2$ as $\Sigma$ is  compact.  As the $L^2$-metric $\Omega$ is  a smooth metric on $T\wtB|_{\cO_{u_0}}$,   the normal bundle
$\cN$ can be identified as  a smooth sub-bundle of $T\wtB|_{\cO_{u_0}}$ such that
\[
  T\wtB|_{\cO_{u_0}} \cong T\cO_{u_0} \oplus \cN_{\cO_{u_0}}.
  \]

(2) Given     a  sufficiently small positive number $\eps_0 $ .  Let $\cN^{\eps_0}_{u_0}$ be the open  $\epsilon_0$-ball
in the fiber $\cN_{u_0}= \cN_{\cO_{u_0}}|_{u_0}$ with respect to the induced  $W^{1,p}$-metric from $T_{u_0}\wtB$, that is,
\[
\cN^{\eps_0}_{u_0}=\{ v\in \cN_{u_0}|  \|v\|_{W^{1, p}(\Sigma, u_0^*TX)} < \eps_0 \}.
\]
  Then the exponential map $\exp_{u_0}: \cN^{\eps_0}_{u_0} \to \wtB$ is injective. Define
  \[\cS^{\eps_0}_{u_0} = \exp_{u_0} ( \cN^{\eps_0}_{u_0}).
  \]
  Then $ \pi_\cB:  \cS^{\eps_0}_{u_0}  \to \cB$
is a homeomorphism onto its image $\bU^{\eps_0}_{u_0}$.    So  $\cS^{\eps_0}_{u_0} = \exp_{u_0} (  \cN^{\eps_0}_{u_0})  $  is a
{\bf slice}  at $u_0$ for the  $G$-action.   The map
\[
\phi_{u_0}= \pi_\cB\circ \exp_{u_0}:   \cN^{\eps_0}_{u_0}   \to \bU^{\eps_0}_{u_0}
\]
is a homeomorphism,  which is  a  Banach  coordinate chart.

(3) The  tubular neighbourhood   $\cT^{\eps_0}_{u_0}$  of $\cO_{u_0}$ in $\wtB$ is defined to be  the $G$-saturation of the slice $\cS^{\eps_0}_{u_0}$, which is
\[
\cT^{\eps_0}_{u_0} = G\cdot  \cS^{\eps_0}_{u_0} =\pi^{-1}_\cB (\bU_{[u_0]}).
\]
Due to the non-differentiablity of  the $G$-action, we need to show that $\cT^{\eps_0}_{u_0}$ is a smooth neighbourhood of
$\cO$.  As $\cN_{\cO_{u_0}}$ is a smooth sub-bundle of $T\wtB|_{\cO_{u_0}}$, the exponential map $\exp: \cN_{\cO_{u_0}} \to \wtB$ defined by $\exp_{u'}$ along the fiber at each $u' \in\cO_{u_0}$  is a smooth map.

We claim that the map $\exp: \cN_{\cO_{u_0}} \to \wtB$  is  also $G$-equivariant in the following sense, for any
$v\in \cN_{u_0}$ and any $g\in G$, we have
\ba\label{G:equ}
g\cdot \exp_{u_0} (v) = \exp_{g\cdot u_0} (g\cdot v):  \Sigma\longrightarrow X,
\na
which can be verified directly.

 Using the $W^{1, p}$-norm on $T_{u_0}\wtB$, we can define a new
$W^{1, p}$-norm on $ T_{g\cdot u_0}\wtB$ satisfying the $G$-invariant property, that is, for $v\in T_{g\cdot u_0}\wtB$
\ba\label{new:metric}
\|v\|'_{W^{1, p}} = \|g^{-1}\cdot v\|_{W^{1, p}(\Sigma, u_0^*TX)}.
\na
 The proof of Lemma \ref{L^2-metric}
can be employed here to show that this new $W^{1, p}$-norm  on $T\wtB|_{\cO_{u_0}}$  is a smooth $G$-invariant
norm.

Let $\cN_{\cO_{u_0}}^{\eps_0}$ be the  open $\eps_0$-ball bundle of $\cN_{\cO_{u_0}}$ with respect to this  $G$-invariant $W^{1, p}$-norm.
Then $\cN_{\cO_{u_0}}^{\eps_0}  = G\cdot \cN^{\eps_0}_{u_0}$ is a  $G$-invariant  smooth Banach submanifold of $ \cN$.  From
the G-equivariant property (\ref{G:equ}), we obtain
\[
\cT^{\eps_0}_{u_0} = G\cdot  \cS^{\eps_0}_{u_0}  = G\cdot \exp_{u_0}(\cN^{\eps_0}_{u_0})   =\exp (\cN_{\cO_{u_0}}^{\eps_0})
\]
is a smooth tubular neighbourhood of $\cO$ which is $G$ invariant.  The principal
$G$-bundle $\cT^{\eps_0}_{u_0} \to \bU^{\eps_0}_{u_0}$ is topological trivial, as the slice at $u_0$ defines a global trivialising section.
 \end{proof}

\begin{remark} \begin{enumerate}
\item   The obvious $L^2$-metric on $T\wtB|_\cO$ induced from its own $W^{1, p}$-norm is not $G$-invariant, as at $g\cdot u_0$ and for $v, w\in T_{g\cdot u_0}\wtB$, the inherited  inner product is given by
 \[
  \Omega_{g\cdot  u_0}( v, w)   =   \disp{ \int_\Sigma} h_{g\cdot u_0(x)}(v(x),w( x) ) dvol_\Sigma(x)  \]
  which is different to the  $G$-invariant $L^2$-metric $\Omega$ at $g\cdot u_0$ defined by  (\ref{at gu_0}).
 \item
 Note  that $\cN_{\cO_{u_0}}$ is only a topological
$G$-equivariant Banach bundle over $\cO$ due to the non-smooth action of $G$ on $\wtB$. The exponential map
\[
\exp: \cN_{\cO_{u_0}} \longrightarrow \wtB
\]
is a smooth map, with respect to the topological $G$-actions, it is also $G$-equivariant. We will use the notion of
{\em topological $G$-equivariant smooth map} to describe this property, so it does not cause any confusion with the usual
smooth $G$-equivariant map.
\item The normal bundle $\cN_{\cO_{u_0}}$  depends on the choice of
$u_0$  as  the $L^2$-metric (\ref{at u_0})  depends on $u_0$.
 If we take another  point  $u'\in  cO_{u_0} $, the $G$-invariant $L^2$-metric on  $T\wtB|_{\cO_{u_0}}$ induced from the
$L^2$-metric on $T_{u'}\wtB$ would be different to the metric (\ref{at u_0}). The normal bundle $\cN_{\cO_{u_0}}$ would be   a different sub-bundle  of $T\wtB|_{\cO_{u_0}}$.  Then  this difference will result in  a different slice $\cS^{\eps'}_{u'}$  at $u'$,  a different  tubular
neighbourhood $\cT^{\eps'}_{u'}$, and a different Banach coordinate chart for $\bU^{\eps'}_{u'}$ at $[u']=[u_0]$. On the overlap
$\bU^{\eps_0}_{u_0}\cap \bU^{\eps'}_{u'}$, the coordinate change
\[
\phi_{u_0, u'} = \phi_{u'}\circ \phi_{u_0}^{-1}: \phi_{u_0}^{-1} (\bU^{\eps_0}_{u_0}\cap \bU^{\eps'}_{u'}) \longrightarrow \phi_{u'}(\bU^{\eps_0}_{u_0}\cap \bU^{\eps'}_{u'})
\]
is only a homeomorphism, so there is no canonical smooth Banach structure near  $[u_0]\in \cB$ for any smooth map $u_0\in \wtB$.
By the same argument, given two smooth maps $u$ and $u'$ in $\wtB$ such that $[u]\neq [u'] \in \cB$,  then the coordinate change
on the overlap  $\bU^{\eps_0}_{u}\cap \bU^{\eps'}_{u'}$ is only a homeomorphism.
\end{enumerate}
\end{remark}

Now suppose that the finite group $G_{u_0}$ is not trivial, then the above constructions still work with some modifications at a few places. Lemma \ref{smooth-orbit} still holds with the $G$-orbit $\cO_{u_0} = G_\bullet u_0$ parametrized by  the homogenous space $G/G_{u_0}$.
The $L^2$-inner product on $T_{u_0}\wtB$, previously defined for the trivial isotropy case, has to be averaged over the $G_{u_0}$-action,  that is,
\ba\label{ave:at u_0}
\Omega_{u_0}( v, w) = \sum_{\gamma \in G_{u_0}}  \int_\Sigma h_{u_0(x)}(\gamma \cdot v(x), \gamma \cdot w(x)) dvol_\Sigma(x),
\na

The proof of Lemma \ref{L^2-metric} can be carried over to this situation, so the resulting metric $\Omega$ is a $G$-invariant
smooth metric on $T\wtB|_{\cO_{u_0}}$.  Taking the $L^2$-orthogonal  bundle of $T\cO_{u_0}$ in $T\wtB|_{\cO_{u_0}}$, we get a smooth decomposition
\[
T\wtB|_{\cO_{u_0}}  \cong T \cO_{u_0} \oplus \cN_{\cO_{u_0}}.
\]
This provides a smooth sub-bundle structure on the normal bundle $ \cN_{\cO_{u_0}}$.

Similarly, there is a
$G$-invariant smooth $W^{1, p}$-norm on $\cN_{\cO_{u_0}} \subset T\wtB |_{\cO_{u_0}}$ given  by taking
\ba\label{new:2}
\|v\|'_{W^{1, p}} =\sum_{\gamma\in G_{u_0}}  \| \gamma\cdot g^{-1}\cdot v\|_{W^{1, p}(\Sigma, u_0^*TX)}.
\na
for $v\in T_{g\cdot u_0}\wtB$.
 Let $\epsilon_0 $ be  a positive real number smaller than
the injective radius of $X$ at $u_0(x)$ for any $x\in \Sigma$. Let $\cN_{\cO_{u_0}}^{\epsilon_0}$ be the open  $\epsilon_0$-ball bundle
in $\cN_{\cO_{u_0}}$ with respect to the new $G$-invariant smooth $W^{1, p}$-norm (\ref{new:2}). Then both  the  fiber $\cN^{\eps_0}_{\cO_{u_0}}|_{u_0}$ and  the
 $\epsilon_0$-slice   $\cS_{u_0}^{\eps_0} = \exp_{u_0} (\cN^{\eps_0}_{\cO_{u_0}}|_{u_0})$  are $G_{u_0}$-invariant, and the map
 \[
 \exp_{u_0}: \qquad  \cN_{u_0}^{\epsilon_0} := \cN^{\eps_0}_{\cO_{u_0}}|_{u_0} \longrightarrow  \cS_{u_0}^{\eps_0}
 \]
 is $G_{u_0}$-equivariant.  This provides a local Banach orbifold chart of $\cB$ at $[u_0]$.  The proof of Theorem \ref{slice-theorem} can be carried over to establish the   corresponding slice and tubular neighbourhood structures near each topological  orbifold point  $[u_0]\in \cB$.

\begin{theorem} \label{slice-theorem:orbifold}   Let $u_0: \Sigma \to X$ be a smooth  map in $\wtB$ with  a finite isotropy group $G_{u_0}$. Let $\pi_\cB$ be the   quotient map $\pi_\cB: \wtB \to \cB$.
\begin{enumerate}
\item The $G$-orbit $\cO_{u_0}$ through $u_0$ is a smooth sub-manifold of $\wtB$. The  normal bundle  $\cN_{\cO_{u_0}}$ is a smooth
Banach bundle and  can be identified with a topological  $G$-equivariant
smooth sub-bundle of $  T\wtB|_{\cO_{u_0}}$.

\item   There exists a sufficiently small $\eps_0$  such that the exponential map $\exp_{u_0}:  \cN^{\eps_0}_{\cO_{u_0}} \to \wtB$
is well-defined. Here $\cN^{\eps_0}_{\cO_{u_0}}$ is the $\eps_0$-ball bundle of $\cN_{\cO_{u_0}}$. Moreover,
the tubular neighbourhood   
\[
\cT^{\eps_0}_{u_0} = \exp (\cN^{\eps_0}_{\cO_{u_0}})
\]
  is a $G$-invariant open set in $\wtB$ containing
$\cO_{u_0}$.

\item There exists  a  $G_{u_0}$-invariant slice  $\cS^{\epsilon_0}_{u_0} \subset \wtB$ at $u_0$
such that the quotient space  $\cS^{\epsilon_0}_{u_0}/G_{u_0}$ is homeomorphic to
$ \pi_\cB(\cS^{\epsilon_0}_{u_0})$,   denoted by
$\bU_{u_0}^{\epsilon_0}$.
The  map
\[
\phi_{u_0}: \qquad  \cN^{\epsilon_0}_{u_0}:   =   \cN^{\epsilon_0}_{\cO_{u_0}}|_{u_0}  \xto{\ \ \exp_{u_0} \ \ } \cS^{\epsilon_0}_{u_0}  \xto{\pi_\cB}  \bU_{u_0}^{\epsilon_0}
\]
provides  a {\bf Banach  orbifold  chart} $( \cN^{\epsilon_0}_{u_0}, G_{u_0}, \phi_{u_0})$  for $\bU_{u_0}^{\epsilon_0}$.
\end{enumerate}

\end{theorem}

 Due to the non-differentiability action of $G$, those Banach orbifold charts from Theorem \ref{slice-theorem:orbifold} only provide a topological orbifold structure on a neighbourhood of $\cM$ in $\cB$
\[
\cW= \bigcup_{[u_0]\in \cM}  \bU_{u_0}^{\eps_0}.
\]
To see this,  suppose that
 $ \bU_{u_0}^{\eps_0} \subset  \bU_{u_1}^{\eps_1}$ for $u_0, u_1 \in \wtM$.
 There exists  the topological   embedding of orbifold charts
 \ba\label{orbi:change}
 \phi_{u_0, u_1}:  ( \cN^{\epsilon_0}_{u_0}, G_{u_0}, \phi_{u_0}) \longrightarrow ( \cN^{\epsilon_1}_{u_1}, G_{u_1}, \phi_{u_1})
 \na
 that is, there exists  a topological embedding
 \[
 \phi_{u_0, u_1}: \cN^{\epsilon_0}_{u_0} \longrightarrow \cN^{\epsilon_1}_{u_1}
 \]
 which is equivariant with respect to  an injective   group homomorphism  $\iota_{u_0, u_1}:  G_{u_0} \to G_{u_1}$ such that the following diagram commutes
 \[\xymatrix{
  \cN^{\epsilon_0}_{u_0}\ar[rr]^{\phi_{u_0, u_1}}\ar[d]_{\phi_{u_0}}& & \cN^{\epsilon_1}_{u_1}\ar[d]
  ^{\phi_{u_1}} \\
   \bU_{u_0}^{\eps_0} \ar[rr]^{\subset} & &  \bU_{u_1}^{\eps_1}.
 }\]
 Therefore,   the induced embedding of their slices
 \ba\label{orbi:slice:change}
 \Phi_{u_0, u_1} = \exp_{u_1} \circ  \phi_{u_0, u_1} \circ \exp_{u_0}^{-1}:
 \cS_{u_0}^{\eps_0}  \longrightarrow  \cS_{u_1}^{\eps_1}
 \na
 is  only a  topological embedding,  which is also equivariant with respect to     $\iota_{u_0, u_1}$.

Though the orbifold coordinate changes (\ref{orbi:change}) and  (\ref{orbi:slice:change})  are only   homeomorphisms,  we have   a smooth inclusion
\[
\cT^{\eps_0}_{u_0}\subset \cT^{\eps_1}_{u_1}
\]
 for   their  tubular neighbourhoods as  prescribed
by Theorem \ref{slice-theorem:orbifold}. Given a
smooth function  $f$ on $\cS_{u_1}^{\eps_1}$, then under (\ref{orbi:slice:change}),    a priori,
$f \circ \Phi_{u_0, u_1}$
 might not be a  smooth function on $\cS_{u_0}^{\eps_0}$.
 We remark that when $f$  has a $G$-invariant smooth  extension  $\hat f:  \cT^{\eps_1}_{u_1} \to \R$,
 then  $f \circ \Phi_{u_0, u_1}$ is actually a smooth function, as it is the restriction of
 $\hat f$  to  $\cS_{u_0}^{\eps_0}$.
 This simple observation plays a crucial role in the notion of
 smooth functions/bundles on a neighbourhood of $\cM$.

 \subsection{Smooth functions on a neighbourhood of $\cM$} \label{4.2}

  By our  assumption, the moduli space $\cM = \cM_{0, 0}(X, J, A)$ is a compact subset of $\cB$. There exist  finitely  many  points $\{[u_1], [u_2],  \cdots, [u_n]\}$ in $\cM$ with their smooth representatives
 \ba\label{choose:ui}
 u_1,  u_2,  \cdots, u_n
 \na
 in $\wtB$  whose   slices  $\{\cS_{u_i}^{\eps_i}\}_i$   and  tubular neighbourhoods
 $\{\cT_{u_i}^{\eps_i}\}_i$,  prescribed by Theorem \ref{slice-theorem:orbifold},   define
  a system of Banach orbifold charts
    \[
 (\cN_{u_i}^{\eps_i}, G_{u_i}, \phi_{u_i}, \bU^{\eps_i}_{u_i})  \qquad i=1, \cdots, n,
 \]
  covering an open neighbourhood  of $\cM$.  By a routine procedure, we can choose $\{[u_1], [u_2],  \cdots, [u_n]\}$ such that
 \ba\label{smooth:cover}
   \bigcup_{i=1}^n \bU ^{\eps_i/3}_{u_i} \supset \cM.
 \na

 For simplicity again, we first  assume that all isotropy groups involved are trivial. Denote by
 \[
 \cU =  \bigcup_{i=1}^n \bU ^{\eps_i}_{u_i}.
 \]
 Each $\bU ^{\eps_i}_{u_i}$ is endowed with a Banach coordinate chart $(\cN_{u_i}^{\eps_i}, \phi_{u_i})$, a tubular neighbourhood
 $\cT^{\eps_i}_{u_i}$ for $\cO_{u_i} = G\bullet u_i$ in $\wtB$ as in Theorem \ref{slice-theorem}.

 \begin{definition}\label{smooth}  Given an  open set $\cV\subset \cU$,
 a function   $f: \cV \to \R$ is called smooth with respect to the   coordinate system
$ \{(\phi_{u_i}:  \cN^{\eps_i}_{u_i} \longrightarrow \bU ^{\eps_i}_{u_i})\}_i$ if for each $i$, the coordinate function
\[
f \circ  \phi_{u_i}: \qquad   \phi_{u_i}^{-1} (\cV\cap \bU_{u_i}^{\eps_i}) \longrightarrow \cV  \longrightarrow \R
\]
is a smooth function.
  \end{definition}

  Note that the coordinate system $ \{(\phi_{u_i}:  \cN^{\eps_i}_{u_i} \longrightarrow \bU ^{\eps_i}_{u_i})\}$
is only a topological atlas, not a smooth atlas, as the transition functions  are only homeomorphisms.   At first  sight,
Definition \ref{smooth}  seems inconsistent with the non-smooth coordinate changes.
We will give  some  interesting smooth
functions on some open set in $\cU$ to show that Definition \ref{smooth} provides a correct notion of smoothness for functions defined on a neighbourhood of  $\cM$.

\begin{lemma} \label{smooth:ontube}
 Any $G$-invariant smooth function on the  tubular neighbourhood  $\cT_{u_i}^{\eps_i}$  defines a smooth  function on
$\bU_{u_i}^{\eps_i} \subset  \cU$ in the sense of Definition \ref{smooth}.
\end{lemma}
\begin{proof} Let $f$ be a  $G$-invariant smooth function on the  tubular neighbourhood  $\cT_{u_i}^{\eps_i}$. Then
$f$ defines a function $\bar f$ on $\bU_{u_i}^{\eps_i} $ and
\[
\bar f\circ \phi_{u_i} = f|_{\cS_{u_i}^{\eps_i}} \circ \exp_{u_i}:  \cN_{u_i}^{\eps_i} \to \R
\]
 is a smooth function. It suffices to show that
\[
\bar f \circ  \phi_{u_j}: \qquad   \phi_{u_j}^{-1}  (\bU_{u_i}^{\eps_i}  \cap \bU_{u_j}^{\eps_j}) \longrightarrow\bU_{u_i}^{\eps_i}  \longrightarrow \R
\]
is smooth  for all $j\neq i$ with a non-empty intersection $\bU_{u_i}^{\eps_i} \cap \bU_{u_j}^{\eps_j}$. Note that
 $f$ restricts to a smooth function
on $\cT_{u_i}^{\eps_i}\cap \cS_{u_j}^{\eps_j}$.  As $\cT_{u_i}^{\eps_i}\cap \cS_{u_j}^{\eps_j}$ is mapped to
$ \phi_{u_j}^{-1}  (\bU_{u_i}^{\eps_i}  \cap \bU_{u_j}^{\eps_j})$ under the diffeomorphism $\exp_{u_j}$. We have
\[
\bar  f\circ \phi_{u_j} = f |_{\cT_{u_i}^{\eps_i}\cap S_{u_j}^{\eps_j}} \circ \exp_{u_j}:   \phi_{u_j}^{-1}  (\bU_{u_i}^{\eps_i}  \cap \bU_{u_j}^{\eps_j}) \longrightarrow \cT_{u_i}^{\eps_i}\cap \cS_{u_j}^{\eps_j}   \longrightarrow \R
\]
 which is smooth.  So $\bar f: \bU_{u_i}^{\eps_i}  \to \R$ is smooth in the sense of Definition \ref{smooth}.
\end{proof}

With this lemma,  in order to show that  the set of  smooth functions on $\cU$ near the moduli space $\cM$  as defined in Definition
\ref{smooth} is not empty,   we need to  introduce
a special set of $G$-invariant smooth functions on  some tubular neighbourhoods $\{\cT_{u_i}^{\eps_i}\}$. These include  the cut-off functions used in applying the virtual neighborhood technique to   $(\wtB, \wtE, \dbar_J)$.
We begin with the norm function on some $W^{k, p}$-spaces.

\begin{proposition} \label{norm:function}
  For any {\bf even}  positive   integer $p>2$, the norm function $f(v) = \|v\|^p$
 is a smooth function on the Banach space $ W^{k, p}(\Sigma, u_i^*TX) $ for each  $u_i$ in (\ref{choose:ui}).
\end{proposition}
\begin{proof}  Assume that $k=0$, then for $v\in W^{k, p}(\Sigma, u_i^*TX) $,
\[
f(v) =  \int_{\Sigma} |v(x)|_{h(u_i(x))}^p dvol_\Sigma (x) =\int_\Sigma h_{u_i(x)}(v(x), v(x))^{p/2} dvol_\Sigma(x).
\]
Here $h$ is the Riemannian metric on $X$ defined by the symplectic form $\omega$ and its compatible almost complex structure $J$. The  first   derivative in  the direction  $w\in L^p(\Sigma, u_i^*TX)$ is
\[\begin{array}{lll}
D_w f(v) &= &\left. \dfrac{d}{dt} \right\vert_{t=0} \disp{\int_\Sigma  } h_{u_i(x)}(v(x) +  t w(x), v(x)+t w(x))^{p/2} dvol_\Sigma(x)\\[3mm]
&=&p  \disp{\int_\Sigma  }    h_{u_i(x)}(v(x), v(x))^{p/2-1}   h_{u_i(x)}(v(x), w(x))  dvol_\Sigma(x)\\[3mm]
&=&   \disp{\int_\Sigma  }  p   |v(x)|_{h(u_i(x))}^{p-2}   h_{u_i(x)}(v(x), w(x))  dvol_\Sigma(x).
\end{array}
\]
So
\[|D_w f(v)|  \leq    \disp{\int_\Sigma  }   p   |v(x)|_{h(u_i(x))}^{p-1}  \cdot    |w(x)|_{h(u_i(x))}  dvol_\Sigma(x).
\]
Applying H\"older's inequality, we obtain
\[
|D_w f(v)|  \leq \left(\disp{\int_\Sigma  }    \big(  p   |v(x)|_{h(u_i(x))}^{p-1}  \big)^{\frac{p}{p-1}} dvol_\Sigma(x)\right)^{\frac{p-1}{p}}
   \left(\disp{\int_\Sigma  }   |w(x)|^p _{h(u_i(x))   } dvol_\Sigma(x) \right)^{\frac{1}{p}}.
   \]
The right hand side is, for some constant $C_p$  depending on $p$,
 \[\begin{array}{lll}
&&  C_p  \left(\disp{\int_\Sigma  }     |v(x)|_{h(u_i(x))}^p   dvol_\Sigma(x)\right)^{\frac{p-1}{p}}
   \left(\disp{\int_\Sigma  }   |w(x)|^p _{h(u_i(x))   } dvol_\Sigma(x) \right)^{\frac{1}{p}}\\[3mm]
   &=&  C_p \left(  \big(\disp{\int_\Sigma  }     |v(x)|_{h(u_i(x))}^p   dvol_\Sigma(x)\big)^{\frac{1}{p}} \right)^{p-1}
   \left(\disp{\int_\Sigma  }   |w(x)|^p _{h(u_i(x))   } dvol_\Sigma(x) \right)^{\frac{1}{p}}\end{array}
 \]
  which implies
   \[
   |D_w f(v)|  \leq  C_p\|v\|^{p-1}_{L^p} \|w\|_{L^p}.
   \]
   Hence, the first derivative is continuous.  The second derivative of $f(v)$  for $w_1, w_2 \in L^p(\Sigma, u_i^*TX)$ is
   \[\begin{array}{lll}
   D_{w_1}D_{w_2} f(v) &=& \disp{\int_\Sigma  }  p   |v(x)|_{h(u_i(x))}^{p-2}   h_{u(x)}(w_1(x), w_2(x)) dvol_\Sigma(x) \\[3mm]
   &&  +
    \disp{\int_\Sigma  }  p (p-2)   |v(x)|_{h(u_i(x))}^{p-4}   h_{u_i(x)}(v(x), w_1(x)) h_{u_i(x)}(v(x), w_2(x))dvol_\Sigma(x) .
    \end{array}
    \]
    So
    \[
      | D_{w_1}D_{w_2} f(v)|   \leq C_p   \disp{\int_\Sigma  }     |v(x)|_{h(u_i(x))}^{p-2}  \cdot    |w_1(x)|_{h(u_i(x))}
       |w_2(x)|_{h(u_i(x))}dvol_\Sigma(x).\]
    Applying  the generalized  H\"older's inequality  for
   $
    \dfrac{p-2}{p} + \dfrac{1}{p} + \dfrac{1}{p} = 1,
  $
     we get
    \[
      | D_{w_1}D_{w_2} f(v)|   \leq C_{p}\|v\|^{p-2}_{L^p} \|w_1\|_{L^p} \|w_2\|_{L^p},
      \]
      which implies that the second derivative is also continuous.
   Repeating this procedure, which will automatically  terminate after a few steps for  any even positive integer $p$,
   we establish that $f$ is a smooth function.

   For $k=1$, the function $f$ is given by
\[
f(v) =  \int_{\Sigma} |v(x)|_{h(u_i(x))}^p dvol_\Sigma (x)  +
\int_\Sigma  |\nabla v(x) |_{h(u_i(x))}^p dvol_\Sigma (x).
\]
 Then just repeating the above calculations for each of these two terms, we get the smoothness for $f$. Similar arguments  can be applied for  all other $k$.
   \end{proof}

 From now on, we will assume that $p>2$ is an even integer  for the $W^{1, p}$ spaces.
   In particular, the norm function $f_{u_i}:  T_{u_i} \wtB = W^{1, p}(\Sigma, u_i^*TX)  \to \R$ defined by
   \ba\label{norm:f-u-i}
   f_{u_i}(v) =  \int_{\Sigma} |v(x)|_{h(u_i(x))}^p dvol_\Sigma (x)  +
\int_\Sigma  |\nabla v(x) |_{h(u_i(x))}^p dvol_\Sigma (x)
\na
    is a smooth function.  The next proposition  says that the
   $G$-invariant extension of this norm function $f_{u_i}$ is a smooth function on the smooth Banach bundle $\cN_{u_i}$ and hence  defines
   a smooth function $\bU_{u_i}^{\eps_i} \subset \cU$ by Lemma \ref{smooth:ontube}.

   \begin{proposition} \label{smooth:ext}   Let $\tilde{F}_{u_i}$ be the $G$-invariant extension of the norm function $f_{u_i}:
   T_{u_i} \wtB  \to \R$
  to $T  \wtB  | _{\cO_{u_i}}$.  Then  $\tilde{F}_{u_i}$  is  a smooth function on $T  \wtB  | _{\cO_{u_i}}$.  Let
 $\tilde{f}_{u_i}$  be the  restriction of    $\tilde{F}_{u_i}$ to $\cN_{\cO_{u_i}}^{\eps_i} \subset T  \wtB  | _{\cO_{u_i}} $. Then
   \[
   \tilde f_{u_i}\circ \exp^{-1}:  \cT_{u_i}^{\eps_i} \longrightarrow  \cN_{\cO_{u_i}}^{\eps_i}  \longrightarrow  \R
   \]
    is a smooth  $G$-invariant function.
   \end{proposition}

    \begin{proof}
   As  $\tilde F_{u_i}$ is  the $G$-invariant extension of $f_{u_i}: T_{  u_i} \wtB \to \R$, we have, for
  $w\in  T_{g\cdot u_i} \wtB$,
  \[\begin{array}{lll}
  \tilde F_{u_i} (g\cdot u_i, w)  &=& f_{u_i} (g^{-1}\cdot w)\\[3mm]
  &=&\disp{ \int_\Sigma} \big( |w(g(x))|_{h(u_i(x))}^p +  |\nabla w(g(x))|_{h(u_i(x))}^p\big) dvol_\Sigma(x)\\[3mm]
   &=&\disp{ \int_\Sigma} \big( |w(y)|_{h(u_i(g^{-1}(y)))}^p +  |\nabla w(y)|_{h(u_i(g^{-1}(y)))}^p\big) Jac_g^{-1} dvol_\Sigma(y)
   \end{array}
\]
where $Jac_g$ is the Jacobian associated to the automorphism $g$ of $\Sigma$. Using the same argument as in Lemma
\ref{L^2-metric} and Proposition  \ref{norm:function}, we know that $\tilde F_{u_i}$ is smooth in $g\in G$ and $w\in  \cN^{\eps_i}_{g\cdot u_i}$. Therefore  $\tilde{f}_{u_i}  = \tilde{F}_{u_i}|_{\cN_{\cO_{u_i}}^{\eps_i}} $ is a $G$-invaraint  smooth function.
As $\exp:  \cN_{\cO_{u_i}}^{\eps_i}\to  \cT_{u_i}^{\eps_i} $ is a diffeomorphism and $G$-equivariant as a homeomorphism,  hence  $\tilde{f}_{u_i} \circ \exp^{-1}$ is  a smooth function on
   $ \cT_{u_i}^{\eps_i}$.

   \end{proof}

\begin{remark} \label{cut-off} Let $\beta_i: \R \to [0, 1] $ be a cut-off function such that
 \begin{enumerate}
\item  $\beta_i(x) =1$ for $|x| <\eps_i/3$,
 \item  $\beta_i(x) =0$ for
 $|x|>2\eps_i/3$.
 \end{enumerate}
 Then by  Proposition \ref{smooth:ext}  and  Lemma \ref{smooth:ontube}, the composition of the norm function $f_{u_i}$ on $\cN^{\eps_i}_{u_i}$
 and $\beta_i$ defines a  smooth  function on $\cU$
  in the sense of Definition \ref{smooth}. The resulting cut-off function will also be denoted by
  \[
  \beta_i:  \cU= \bigcup_{j=1}^n \bU_{u_j}^{\eps_j} \longrightarrow [0, 1]
  \]
  with its support $supp(\beta_i) \subset  \bU_{u_i}^{2\eps_i/3}$, and $\beta_i|_{\bU_{u_i}^{\eps_i/3}}=1$. The corresponding $G$-invariant smooth  function on
  $\pi^{-1}_\cB(\cU)$ has its support contained in $\cT_{u_i}^{2\eps_i/3}$.
 \end{remark}

 In the presence of non-trivial isotropy groups, then  $\cU=\bigcup_{i=1}^n \bU_{u_i}^{\eps_i}$ is equipped with a
 system of  orbifold Banach charts
 \[
 \{ (\cN_{u_i}^{\eps_i}, G_{u_i}, \phi_{u_i}, \bU^{\eps_i}_{u_i})| i= 1, 2,  \cdots, n\}.
 \]
 As in the  trivial isotropy case,
 the system of  Banach  orbifold  charts
\[
 \{ (\cN_{u_i}^{\eps_i}, G_{u_i}, \phi_{u_i}, \bU^{\eps_i}_{u_i})| i= 1, 2,  \cdots, n\}.
 \]
 only provides a topological orbifold structure.  We remark that
the above constructions   still hold  after  some modifications to take care of the  non-trivial isotropy group action.

 \begin{definition}\label{smooth:orbifold}  Given an  open set $\cV\subset \cU$,
 a function   $f: \cV \to \R$ is called smooth with respect to the
 system of  orbifold Banach charts
 $
 \{ (\cN_{u_i}^{\eps_i}, G_{u_i}, \phi_{u_i}, \bU^{\eps_i}_{u_i})\}_i
$
 if for each $i$, the orbifold coordinate function
\[
f \circ  \phi_{u_i}: \qquad   \phi_{u_i}^{-1} (\cV\cap \bU_{u_i}^{\eps_i}) \longrightarrow \cV  \longrightarrow \R
\]
is a  $G_{u_i}$-invariant smooth function.
  \end{definition}

The statements in Lemma \ref{smooth:ontube}, Proposition \ref{norm:function}  still hold without any modification for orbifold cases. The  proofs are almost identical except the proof of Lemma \ref{smooth:ontube} which  needs to be changed slightly to suit  Definition \ref{smooth:orbifold}. To get similar results in   Proposition \ref{smooth:ext}  and Remark
 \ref{cut-off}, we only need to make the norm function $f_{u_i}: \cN_{u_i}\to \R$   $G_{u_i}$-invariant. This can be easily achieved by averaging over $G_{u_i}$ as follows
 \ba\label{norm_ave}
   f_{u_i}(v) =\sum_{g\in G_{u_i}} \big(   \int_{\Sigma} |v(g(x))|_{h(u_i(x))}^p dvol_\Sigma (x)  +
\int_\Sigma  |\nabla v(g(x)) |_{h(u_i(x))}^p dvol_\Sigma (x) \big).
\na

   \begin{proposition} \label{smooth:ext:orbi}
Let $\tilde{f}_{u_i}:  \cN^{\eps_i}_{\cO_{u_i}}\to \R$ be the $G$-invariant extension of the $G_{u_i}$-invariant  norm function $f_{u_i}$ on $\cN^{\eps_i}_{u_i}$. Then
   \[
   \tilde f_{u_i}\circ \exp^{-1}:  \cT_{u_i}^{\eps_i} \to \R
   \]
    is a smooth  $G$-invariant function.
   \end{proposition}

\begin{remark} \label{orbi:cut-off}
We have the   corresponding cut-off functions   similarly constructed  as in Remark \ref{cut-off}  when  there are non-trivial isotropy groups. These constructions provide     a system of   cut-off  functions
\[
\{\beta_i:  \bU_{u_i}^{\eps_i}  \to [0, 1]  | i= 1, \cdots, n \}
\]
 with respect to the above
system of Banach orbifold charts such that, for $\tilde\beta_i = \beta\circ \pi_\cB: \cT_{u_i}^{\eps_i}  \to [0, 1]$,  we have
\begin{enumerate}
\item[(i)] $\tilde \beta_i|_{ \cT_{u_i}^{\eps_i/3}}\equiv 1$,
\item[(ii)]  $\supp (\tilde \beta_i)\subset \cT_{u_i}^{2\eps_i/3}$,
 \item[(iii)]   $\tilde \beta_i|_{ \cS_{u_i}^{\eps_i}}$ is $G_{u_i}$-invariant.
  \item[(iv)]   $\tilde \beta_i|_{ \cT_{u_i}^{\eps_i}}$ is $G$-invariant.
\end{enumerate}
\end{remark}

\subsection{Local obstruction bundles}\label{4.3}

Recall that $\wtE$ is  the infinite dimensional Banach bundle over $\wtB$ whose fiber at $u$ is $L^p (\Sigma, \Lambda^{0,1} T^*\Sigma\otimes_\C u^*TX)$, the space of
$L^p$-sections  of the bundle $ \Lambda^{0,1} T^*\Sigma\otimes_\C u^*TX$.  The Cauchy-Riemann operator defines a  smooth   section $\dbar_J:  \wtB \to \wtE$.  The  zero set of this section
\ba\label{kernel}
\wt\cM =\dbar_J^{-1}(0) \subset \wtB
\na
is our moduli space of parametrized pseudo-holomorphic spheres in $(X,\omega, J)$ with homology class
$A$.
Life would be much simpler  and  dull if $\dbar_J$ were  transversal to the zero section, that would mean,   at any point $u_0 \in \wt\cM$, the linearization of the Cauchy-Riemann operator defines a surjective
linear operator
\ba\label{D:u}
D\dbar_J (u_0): T_{u_0} \wtB = W^{1, p}(\Sigma, u_0^*TX) \longrightarrow  \wtE_{u_0} = L^p(\Sigma,  \Lambda^{0,1} T^*\Sigma\otimes_\C u^*TX).
\na

However, $D\dbar_J (u_0)$ is not surjective in general. As $D\dbar_J (u_0)$ is a first order elliptic differential operator on $\Sigma$, the standard elliptic regularity for elliptic differential operators implies that
$D\dbar_J (u_0)$ is a Fredholm operator, and the kernel of $D\dbar_J (u_0)$ and the kernel of
its adjoint  operator $(D\dbar_J (u_0))^*$ are  finite dimensional and consist of smooth sections of $u_0^*TX$ and
$ \Lambda^{0,1} T^*\Sigma\otimes_\C u_0^*TX$ respectively. Let
\[
E_{u_0} = Ker (D\dbar_J (u_0))^*  \subset   \wtE_{u_0}
\]
  be the kernel of $(D\dbar_J (u_0))^*$. Clearly, the linear operator
  \ba\label{surj:u_0}
  D\dbar_J (u_0) + Id:  T_{u_0}\wtB \oplus E_{u_0} \longrightarrow   \wtE_{u_0}
  \na
  is surjective.  Hence, the space $E_{u_0}$ is called the obstruction space at $u_0$, which is
  $G_{u_i}$-invariant.

   The main result of this Section is to extend this obstruction space $E_{u_0}$ to a smooth  complex
   vector bundle over
  the tubular neighbourhood $\cT_{u_0}^{\eps_0}$ provided by Theorem  \ref{slice-theorem:orbifold}.  This
  obstruction bundle defines a smooth complex
   vector bundle over $\bU_{u_0}^{\eps_0}$ such that under the pull-back of the non-smooth coordinate changes it is still smooth.

  For simplicity, we start  with  the trivial  isotropy group case. The constructions can be easily adapted to the general case  of  non-trivial isotropy groups.

 With the assumption of the isotropy group $G_{u_0}$ being trivial, we know
 that $\wtE|_{\cO_{u_0}}$ is a smooth Banach bundle over the $G$-orbit $\cO_{u_0}$, which is
 $G$-equivariant  under the topological action  of $G$  on $\wtE|_{\cO_{u_0}}$.  In particular, as topological vector
 bundles, we have
 \[
 \wtE|_{\cO_{u_0}} = \bigcup_{g\in G} g\cdot \wtE_{u_0} = \Phi (G\times  \wtE_{u_0})
  \]
 where  the map  $\Phi:  G\times \wtE|_{\cO_{u_0}}  \to \wtE|_{\cO_{u_0}}$ is defined by  the action of $G$.

 \begin{lemma}   \label{obs:orbit}  Define $E_{\cO_{u_0}} = \Phi  (G \times  E_{u_0}  ) $ as a topological
 sub-bunlde of $\wtE|_{\cO_{u_0}}$.
 The bundle $E_{\cO_{u_0}} $ is a smooth sub-bundle of  $\wtE|_{\cO_{u_0}}$, which is
also  $G$-equivariant.
 \end{lemma}
  \begin{proof}
   Note that the action of
 $G$ on $\cO_{u_0}$ is smooth and $E_{u_0} = ker ((D\dbar_J (u_0))^*)$ is finite dimensional and consists of smooth sections of $ \Lambda^{0,1}\Sigma\otimes_\C u_0^*TX$. By the  direct calculation  of the differentials of the actions in the following diagram,
 \[
 \xymatrix{
 G\times E_{\cO_{u_0}} \ar[r] \ar[d] & E_{\cO_{u_0}}\ar[d]  \\
  G\times  \cO_{u_0}  \ar[r]   & {\cO_{u_0}}
 }
 \]
 we know that $E_{\cO_{u_0}}$ is a $G$-equivariant smooth
  vector bundle over  $\cO_{u_0}$, in particularly, it is  a smooth sub-bundle of  $\wtE|_{\cO_{u_0}}$.
  \end{proof}

  Let $\cN^{\eps_0}_{\cO_{u_0}}$ be the  open $\eps_0$-ball bundle of $\cN_{\cO_{u_0}}$ with respect to this  $G$-invariant $W^{1, p}$-norm as in the proof of Theorem \ref{slice-theorem}. Let $\pi : \cN_{\cO_{u_0}}^{\eps_0} \to \cO_{u_0}$  be the projection. Then the pull-back bundle $\pi ^*  E_{\cO_{u_0}}$ defines a smooth bundle over
 $\cN^{\eps_0}_{\cO_{u_0}}$.  Applying the inverse map to the diffeomorphism
$\exp: \cN^{\eps_0}_{\cO_{u_0}} \to \cT_{u_i}^{\eps_0}$, we get a smooth bundle $\widehat E_{\cT_{u_0}}$  over $\cT_{u_0}^{\eps_0}$.
 As $\cN_{\cO_{u_0}}^{\eps_0}$  and  $\cT_{u_0}^{\eps_0}$ only admit  a topological $G$-action, the resulting bundles
 $\pi ^*  E_{\cO_{u_0}} \to \cN_{\cO_{u_0}}^{\eps_0} $ and  $\widehat  E_{\cT_{u_0}} \to \cT_{u_0}^{\eps_0}$ are $G$-equivariant as topological vector bundles.

   \begin{lemma}\label{obs:id}      There is a canonical bundle  pull-back diagram
   \ba\label{pull-back}
\xymatrix{
\pi^* (\wtE|_{\cO_{u_0}})  \ar[r]^{ \text{Exp}}   \ar[d] & \wtE|_{\cT_{u_0}^{\eps_0}} \ar[d]  \\
\cN_{\cO_{u_0}}^{\eps_0}  \ar[r]^{\exp}_\cong    & \cT^{\eps_0}_{ u_0}
 }
   \na
   such  that $\widehat E_{\cT_{u_0}}$ is  canonically embedded into $\wtE|_{\cT_{u_0}^{\eps_0}}$ as a smooth sub-bundle
 over $ \cT^{\eps_0}_{ u_0}$.
   \end{lemma}
 \begin{proof}
 Note that $\pi^*  E_{\cO_{u_0}}$ is a smooth sub-bundle of $ \pi ^* (\wtE|_{\cO_{u_0}} )$ by Lemma
\ref{obs:orbit}. We only need to  define a canonical bundle map $\text{Exp}: \pi^* (\wtE|_{\cO_{u_0}})  \to  \wtE|_{\cT_{u_0}^{\eps_0}}$  and  show that the diagram (\ref{pull-back}) is a pull-back
diagram.

Given
$(u, w, \eta)$ in the fiber of $\pi^* (\wtE|_{\cO_{u_0}})$ over a point $(u, w)\in \cN_{\cO_{u_0}}^{\eps_0}$,   $w\in \cN_{u_0}^{\eps_0}$, then we have  $u=g\cdot u_0 \in \cO_{u_0}$,
$w\in \cN_{u}^{\eps_0} \subset W^{1, p}(\Sigma, u^*TX)$ with the $G$-invariant $W^{1, p}$-norm $\|w\|' < \eps_0$ and $\eta\in L^p(\Sigma,  \Lambda^{0,1} T^*\Sigma\otimes_\C u^*TX)$. In particular, for any
$x\in \Sigma$, $w(x) \in T_{u(x)}X$ and $\eta(x)$ is a $T_{u(x)}X$-valued $(0, 1)$-form at $x$. As
\[
 \exp_{u}(w) =   \exp_{g\cdot u_0}(w) =  g\cdot \exp_{u_0} (g^{-1}\cdot w)   \in g\cdot \cS_{u_0}^{\eps_0} \subset \cT^{\eps_0}_{ u_0},
 \]
 we can apply the parallel transport  map $\tau \big(u(x), \exp_{u(x)}(w(x))\big):  T_{u(x)} X \to T_{\exp_{u(x)} (w(x))}X$
 to $\eta (x) $  along the geodesic
 $\exp_{u(x)} (tw(x))$ for $t\in [0, 1]$ to get an element in $\wtE_{\exp_u(w)}$.  Here the parallel transport  map  is defined  with respect to the Levi-Civita connection for the Riemannian metric on $X$.   Define
 \[
 \text{Exp}:  \pi^* (\wtE|_{\cO_{u_0}})  \to  \wtE|_{\cT_{u_0}^{\eps_0}}
 \]
by $ \text{Exp}(u, w, \eta) (x)  = \tau \big(u(x), \exp_{u(x)}(w(x))\big) (\eta(x)).$
 It is easy to see that $\text{Exp}$ is an isomorphism along the fiber, and $\pi^* (\wtE|_{\cO_{u_0}})
 \cong \exp^*( \wtE|_{\cT_{u_0}^{\eps_0}})$. Therefore, the diagram \ref{pull-back} is a pull-back diagram of vector bundles.
 \end{proof}

  With these two lemmas  understood, we can define local obstructions for  each coordinate chart  in Theorem \ref{slice-theorem}.

   \begin{definition} \label{obs:local}   The local obstruction  bundle $\widehat E_{\cT_{u_0}}$ over
  $  \cT_{u_0}^{\eps_0}$ (the tubular neighbourhood of $\cO_{u_0}$)  is defined to be $ \widehat{E}_{\cT_{u_0}} = (\exp^{-1})^* \pi^*  E_{\cO_{u_0}}$ where $\eps_0$ is chosen such that under the identification of   Lemma  \ref{obs:id}, the linear operator
  \[
  D\dbar_J(u) + Id:  T_u\wtB \oplus \widehat{E}_u \longrightarrow \wtE_u
  \]
  is surjective for any $u\in  \cT_{u_0}^{\eps_0}$.
  The local obstruction bundle $\bE_{u_0}$   over $ \bU_{u_0}^{\eps_0}$ is defined in terms  of its  Banach coordinate chart
  $( \cN_{u_0}^{\eps_0}, \phi_{u_0})$ by the restriction of  $\pi ^* E_{\cO_{u_0}}$ to $ \cN_{u_0}^{\eps_0}$.
   \end{definition}

  \begin{remark}\label{obs:remark}
   \begin{enumerate}
\item
 From the definition, the local  obstruction bundle over $\bU_{u_0}^{\eps_0}$ is trivial in the sense that
  \[
(  \pi ^* E_{\cO_{u_0}} ) |_{ \cN_{u_0}^{\eps_0}} \cong  \cN_{u_0}^{\eps_0}  \times  E_{u_0}.
  \]
 The  local obstruction  bundle over $  \cT_{u_0}^{\eps_0}$ is not trivial  as a smooth sub-bundle of  $\wtE|_{\cT_{u_0}^{\eps_0}}$, though it is   so as a topological vector  bundle,  due to the  non-differential action of $G$ on $\cT_{u_0}$.
 \item These  local obstruction bundles are smooth bundles in the sense that given a local obstruction bundle
 $\bE_{u_0}$  over $\bU_{u_0}^{\eps_0}$ and another Banach coordinate chart
 $( \cN_{u_1}^{\eps_1}, \phi_{u_1})$ for $ \bU_{u_1}^{\eps_1}$ with
 \[
 \bU_{u_0}^{\eps_0} \cap \bU_{u_1}^{\eps_1} \neq \emptyset,
 \]
 even though the coordinate change
 \[
 \phi_{u_1, u_0}: \phi_1^{-1} ( \bU_{u_0}^{\eps_0} \cap \bU_{u_1}^{\eps_1} ) \longrightarrow \phi_{u_0} ( \bU_{u_0}^{\eps_0} \cap \bU_{u_1}^{\eps_1} )
 \]
 is only a homeomorphism, the pull-back bundle $ \phi_{u_1, u_0}^*\bE_{u_0}$ is a smooth vector bundle. This is because
 under the exponential map $\exp_{u_1}:   \cN_{u_1}^{\eps_1} \to \cS_{u_1}^{\eps_1}$,
$ \phi_{u_1, u_0}^*\bE_{u_0}$ is given by the restriction of the local  obstruction  bundle $\widehat E_{\cT_{u_0}}$ over   $  \cT_{u_0}^{\eps_0}$ to the slice $\cS_{u_1}^{\eps_1}\cap  \cT_{u_0}^{\eps_0}$.
\end{enumerate}
  \end{remark}

Now we discuss  the general cases with non-trivial isotropy groups,  where we have a collection of  orbifold
Banach charts
\[
(\cN_{u_i}^{\eps_i}, G_{u_i},  \phi_{u_i},  \bU_{u_i}^{\eps_i} ), \qquad  i=1, \cdots, n,
\]
as given by Theorem \ref{slice-theorem:orbifold},
together with  $E_{u_i} = Ker  (D\dbar_J (u_i))^* \subset \wtE_{u_i}$.  Then each $E_{u_i}$ is a $G_{u_i}$-vector space, hence the local bundle  defined by
\ba\label{obs:local-orbi}
\bE_{u_i} = (\cN_{u_i}^{\eps_i} \times E_{u_i}, G_{u_i},  \phi_{u_i},  \bU_{u_i}^{\eps_i} )
\na  is an orbifold vector bundle over $\bU_{u_i}^{\eps_i}$. Using the topological action of
$G$ on $\wtE|_{\cO_{u_i}}$, we can define $E_{\cO_{u_i}}$ as in Lemma \ref{obs:orbit}. The same proof there implies that $E_{\cO_{u_i}}$ is a smooth $G$-equivariant vector bundle over $\cO_{u_i}$.   Proceeding as in Lemma \ref{obs:id} and Definition \ref{obs:local}, we get a local obstruction bundle $\widehat{E}_{\cT_{u_i}}$ over the tubular neighbourhood $\cT_{u_i}^{\eps_i}$ such that $\widehat{E}_{\cT_{u_i}}$ is a smooth sub-bundle of
$\wtE|_{\cT_{u_i}^{\eps_i}}$. The restriction  of  $\widehat{E}_{\cT_{u_i}}$  to  the  slice
\[
\cS_{u_i}^{\eps_i} = \exp_{u_i} (\cN_{u_i}^{\eps_i})
\]
is isomorphic to $\bE_{u_i}$ under
the exponential map given by (\ref{pull-back}). We often use the same notation $\bE_{u_i}$ to denote the
local obstruction bundle over the slice $\cS_{u_i}^{\eps_i}$. This local obstruction bundle $\bE_{u_i}$ is a smooth $G_{u_i}$-equivariant vector bundle.

\begin{remark} We remark that  $ (u_i, \eps_i), i=1, \cdots, n  $ in (\ref{obs:local-orbi})   are chosen with
  the following two  constraints:
\begin{enumerate}
\item Each $\eps_i$ is chosen such that  the linear operator
  \[
  D\dbar_J(u) + Id:  T_u\wtB \oplus \widehat{E}_u \longrightarrow \wtE_u
  \]
  is surjective for any $u\in  \cT_{u_i}^{\eps_i}$.   Here $ \widehat{E}_u$ is the fiber of $\widehat E_{\cT_{u_i}}$ at $u$.
  \item $\cM \subset \bigcup_{i=1}^n   \bU_{u_i}^{\eps_i/3}$.
\end{enumerate}
\end{remark}

In summary, we have the following system of local obstruction bundles for the moduli space
$\cM$.

\begin{theorem} \label{obs:coherent}  Assume that the moduli space $\cM = \cM_{0, 0}(X, J, A)$ is compact, there exists a system of Banach orbifold charts
\[
\{(\cN_{u_i}^{\eps_i}, G_{u_i},  \phi_{u_i},  \bU_{u_i}^{\eps_i} )|   i=1, \cdots, n
\}
\]
provided by Theorem \ref{slice-theorem:orbifold} with the following  system of local obstruction bundles.
\begin{enumerate}
\item For each $i$, $ E_{u_i} =  Ker (D\dbar_J(u_i))^* $ defines an orbifold vector bundle $\bE_{u_i}$  over
$  \bU_{u_i}^{\eps_i}$ which is represented by a $G_{u_i}$-equivariant vector bundle over the slice $\cS_{u_i}^{\eps_i}$. The   bundle $\pi_{u_i}:  \bE_{u_i} \to \cS_{u_i}^{\eps_i}$   is   called the {\bf  local obstruction bundle}.
\item Each local obstruction bundle $\pi_{u_i}: \bE_{u_i} \to \cS_{u_i}^{\eps_i}$ extends to a smooth bundle  over $\cT_{u_i}^{\eps_i}$ (the tubular neighbourhood of the orbit $\cO_{u_i}$). The resulting bundle $\widehat E_{\cT_{u_i}}$
 is a smooth sub-bundle of  $\wtE|_{\cT_{u_i}^{\eps_i}}$, which is also  topologically $G$-equivariant.
\item For any  $u\in \cT_{u_i}^{\eps_i}$, the linear operator
  \[
  D\dbar_J(u) + Id:  T_u\wtB \oplus \widehat{E}_u \longrightarrow \wtE_u
  \]
  is surjective. Here $ \widehat{E}_u$ is the fiber of $\widehat E_{\cT_{u_i}}$  at $u$.
\end{enumerate}
\end{theorem}

   \subsection{Invariant  local virtual    neighbourhoods}\label{4.4}

   Having prepared the analytical foundation for the moduli spaces of pseudo-holomoprhic spheres, we demonstrate
   in this subsection how to get a system of
    $G$-invariant   local virtual neighbourhoods for  $\wtM$ whose slices provide  a system of  smooth  virtual orbifold neighbourhoods
   for $\cM=\wtM/G$. In order to patch these  virtual orbifold neighbourhoods together to get a finite dimensional virtual orbifold system as
   in Definition \ref{virtual-sys:orb} so that the virtual integration is well-defined, we need to resort to the full machinery of the virtual neighborhood technique:
   local and global stabilizations. This will be accomplished in Section \ref{5}.

Denote by $\wt\cU^{\eps_i}_{u_i}$ the total space of the local obstruction  bundle  $\hat \pi_{u_i}:  \widehat E_{\cT^{\eps_i}_{u_i}} \to \cT_{u_i}^{\eps_i}$ in Theorem \ref{obs:coherent}.  Then  $\wt\cU^{\eps_i}_{u_i}$ is a smooth Banach manifold with an induced  topological $G$-action.   The projection map $\wt\cU^{\eps_i}_{u_i} \to  \cT_{u_i}^{\eps_i}$  is still denoted by $\hat \pi_{u_i}$. The pull-back bundle
$
 \hat \pi_{u_i}^*\big( \wtE|_{\cT^{\eps_i}_{u_i}} \big)  \longrightarrow \wt\cU_{u_i},
$
has    a  section
 \ba\label{Xi_ui}
 \xymatrix{  \hat\pi_{u_i}^*\wtE|_{\cT_{u_i}^{\eps_i} }  \ar[d]    \\
   \wt\cU_{u_i}   \ar@/_1pc/[u]_{\Xi_{u_i}} }
     \na
     defined by
\[
\eta \mapsto \dbar_J \big(  \hat \pi_{u_i} (\eta)\big)  + \eta
\]
for  any $\eta \in \wt\cU^{\eps_i}_{u_i}$, which is identified with an element in  the fiber of   $\wtE|_{\cT^{\eps_i}_{u_i}}$ at
$\hat \pi_{u_i} (\eta)$. Here  we use the fact that $ \widehat E_{\cT_{u_i}} $ is a smooth sub-bundle of
 $\wtE|_{\cT^{\eps_i}_{u_i}}$.

  \begin{proposition}  \label{G-virtual:neigh} Let  $(\wt\cU^{\eps_i}_{u_i}, \hat \pi_{u_i}^*\big( \wtE|_{\cT^{\eps_i}_{u_i}} \big), \Xi_{u_i})$ be  the triple constructed as above.  Denote by
  \[
  \wt\cV_{u_i} = \Xi_{u_i}^{-1}(0) \cap  \wt\cU^{\eps_i/3}_{u_i}
  \]
     the  zero set of   the section  $\Xi_{u_i}$. Then
  we have a  $G$-invariant virtual neighbourhood
  \[
  (\wt\cV_{u_i},  \wt E_{u_i},\wt\sigma_{u_i} )
  \]
  of $\wtM \cap  \cT_{u_i}^{\eps_i}$  in the following sense.
  \begin{enumerate}
\item    $\wt\cV_{u_i}$ is  a  finite dimensional  smooth  $G$-manifold.
\item  $\wt E_{u_i}$ is a   $G$-equivariant bundle given by the restriction of the pull-back bundle $\hat \pi_{u_i}^*\big(  \widehat E_{\cT^{\eps_i}_{u_i}}  \big)$.
\item   $\wt\sigma_{u_i}$ is a canonical $G$-invaraint section of  $\wt E_{u_i} $ whose zero set
$  \wt\sigma_{u_i}^{-1}(0)$ is $ \wtM \cap  \cT_{u_i}^{\eps_i}$.
\end{enumerate}
\end{proposition}

  \begin{proof}  By our construction,  $\wt\cU^{\eps_i}_{u_i}$ is an infinite dimensional Banach manifold whose tangent space at $\eta$ is given by
  \[
  T_\eta \wt\cU^{\eps_i}_{u_i} \cong T_{\hat \pi_{u_i} (\eta)}  \wtB \oplus  \widehat{E}_{\hat \pi_{u_i} (\eta)},
  \]
  and $ \hat \pi_{u_i}^*\big( \wtE|_{\cT^{\eps_i}_{u_i}} \big)$ is a Banach bundle  over  $\wt\cU^{\eps_i}_{u_i}$ with a Fredholm section $\Xi_{u_i}$.  From  the property of local obstruction bundles in Theorem \ref{obs:coherent}, we know that the section $\Xi_{u_i}$ is transversal to the zero section.  Though  the bundle $\hat \pi_{u_i}^*\big( \wtE|_{\cT_{u_i}} \big)  \to \wt\cU_{u_i}$ is  smooth, it is  $G$-equivalent only as  a topological bundle. So a priori, the zero set $ \wt\cV_{u_i}$ is  only a topological $G$-space.
  Being a solution to an elliptic differential equation, each point in  $ \wt\cV_{u_i}$ consists of
  a smooth map $u:  \Sigma \to X$ and an element $\eta_u$ is a smooth section of the  bundle $ \Lambda^{0,1} T^*\Sigma\otimes_\C u^*TX$, therefore  the  $G$-action on $ \wt\cV_{u_i}$  is smooth.   This implies that the zero
  set
  \[
  \wt\cV_{u_i} = \Xi_{u_i}^{-1}(0) \subset \wt\cU_{u_i}
  \]
  is a finite  dimensional  smooth manifold.

  By the same argument, the restriction of $\hat \pi_{u_i}^*\big(  \widehat E_{\cT^{\eps_i}_{u_i}}  \big)$ to $\wt\cU_{u_i}$ is a smooth $G$-equivariant vector bundle.   The bundle $\wt E_{u_i} \to \wt\cV_{u_i} $ has a canonical $G$-invariant section $\wt \sigma_{u_i}$ given by
  \[
  \eta = (u=   \hat \pi_{u_i} (\eta), \eta_u)\mapsto \eta_u
  \]
  for $u: \Sigma \to X$  and $\eta_u  \in \wt E_{u_i} |_{\eta} =  \widehat E_{\cT_{u_i}} |_u \subset \C^\infty (\Sigma,  \Lambda^{0,1} T^*\Sigma\otimes_\C u^*TX)$.  A zero point $(u, \eta_u)$  in  $   \wt\sigma_{u_i}^{-1}(0)$ is  a solution to
  the equations
  \[
 \xi_{u_i} (u, \eta_u) =  \dbar_J \big(  u )  + \eta_u =0, \qquad   \wt \sigma_{u_i}(u, \eta_u) =  \eta_u =0.
 \]
 for $u\in \cT_{u_i}^{\eps_i}$, which implies $\dbar_J \big(  u )=0$.  This gives a point in $\wtM \cap  \cT_{u_i}^{\eps_i}$. This correspondence defines a homeomorphism
 \[
 \wt\psi_{u_i}:    \wt\sigma_{u_i}^{-1}(0)  \to \wtM \cap  \cT_{u_i}^{\eps_i}
 \]
 with the $G$-equivariant condition automatically satisfied from the construction.  This completes the proof.
  \end{proof}

Even though the action of $G$ on $ \wt\cU^{\eps_i}_{u_i}$ is not differentiable, it does admit
 a slice  $\hat\pi_{u_i}^{-1} (\cS_{u_i}^{\eps_i})$ inherited from the $G$-slice $\cS_{u_i}^{\eps_i}$.  In fact, we have  the following proposition about the $G$-slice for the action on the virtual neighbourhood
$
  (\wt\cV_{u_i},  \wt E_{u_i},\wt\sigma_{u_i})
$
 in Proposition \ref{G-virtual:neigh}.

     \begin{proposition}  \label{virtual:neigh}
     The smooth $G$-action on the  the
 virtual neighbourhood
$
  (\wt\cV_{u_i},  \wt E_{u_i},\wt\sigma_{u_i} )
$
 in Proposition \ref{G-virtual:neigh} admits a slice which provides a $G_{u_i}$-invariant virtual  neighbourhood
 \[
 (\cV_{u_i},    E_{u_i}, \sigma_{u_i} )
 \]
  for $\wtM\cap  \cS_{u_i}^{\eps_i}$   in the following  sense.
  \begin{enumerate}
\item   The $G$-slice  $ \cV_{u_i}$ is  a  finite dimensional  smooth  $G_{u_i}$-manifold.
\item    $E_{u_i}$ is a   $G_{u_i}$-equivariant  vector bundle given by the restriction of $\wt E_{u_i}$.
\item   $\sigma_{u_i}$ is a canonical $G_{u_i}$-invaraint section of $  E_{u_i}$ whose zero set
   $ \sigma_{u_i}^{-1}(0)$ is  $ \wtM \cap  \cS_{u_i}^{\eps_i}$.
  \end{enumerate}
\end{proposition}

  \begin{proof}   From Theorem  \ref{slice-theorem:orbifold}, we know that the slice of  the $G$-action at $u_i$ is $\cS_{u_i}^{\eps_i}$ given by the exponential map $\exp: \cN_{u_i}^{\eps_i}\to \cS_{u_i}^{\eps_i}$ where
 $(\cN_{u_i}^{\eps_i}, G_{u_i},  \phi_{u_i},  \bU_{u_i}^{\eps_i})$ is a Banach orbifold chart for $\bU_{u_i}^{\eps_i}$.
 Let $\wt\cS_{u_i}^{\eps_i}$ be the  total space of the local obstruction bundle $\bE_{u_i} =  \widehat E_{\cT_{u_i}}|_{\cS_{u_i}^{\eps_i}}$. Then we have the corresponding restrictions of the bundle $  \hat \pi_{u_i}^*\big( \wtE|_{\cT_{u_i}} \big)$ and  its section
  $ \Xi_{u_i}$. Denote the resulting triple by
  $(\wt\cS_{u_i}^{\eps_i}, \wt\bE_{u_i},  \Xi_{u_i})$ which is $G_{u_i}$-equivariant in the sense that
  $\wt\cS_{u_i}^{\eps_i}$ is a Banach $G_{u_i}$-manifold, and $\wt\bE_{u_i}$ is a $G_{u_i}$-equivariant vector bundle,
  and $ \Xi_{u_i}$ is a $G_{u_i}$-invariant  Fredholm section  which is transversal to the zero section. Then by the same procedure as in
  Proposition  \ref{G-virtual:neigh},  we get \begin{enumerate}
\item  The $G$-slice $ \cV_{u_i} = \wt\cV_{u_i} \cap \wt\cS_{u_i}^{\eps_i} $ is  finite dimensional  smooth  $G_{u_i}$-manifold.
\item    $E_{u_i} = \wt E_{u_i}|_{ \cV_{u_i}}$   is a   $G_{u_i}$-equivariant  vector bundle.
\item $\sigma_{u_i}=   \wt\sigma_{u_i}|_{ \cV_{u_i}}$ is a canonical $G_{u_i}$-invariant section $\sigma_{u_i}$.
\item The restriction of $ \wt\psi_{u_i}$ to $  \wt\sigma_{u_i}^{-1}(0)\cap  \cV_{u_i}$ is the homeomorphism
 $\psi_{u_i}:     \sigma_{u_i}^{-1}(0)  \to \wtM \cap  \cS_{u_i}^{\eps_i}$.
\end{enumerate}
     \end{proof}

  \begin{remark}  It is clear that the $G_{u_i}$ -invariant virtual neighbourhood  $ (\cV_{u_i},    E_{u_i}, \sigma_{u_i} )$  in Proposition    \ref{virtual:neigh}  provides  a Kuranishi neighbourhood for  $\cM\cap \bU_{u_i}^{\eps_i}$ as defined in \cite{FO99}. Our  construction here comes from a $G$-invariant Kuranishi neighbourhood  $ (\wt\cV_{u_i},  \wt E_{u_i},\wt\sigma_{u_i} )$ for
  $\wtM\cap \cT_{u_i}^{\eps_i}$.
  Note that  these triples
  \[
  \{(\cV_{u_i},    E_{u_i}, \sigma_{u_i})  | i= 1, 2, \cdots, n\}
  \]
  for a system of
Banach orbifold charts
\[
\{(\cN_{u_i}^{\eps_i}, G_{u_i},  \phi_{u_i},  \bU_{u_i}^{\eps_i} )|   i=1,2,  \cdots, n
\}
\]
are  not compatible as a  Kuranishi  structure for $\cM$.  In the next Section, we will  apply  the  virtual neighborhood technique (local and global stabilizations) developed in Section \ref{3} in a $G$-equivariant way to get
   the virtual orbifold  system in Definition \ref{virtual-sys:orb} for $(\cB, \cE, \dbar)$.
\end{remark}

\section{Application of the virtual neighborhood technique  and Genus zero Gromov-Witten invariants} \label{5}

 In this section, we apply the virtual neighborhood technique developed in Section  \ref{3}
  to get a   $G$-invariant  virtual   system for   the Fredholm system  $(\wtB, \wtE, \dbar_J)$,
 whose slice  forms a  virtual  orbifold   system  for   $(\cB,  \cE,   \dbar_J)$.

\subsection{Local  stabilizations for pseudo-holomorphic spheres}

In this subsection, we will  apply the local stabilization process to the topological $G$-equivaraint orbifold
Fredholm system $(\wtB, \wtE, S)$ to  obtain a collection of thickened orbifold
Fredholm systems.

Recall in Remark \ref{orbi:cut-off}, there exist   smooth cut-off functions
\[
\{\wt \beta_i:  \cT_{u_i}^{\eps_i} \to [0 ,1] |  i= 1,2,  \cdots, n\}
 \]
such that \begin{enumerate}
\item[(i)] $\tilde \beta_i|_{ \cT_{u_i}^{\eps_i/3}}\equiv 1$,
\item[(ii)]  $supp (\tilde \beta_i)\subset \cT_{u_i}^{2\eps_i/3}$,
 \item[(iii)]   $\tilde \beta_i|_{ \cS_{u_i}^{\eps_i}}$ is $G_{u_i}$-invariant.
  \item[(iv)]   $\tilde \beta_i|_{ \cT_{u_i}^{\eps_i}}$ is $G$-invariant.
  \end{enumerate}

  These  cut-off functions  fulfil the requirement in Assumption \ref{assump:orbi}.  Here we treat  the nested  $G_{u_i}$-invariant neighbourhoods
  \[
  \cS_{u_i}^{\eps_i/3}\subset \cS_{u_i}^{2\eps_i/3} \subset \cS_{u_i}^{\eps_i}
  \]
  as nested orbifold charts for $U_{[u_i]}^{\eps_i/3}\subset U_{[u_i]}^{2\eps_i/3}\subset  U_{[u_i]}^{\eps_i}\subset \cB$.
 The collection of triples
  \[
\{  ( \cS_{u_i}^{\eps_i},  G_{u_i},  U_{[u_i]}^{\eps_i})|  i= 1,2,  \cdots, n\}
  \]
 generates  a  proper \'etale groupoid.   Set $\overrightarrow\eps =\{\eps_1, \eps_2, \cdots, \eps_n\}$.
  The space of units consists of the disjoint union
\[
\bU_{\overrightarrow\eps}^0 :=  \bigsqcup_i \cS_{u_i}^{\eps_i},
\]
 the space of arrows consist of  triples
 \ba\label{gen:gpoid}
\bU_{\overrightarrow\eps}^1 :=  \{ (u, g, v) |u, v \in  \bigsqcup_i \cS_{u_i}^{\eps_i}, \text{ and  there exists } g\in G \text{ \ such that\ } g\cdot u =v
 \},
 \na
 This groupoid will be  denoted by
  $  \bU_{\overrightarrow\eps}  $,   Similarly, we have groupoids
  $\bU_{\overrightarrow\eps/3}$ and $\bU_{2\overrightarrow\eps/3}$.
  Note  that due to the non-smooth action of $G$, these groupoids
  \[
  \bU_{\overrightarrow\eps/3} , \qquad  \bU_{2\overrightarrow\eps/3}  \quad \text{and} \quad
  \bU_{\overrightarrow\eps}
  \]
  are only  topological proper \'stale  groupoids and
  are   Morita equivalent to
  the action groupoids
  \[
  \cT_{\overrightarrow\eps/3} , \qquad  \cT_{2\overrightarrow\eps/3}  \quad \text{and} \quad
  \cT_{\overrightarrow\eps}
  \]
  associated the $G$-actions on
  \[
  \bigcup_{i=1}^n   \cT_{u_i}^{\eps_i/3}, \qquad  \bigcup_{i=1}^n   \cT_{u_i}^{3\eps_i/3} \quad \text{and} \quad   \bigcup_{i=1}^n
    \cT_{u_i}^{ \eps_i}
  \]
  respectively.

   \begin{remark} \label{generate} This method of generating groupoids  specifying  the  space  of units for a groupoid,  with the corresponding space of arrows defined as
  (\ref{gen:gpoid})  will be applied repeatedly
 in this section.   With this convention understood, the local and global  stabilizations below will be applied to the
  space of units.
  \end{remark}

  Now we can recast Propositions \ref{G-virtual:neigh} and \ref{virtual:neigh} in terms of the $G$-invariant  local stabilization of
  $(\wtB, \wtE, \dbar_J)$.
       Consider the  action  groupoids  $\cT_{ \eps_i/3} $ and $ \cT_{ \eps_i/3} $   for the $G$-action on $ \cT_{u_i}^{\eps_i/3}$  and   $ \cT_{u_i}^{\eps_i}$.    Then  the construction
    in (\ref{Xi_ui}) where the section $\Xi_{u_i}$ is modified by
    \ba\label{Xi:modify}
    \Xi_{u_i} (\eta)  =   \dbar_J \big(  \hat \pi_{u_i} (\eta)\big)  + \tilde \beta_{u_i} ( \eta)
\na
provides a $G$-invariant local stabilization for $(\wtB, \wtE, \dbar_J)$ in the following sense.
The thickened  Fredholm system
\[
 (\wt\cU^{\eps_i}_{u_i}, \hat \pi_{u_i}^*\big( \wtE|_{\cT_{u_i}} \big), \Xi_{u_i})
 \]
  constructed for
Proposition   \ref{G-virtual:neigh}
with the modified section (\ref{Xi:modify})  is $G$-invaraint under the topological $G$-actions on $\wt\cU^{\eps_i}_{u_i}$  and $\hat \pi_{u_i}^*\big( \wtE|_{\cT^{\eps_i}_{u_i}} \big)$ such that
\begin{enumerate}
\item the action groupoid   generated by the $G$-action on
\[
\wt\cV_{u_i}  =  \Xi_{u_i}^{-1}(0) \cap  \wt\cU^{\eps_i/3}_{u_i}
\]
is a finite dimensional  smooth proper Lie groupoid,
\item the action  groupoid $G\ltimes \wt E_{u_i}$ is a finite rank vector bundle over $G\ltimes \wt\cV_{u_i}$ with a canonial
smooth section.
\end{enumerate}
Moreover, the  triple $( G\ltimes \wt\cV_{u_i}, G\ltimes \wt E_{u_i}, \sigma_{u_i})$ is Morita equivalent to
  the orbifold triple
  \[
  (G_{u_i}\ltimes \cV_{u_i}, G_{u_i}\ltimes E_{u_i}, \sigma_{u_i})
  \]
  constructed in     Proposition    \ref{virtual:neigh}

\subsection{Global stabilizations  and invariants for pseudo-homolomorphic spheres}

    Given a non-empty subset $I  \subset \{1,   2, \cdots, n\}$,   define
 \ba\label{cA_I}
 \cA_I = \big(  (  \bigcap_{i\in I} \bU_{u_i}^{\eps_i} ) - (\bigcup_{j\notin I}  \overline{\bU}_{u_j}^{2\eps_j/3} ) \big) \cap
 \big(\bigcup_{i =1}^n   \bU_{u_i}^{\eps_i/3} \big) \subset \cB,
 \na
 where $ \overline{\bU}_{u_j}^{2\eps_j/3}$ is the closure of  $\bU _{u_j}^{2\eps_j/3}$. Then $\cA_I \subset \bU_{u_i}^{\eps_i}$ for every $i \in I$.

 \begin{lemma}
 The collection $\{\cA_I | I\subset \{1, 2, \cdots, n\}, I \neq \emptyset\}$ forms a cover of $\bigcup_{i =1}^n   \bU_{u_i}^{\eps_i/3}$.
 \end{lemma}
 \begin{proof}  The proof follows from a simple set theoretical argument.
 \end{proof}

 Over each $ \cA_I$, we have
 \ba\label{T_I}
 \cT_I : = \pi_\cB^{-1} ( \cA_I ) =  \big(  (  \bigcap_{i\in I} \cT_{u_i}^{\eps_i} ) - (\bigcup_{j\notin I}  \overline{\cT}_{u_j}^{2\eps_j/3} ) \big) \cap
 \big(\bigcup_{i =1}^n   \cT_{u_i}^{\eps_i/3} \big).
 \na
 So $\cT_I \subset \cT_{u_i}^{\eps_i}$ for every $i \in I$.  For  each $\cT_I$ , we have a smooth Banach bundle $\wtE_I = \wtE|_{\cT_I} $  and a smooth local obstruction bundle
 \ba\label{E_I}
\hat\pi_I:   \widehat E_I =  \bigoplus \widehat E_{\cT_{u_i}}|_{\cT_I } \longrightarrow \cT_I.
 \na
  Both  $\wtE_I$ and  $ \widehat E_I $  are topological $G$-equivariant bundles.

  Denote by  $\wt\cU_I$  the total space   of $ \widehat E_I \to \cT_I$.
    Over  $\wt\cU_I$, we have two pull-back  bundles  $\hat\pi_I^* \widehat E_I $ and $\hat\pi_I^*\wtE_I$. They are smooth Banach bundles over the Banach manifold $\wt\cU_I$.  As constructed earlier in this Section, the bundle   $\hat\pi_I^*\wtE_I \to \wt\cU_I$ has a canonical section
  \ba\label{Xi_I}
 \xymatrix{  \hat\pi_I^*\wtE_I   \ar[d]    \\
   \wt\cU_I   \ar@/_1pc/[u]_{\Xi_I}  }
     \na
  given by
   \ba\label{Xi_I:def}
\eta=   (u,  (\eta_i)_{i\in I}) \mapsto \dbar_J (u) +  \sum_{ i\in I} \wt\beta_i (u) \eta_i.
\na

  \begin{theorem}  \label{global:virtual:neigh}
  Denote by $\wt\cV_I= \Xi_I^{-1}(0)$   the zero set of the section $\Xi_I$. Then
  we have a collection of   $G$-invariant virtual neighbourhoods
  \[
  (\wt\cV_I,  \wt E_I,\wt\sigma_I)
  \]
  of $\wtM \cap  \cT_I$
  such that   \begin{enumerate}
\item    $\wt\cV_I$ is  a  finite dimensional  smooth  $G$-manifold.
\item  $\wt E_I$ is a   $G$-equivariant bundle given by the restriction of the pull-back bundle $\hat \pi_I^*\big(  \widehat E_I  \big)$.
\item  $\wt\sigma_I$ is   a canonical $G$-invariant section of $\wt E_I$  whose
zero  set $ \wt\sigma_I^{-1}(0) $ is $ \wtM \cap  \cT_I$.
\item The zero  sets $\{\wt\sigma_I^{-1}(0) \}_I$ form a cover of $\wtM$.
\end{enumerate}
Moreover, fix $i\in I$,   let $\cS_{I, i} = \cT_I\cap \cS_{u_i}^{\eps_i}$, then  the smooth $G$-action on $ (\wt\cV_I,  \wt E_I,\wt\sigma_I,  \wt\psi_I)$ admits a slice
 \[
 (\cV_{I, i},    E_{I, i}, \sigma_{I, i},  \psi_{I, i})
 \] which   provides a $G_{u_i}$-virtual  neighbourhood
 for $\wtM\cap \cS_{I, i}$ such that
  \begin{enumerate}
\item   The $G$-slice  $ \cV_{I, i}$ is  a  finite dimensional  smooth  $G_{u_i}$-manifold.
\item    $E_{I,  i}$ is a   $G_{u_i}$-equivariant  vector bundle given by the restriction of $\wt E_{I}$.
\item  $\sigma_{I, i}$ is a canonical $G_{u_i}$-invariant section  of  $  E_{I, i}$  whose  zeros  $\sigma_{I, i}^{-1}(0)$  is
$\wtM \cap  \cS_{I, i}$.
  \end{enumerate}
  \end{theorem}

  \begin{proof}  Note that   $\wt\cU_I$  is an infinite dimensional  Banach manifold whose tangent bundle at $ \eta= (u, (\eta_i)_{i\in I})$,  with
 each $\eta_i\in \widehat{E}_u$,  is
 \[
 T_{\eta} \wt\cU_I = T_u \wtB \oplus \bigoplus_{i\in I} \widehat{E}_u,  \]
and $ \wt E_I$ is a Banach bundle with  a Fredholm section $\Xi_I$. Then
this section is transversal to the zero section.  To check this claim, for any point $ \eta= (u, (\eta_i)_{i\in I}) \in \wt\cU_I$, there exists
$i_0\in I$ such that $u\in  \cT_{u_{i_0}}^{\eps_{i_0}/3}$, hence, $\wt\beta_{i_0}  (u)=1$. Let $\eta_i$ be zero except for $i= i_0$,  then we know
that $\Xi_I$ is surjective at $(u, (\eta_i)_{i\in I})$ as  for $u\in   \cT_{u_{i_0}}^{\eps_{i_0}/3}$  the linear operator
\[
D\dbar_I (u) + Id:  T_u\wtB \oplus \widehat{E}_{u_i}|_u \longrightarrow \wtE_u
\]
is surjective.
Then the proofs of  Propositions \ref{G-virtual:neigh} and
  \ref{virtual:neigh} can be carried over   for the  triple $(\wt\cU_I, \hat\pi_I^*\wtE_I, \Xi_I)$ in (\ref{Xi_I}) to obtain the results in the theorem.
    \end{proof}

 Associated to a system of Banach orbifold charts
 \[
\{(\cN_{u_i}^{\eps_i}, G_{u_i},  \phi_{u_i},  \bU_{u_i}^{\eps_i} )|   i=1, \cdots, n
\}
\]
prescribed by Theorem  \ref{slice-theorem:orbifold},  together with local obstruction bundles $\{\bE_{u_i}\}_i$ given by
Theorem \ref{obs:coherent},   we obtain a system of virtual  orbifold  neighbourhoods
   \[
\{ (\cV_{I, i},    E_{I, i}, \sigma_{I, i},  \psi_{I, i})|  I \subset \{1, 2, \cdots, n\}, i\in I\neq \emptyset\}
 \]
 as in Theorem \ref{global:virtual:neigh}, for any $I\subset \{1, 2, \cdots, n\}$ and $i\in I$.

 Let  $\cV_I$  and $\bE_I$ be the
 proper \'etale groupoids  generated by the action of elements in $G$ which preserve
 \[
 \bigsqcup_{i\in I} \cV_{I, i} \qquad \text{and} \qquad  \bigsqcup_{i\in I} E_{I, i}
\]
in the sense of  Remark \ref{generate}.  Then $\bE_I$ is  a vector bundle over $\cV_I$ with a canonical   section  $\sigma_I$.   We remark that as
proper Lie groupoids,
\[
G\ltimes \wt\cV_I, \qquad  G_{u_i} \ltimes \cV_{I, i}, \quad \text{and} \quad \cV_I
\]
are all Morita equivalent, both $\cV_I$ and $G_{u_i} \ltimes \cV_{I, i}$ are proper and \'etale.

 \begin{theorem}  \label{main:thm}
  Let $(\wtB, \wtE, \dbar_J)$ be the Fredholm system  associated to pseudo-holomorphic spheres with the topological action
 of the reparametrization group $G=PSL(2, \C)$. Then the virtual  neighborhood
 technique of Section \ref{3} can be applied in a $G$-equivariant way such that the following two statements hold.
  \begin{enumerate}
\item  The collection of triples $\{(\wt\cV_I, \wt  E_I,\wt\sigma_I)| I\subset  \{1, 2, \cdots, n\}\}$ is a finite dimensional virtual system  in the $G$-equivariant sense of Definition \ref{virtual-system} such the zero  sets $\{\wt\sigma_I^{-1}(0) \}_I$ form a cover of $\wtM$.
\item The collection of triples $\{( \cV_I, \bE_I, \sigma_I)| I\subset  \{1, 2, \cdots, n\}\}$ is a finite dimensional virtual orbifold
 system  in the  sense of Definition \ref{virtual-sys:orb} such the zero  sets $\{\sigma_I^{-1}(0) \}_I$ form a cover of $\cM$.
\end{enumerate}
\end{theorem}
\begin{proof}   With the detailed proofs of Theorems \ref{thm:vir:system} and  \ref{thm:orbi-vir-sys},  it is straightforward to check
that
\[
\{ (\wt\cV_I, \wt  E_I,\wt\sigma_I)| I\subset  \{1, 2, \cdots, n\}\}
\]
 is a finite dimensional virtual system. As $\wt E_I$ is a $G$-equivariant
vector bundle for each $I$, it  remain  to check that the patching datum are $G$-invariant. This
can be easily verified.
The collection   $\{\cV_I\}$ of  proper   \'etale   groupoids  with the patching induced from
the $G$-invariant patching datum for $\{\wt\cV_I\}$ ensures that it is a  proper   \'etale   virtual groupoid.  Similarly, we get a
finite rank virtual vector bundle  $\{\bE_I\}$ over $\{\cV_I\}$.
\end{proof}

  In order to define invariants for the moduli space of pseudo-holomorphic spheres, we need to check
  that  $\{ \cV\}$ and $\{\bE_I\}$ are oriented,  and an analog of  Assumption \ref{assump:orbi-2}  is satisfied for
  $(\wtB, \wtE, \dbar_J)$.    The orientation condition can be checked using the classical argument as in page 32 of
   \cite{McS}.  Assumption  \ref{assump:orbi-2}  follows from the standard elliptic regularity theory for
   the perturbed Cauchy-Riemann equations.

   So we can choose a partition of unity $\{\eta_I\}$ on $\cV= \{\cV_I\}$ and a virtual Euler form $\{\theta_I\}$ of $\bE=\{\bE_I\}$.
   Then given a virtual differential form
   $\alpha=\{\alpha_I\}$ on  $\cV$, we can define the invariant to be
   \ba\label{final:inv}
   \Phi (\alpha) = \int^{vir}_\cV \alpha = \sum_{I} \int_{\cV_I} \eta_I\cdot \theta_I \cdot \alpha_I.
   \na
   One could take  $\alpha$  the constant function  $1$,
   then
   \[
   \Phi (1) =  \sum_{I} \int_{\cV_I} \eta_I\cdot \theta_I
   \]
   which will be zero if the virtual dimension of $\cV$ is not the same as the virtual rank of $\bE$. By Remark
   \ref{rem:ind}, we know that  the invariant (\ref{final:inv}) is indeed well-defined, that is, it does not depend
   on various choices in the  process.

   \begin{remark} We can consider the moduli space of pseudo-holomorphic spheres with  $k$ marked points. Assume that $\cM$ is compact. If $k\geq 3$, then $(\cB, \cE, \dbar_J)$ itself is an orbifold Fredholm system. The virtual neighborhood technique developed in Section \ref{3}
   can be applied directly to get a virtual orbifold system. When $k\leq 2$, and so  the reparametrization group is a  Lie subgroup  $G$ of
   $PSL(2, \C)$,  then the slice and tubular neighbourhood theorem (Theorem \ref{slice-theorem}) still  hold. The virtual neighborhood technique can
   still be applied in principle as in Section \ref{4} and this Section to obtain  a well-defined  virtual orbifold system.  Now we can take many interesting virtual differential forms obtained by the pull-back closed  differential forms on $X$ using the evaluation maps.  We remark
   that by an appropriate implementation of the  virtual neighborhood technique, the evaluation maps can be made into  smooth maps.  We will discuss these issues in details in \cite{CLW3}. 
     \end{remark}

\vskip .2in
\noindent
{\bf Acknowledgments}  This work is  supported by   the Australian Research Council   Grant
(DP1092682) and  the National Natural Science Foundation of China Grant.   Chen likes to express his deep appreciations to Yongbin Ruan and Gang Tian for their
 continuous supports   and encouragements.    Wang likes to thank Alan Carey for his unyielding
 support over many years.  We  thank   Huijun Fan,   Jianxun Hu
 and Kaoru Ono for many interesting
  discussions.  We also  thank Dusa McDuff for her  careful reading and  useful comments  on the manuscript. The material  in this paper was  presented in the workshops  at BICMR in May 2013 and at SCGP in March 2013, we like to thank these institutes for 
  their hospitalities. In particular, we thank all the participants at the SCGP workshop for their pointed questions,  feedbacks and encouragements.

\end{document}